\documentclass[leqno,11pt]{amsart}

\numberwithin{equation}{section}

\usepackage[utf8]{inputenc}
\usepackage[english]{babel}

\usepackage{amsmath, amsfonts, amssymb, esint, amsthm, nicefrac}
\usepackage{booktabs, accents, csquotes}
\usepackage{epsfig, graphicx}

\usepackage{psfrag}
\usepackage{graphpap, latexsym, epsf, color}
\usepackage{amssymb, mathrsfs, enumerate}
\usepackage[colorlinks, citecolor=citegreen, linkcolor=refred]{hyperref}

\usepackage{caption, float, enumitem}

\newcommand{\R}{\mathbb{R}}
\newcommand{\RRR}{\mathbb{R}^{3{\times}3}}
\newcommand{\dist}{{\rm dist}}
\newcommand{\SO}{{SO}}

\newcommand{\sym}{{\rm sym}\,}
\newcommand{\curl}{{\rm curl}\,}
\newcommand{\rmd}{{\rm d}}
\newcommand{\LL}{{L}}
\newcommand{\WW}{{W}}
\newcommand{\trsp}{\scriptscriptstyle{\mathsf{T}}}
\mathchardef\emptyset="001F

\definecolor{vgreen}{rgb}{0.1,0.5,0.2}
\definecolor{viola}{RGB}{85,26,139}
\definecolor{citegreen}{rgb}{0,0.6,0}
\definecolor{refred}{rgb}{0.8,0,0}

\newtheorem{thm}{Theorem}[section]    
\newtheorem{lemma}[thm]{Lemma}

\newtheorem{cor}[thm]{Corollary}

\theoremstyle{definition}
\newtheorem{remarkr}[thm]{Remark}
\newenvironment{remark}
  {\pushQED{\qed}\remarkr}
  {\popQED\endremarkr}

\begin{document}

\title[Dimension reduction with oscillatory prestrain]
{Dimension reduction for thin films with transversally varying
  prestrain:\\ the oscillatory and the non-oscillatory case} 
\author{Marta Lewicka and Danka Lu\v ci\'c}
\address{Marta Lewicka: University of Pittsburgh, Department of Mathematics, 
139 University Place, Pittsburgh, PA 15260}
\address{Danka Lu\v ci\'c: SISSA, via Bonomea 265, 34136 Trieste, Italy}
\email{lewicka@pitt.edu, dlucic@sissa.it} 

\begin{abstract}
We study the non-Euclidean (incompatible) elastic energy functionals in the
description of prestressed thin films, at their singular limits ($\Gamma$-limits)
as $h\to 0$ in the film's thickness $h$. Firstly, we extend the prior
results \cite{LP11, BLS, LRR} to arbitrary incompatibility metrics
that depend on both the midplate and the transversal variables (the
``non-oscillatory'' case). Secondly, we analyze a more general class of
incompatibilities, where the transversal dependence of the lower order
terms is not necessarily linear (the ``oscillatory'' case), extending
the results of \cite{ALL, Sch} to arbitrary metrics and higher order scalings. We exhibit
connections between the two cases via projections of appropriate
curvature forms on the polynomial tensor spaces. We also show the
effective energy quantisation in terms of scalings as a power  of $h$
and discuss the scaling regimes $h^2$ (Kirchhoff), $h^4$ (von
K\'arm\'an) in the general case, as well as all possible (even powers)
regimes for conformal metrics, thus paving the way to the subsequent
complete analysis of the non-oscillatory setting in \cite{Lewicka_last}.  Thirdly, we prove the
coercivity inequalities for the singular limits at $h^2$- and $h^4$-
scaling orders, while disproving the full coercivity of the classical
von K\'arm\'an energy functional at scaling $h^4$.
\end{abstract}

\date{\today} 

\maketitle

\section{Introduction}

The purpose of this paper is to further develop the 
analytical tools for understanding the me\-cha\-nisms through which
the local properties of a material lead to changes in its mechanical responses.

Motivated by the idea of imposing and controlling the {\em prestrain} (or
{\em``misfit''}) field in order to cause the plate to achieve a desired shape,
our work is concerned with the analysis of thin elastic films
exhibiting residual stress at free equilibria. Examples of this type of
structures and their actuations include: plastically strained sheets,
swelling or shrinking gels, growing tissues
such as leaves, flowers or marine invertebrates,  nanotubes, atomically thin graphene
layers, etc.  In the same vein, advancements in the construction of novel materials in thin film format 
require an analytical insight how the 
parameters affect the 
product and how to mimic the architectures found in nature.

In this paper, we will be concerned with the forward problem
associated to the mentioned structures, based on the minimization of
the elastic energy with incorporated inelastic effects.


\subsection{The set-up of the problem}

Let $\omega\subset\mathbb R^2$ be an open, bounded, connected set
with Lipschitz boundary. We consider a family of thin hyperelastic
sheets occupying the reference domains: 
$$\Omega^h=\omega\times \Big(-\frac{h}{2},\frac{h}{2}\Big)\subset
\R^3, \qquad 0<h\ll 1.$$
A typical point in $\Omega^h$ is denoted by
$x=(x_1,x_2,x_3)=(x',x_3)$. For $h=1$ we use the notation
$\Omega=\Omega^1$ and view $\Omega$ as the referential rescaling of
each $\Omega^h$ via: $\Omega^h\ni (x', x_3)\mapsto (x', x_3/h)\in\Omega$.

\medskip

In this paper we study the limit behaviour, as $h\to 0$, of the energy functionals:
\begin{equation}\label{functional}
\mathcal{E}^h(u^h)=\frac{1}{h}\int_{\Omega^h} W\big(\nabla u^h(x)G^h(x)^{-1/2}\big)\,\rmd x
\end{equation}
defined on vector fields  $u^h\in \WW^{1,2}(\Omega^h, \mathbb{R}^3)$,
that are interpreted as deformations of $\Omega^h$. We view the limit in
the vanishing thickness $h$ as the {\em singular limit}: indeed at $h=0$
the three-dimensional sheets $\Omega^h$ are  reduced to the two-dimensional midplate
$\omega$, and it is the goal of this paper to derive the energy of
its (non-equidimensional) deformations $y:\omega\to\mathbb{R}^{3}$
relevant for the asymptotics of the minimizing sequences of the
(equidimensional) deformations in (\ref{functional}).
The sheets $\Omega^h$ are characterized by the smooth incompatibility
(Riemann metric) tensors $G^h\in\mathcal{C}^\infty(\bar\Omega^h,
\mathbb{R}^{3\times 3}_{\mathrm{sym, pos}})$, satisfying the following structure
assumption, referred to as {\em``oscillatory''}:
\begin{equation} \label{O}\tag{O}
\left[~\mbox{\begin{minipage}{15cm} $$\mbox{\sc Oscillatory case}:\vspace{1mm}$$
$$G^h(x) = \mathcal{G}^h(x', \frac{x_3}{h}) \quad \mbox{ for all } ~x=(x',x_3)\in \Omega^h,$$
$$\mathcal{G}^h(x', t) = \bar{\mathcal{G}}(x') + h \mathcal{G}_1(x', t) + 
\frac{h^2}{2} \mathcal{G}_2(x', t) + o(h^2)\in\mathcal{C}^\infty(\bar\Omega,\RRR_{\mathrm{sym, pos}}),$$
where $\bar{\mathcal{G}}\in\mathcal{C}^\infty(\bar\omega,\RRR_{\mathrm{sym,
    pos}})$, $\mathcal{G}_1, \mathcal{G}_2\in\mathcal{C}^\infty(\bar\Omega,\RRR_{\mathrm{sym}})$
and $\int_{-1/2}^{1/2}\mathcal{G}_1(x',t)\,\rmd t=0$ for all $x'\in \bar \omega.$
\end{minipage}}\right.
\end{equation}
The requirement of $\bar{\mathcal{G}}$ being independent of the
transversal variable $t\in (-1/2, 1/2)$ is essential for the energy
scaling order: $\inf\mathcal{E}^h\leq Ch^2$. The 
zero mean requirement  on $\mathcal{G}_1$ can be relaxed to requesting
that $\int_{-1/2}^{1/2}\mathcal{G}_1(x',t)_{2\times 2}\,\rmd
t$ be a linear strain with respect to the leading order midplate
metric $(\bar{\mathcal{G}})_{2\times 2}$ (in case
$(\bar{\mathcal{G}})_{2\times 2}=Id_2$ the sufficient and necessary
condition for this to happen is $\curl^{\trsp}\curl
\int_{-1/2}^{1/2}\mathcal{G}_1(x',t)_{2\times 2}\,\rmd t=0$; this case
has been studied in \cite{ALL} where $\bar{\mathcal{G}}=Id_3$),
and we also conjecture that it can be removed altogether, which will be the content of
future work. In the present work, we assume the said condition in
light of the special case (\ref{NO}) below.  

\medskip

\noindent We refer to the family of films $\Omega^h$ prestrained by metrics in (\ref{O}):
\begin{equation}\label{full}
G^h(x) = \bar{\mathcal{G}}(x') + h \mathcal{G}_1(x',
\frac{x_3}{h})+ \frac{h^2}{2} \mathcal{G}_2(x', \frac{x_3}{h}) +
o(h^2)  \quad \mbox{ for all } ~x=(x',x_3)\in \Omega^h, 
\end{equation}
as {"oscillatory"}, and note that this set-up includes a subcase of a
single metric $G^h=G$, upon taking:
$$\mathcal{G}_1(x', t) = t\bar{\mathcal{G}}_1(x'),  
\qquad \mathcal{G}_2(x', t) = t^2\bar{\mathcal{G}}_2(x').$$
We refer to this special case as {\em ``non-oscillatory'}';
formula (\ref{full}) becomes then Taylor's expansion in:
\begin{equation} \label{NO}\tag{NO}
\left[~\mbox{\begin{minipage}{14.9cm} $$\mbox{\sc Non-oscillatory case}:\vspace{1mm}$$
$$G^h = G_{\mid \bar\Omega^h} \quad \mbox{ for some } ~
G\in\mathcal{C}^\infty(\bar\Omega,\RRR_{\mathrm{sym, pos}}),$$
$$ G^h(x) = \bar{\mathcal{G}}(x')+x_3\partial_3 G(x', 0)
+ \frac{x_3^2}{2}\partial_{33}G(x',0) + o(x_3^2) \quad \mbox{ for all }~x=(x',x_3)\in \Omega^h.$$
\end{minipage}}\right.
\end{equation}

\noindent Mechanically, the assumption (\ref{NO}) describes thin
sheets that have been cut out of a single specimen block
$\Omega$, prestrained according to a fixed (though arbitrary) tensor $G$.
\noindent As we shall see, the general case (\ref{O}) can be reduced
to (\ref{NO}) via the following {\em effective metric}: 
\begin{equation} \label{EO}\tag{EF}
\left[~\mbox{\begin{minipage}{15cm} $$\mbox{\sc Effective non-oscillatory case}:\vspace{1mm}$$ 
$$\bar G^h(x) = \bar{G}(x) = \bar{\mathcal{G}} (x') + x_3
\bar{\mathcal{G}}_1(x') + \frac{x_3^2}{2}\bar{\mathcal{G}}_2(x')
\quad \mbox{ for all } \; x=(x', x_3)\in\Omega^h,\vspace{1mm}$$ 
$\mbox{where: }\,\, \bar{\mathcal{G}}_1(x')_{2\times 2} =
12\int_{-1/2}^{1/2}t\mathcal{G}_1 (x', t)_{2\times 2}
\,\rmd t$, $~~ \bar{\mathcal{G}}_1(x')e_3 = -60
\int_{-1/2}^{1/2}(2t^3-\frac{1}{2}t)\mathcal{G}_1(x',t)e_3 \,\rmd t$, \vspace{-2mm}\\
$\mbox{and: }\,\, ~~\,\bar{\mathcal{G}}_2(x')_{2\times 2} =
30\int_{-1/2}^{1/2} (6t^2-\frac{1}{2})\mathcal{G}_2(x', t)_{2\times 2}
\,\rmd t.$
\end{minipage}}\right.
\end{equation}

\medskip

In (\ref{functional}), the homogeneous elastic energy density $W:\mathbb{R}^{3\times
  3}\rightarrow [0,\infty]$ is a Borel measurable function, assumed to satisfy the following properties:
\begin{itemize}
\item [(i)] 
$W(RF)=W(F)$ for all $R\in \SO(3)$ and $F\in \RRR$,
\item [(ii)]
$W(F)=0$ for all $F\in \SO(3)$,
\item [(iii)] 
$W(F)\geq C\,{\rm dist}^2\big(F,\SO(3)\big)$ for all $F\in\RRR$, with
some uniform constant $C>0$,
\item [(iv)] 
there exists a neighbourhood $\mathcal U$ of $\SO(3)$ such that $W$ is
finite and  $\mathcal{C}^2$ regular on $\mathcal U$.
\end{itemize}

\medskip

We will be concerned with the regimes of curvatures of $G^h$ in
(\ref{O}) which yield the incompatibility rate, quantified by $\inf \mathcal{E}^h$, of order higher than
$h^2$ in the plate's thickness $h$. With respect to the prior works in
this context, the present paper proposes the following three new contributions.

\subsection{New results of this work: Singular energies in the non-oscillatory case}

\subsubsection{Kirchhoff scaling regime}\vspace{-3mm}
We begin by deriving (in section \ref{2}), the $\Gamma$-limit of the
rescaled energies $\frac{1}{h^2}\mathcal{E}^h$. In the setting of (\ref{NO}), we obtain:
\begin{equation*}
\begin{split}
\mathcal{I}_2(y) & = \frac{1}{2}\big\|Tensor_2 \big\|_{\mathcal{Q}_2}^2=\frac{1}{2}\big\|x_3\big((\nabla
y)^{\trsp}\nabla \vec
b\big)_{\sym}-\frac{1}{2}x_3\partial_3G(x',0)_{2\times 2}\big\|_{\mathcal{Q}_2}^2\\ & =\frac{1}{24}
\big\|\big((\nabla y)^{\trsp}\nabla \vec
b\big)_{\sym}-\frac{1}{2}\partial_3G(x',0)_{2\times 2} \big\|_{\mathcal{Q}_2}^2.
\end{split}
\end{equation*}
We now explain the notation above. Firstly, $\|\cdot
\|_{\mathcal{Q}_2}$ is a weighted $L^2$ norm in (\ref{poly_norm}) on
the space $\mathbb{E}$ of $\R^{2\times 2}_{\sym}$-valued tensor fields on
$\Omega$. The weights in (\ref{Qform}) are determined by the elastic
energy $W$ together with the leading order metric coefficient 
$\bar{\mathcal{G}}$. The functional $\mathcal{I}_2$ is defined on the
set of isometric immersions $\mathcal{Y}_{\bar{\mathcal{G}}_{2\times
    2}} = \{y\in W^{2,2}(\omega,\R^3);~ (\nabla y)^{\trsp}\nabla 
y=\bar{\mathcal{G}}_{2\times 2}\}$; each such immersion generates
the corresponding Cosserat vector $\vec b$, uniquely given by
requesting: $\big[\partial_1y,~\partial_2 y,~ \vec b\big]\in
SO(3)\bar{\mathcal{G}}^{1/2}$ on $\omega$. 
The family of energies obtained in this manner is
parametrised by all matrix fields $S\in \mathcal{C}^{\infty}(\bar\omega,
\mathbb{R}^{2\times 2}_{\mathrm{sym}})$ and $T\in\mathcal{C}^{\infty}(\bar\omega,
\mathbb{R}^{3\times 3}_{\mathrm{sym, pos}})$, namely:
$\mathcal{I}_2^{T,S}(y)=\frac{1}{24}\|((\nabla y)^{\trsp}\nabla\vec
b(y))_{\sym}-S\|_{\mathcal{Q}_2}^2$ defined on the set
$\mathcal{Y}_{T_{2\times 2}}$, where one interprets $T$ as the
leading order prestrain $\bar{\mathcal{G}}$ and $S$ as its first
order correction $\frac{1}{2}\partial_3G(x',0)_{2\times 2}$.

The energy $\mathcal{I}_2$ measures 
the bending quantity $Tensor_2$ which is linear in $x_3$, resulting
in its reduction to the single nonlinear bending term, that equals the
difference of the curvature form $\big((\nabla y)^{\trsp}\nabla \vec 
b\big)_{\sym}$ from the preferred curvature
$\frac{1}{2}\partial_3G(x',0)_{2\times 2}$. The same
energy has been derived in \cite{LP11, BLS} under the assumption that
$G$ is independent of $x_3$ and in \cite{KS14} for a
general manifold $(M^n, g)$ with any codimension submanifold $(N^k,
g_{\mid N})$ replacing the midplate $\omega\times \{0\}$. Since our
derivation of $\mathcal{I}_2$ is a particular case of the result in
case (\ref{O}), we still state it here for completeness. 

In section \ref{conditions_flateness} we identify the necessary and sufficient conditions for
$\min\mathcal{I}_2=0$ (when $\omega$ is simply connected), in terms of the vanishing of the Riemann
curvatures $R_{1212}, R_{1213}, R_{1223}$ of $G$ at $x_3=0$. In this
case, it follows that $\inf \mathcal{E}^h\leq Ch^4$. For the discussed
case (\ref{NO}), the recent work \cite{MS18} generalized the same
statements for arbitrary dimension and codimension.

\subsubsection{Von K\'arm\'an scaling regime}
In section \ref{4} we derive the $\Gamma$-limit of 
$\frac{1}{h^4}\mathcal{E}^h$, which is given by:
$$\mathcal{I}_4(V,\mathbb{S}) = \frac{1}{2}\big\|Tensor_4\big\|_{\mathcal{Q}_2}^2,$$
defined on the spaces of: finite strains
$\mathscr{S}_{y_0}=\{\mathbb{S}=\lim_{n\to\infty, L^2}\big((\nabla
y_0)^{\trsp}\nabla  w_n\big)_{\sym}; ~ w_n\in W^{1,2}(\omega,
\mathbb{R}^3)\}$ and first order infinitesimal
  isometries $\mathscr{V}_{y_0}=\{V\in W^{2,2}(\omega,
\mathbb{R}^3); ~ \big((\nabla y_0)^{\trsp}\nabla V\big)_{\sym} =0\}$
on the deformed midplate $y_0(\omega)\subset\R^3$. Here, $y_0$ is the unique smooth isometric
  immersion of $\bar{\mathcal{G}}_{2\times 2}$ for which
$\mathcal{I}_2(y_0)=0$; recall that it generates the corresponding Cosserat's vector $\vec b_0$. 

The expression in $Tensor_4$ is quite
complicated but it has the structure of a quadratic polynomial in
$x_3$. A key tool for identifying this expression, also in the general case (\ref{O}),
involves the subspaces $\{\mathbb{E}_n\subset \mathbb{E}\}_{n\geq 1}$ in
(\ref{poly_space}), consisting of the tensorial polynomials in $x_3$ of order
$n$. The bases of $\{\mathbb{E}_n\}$  are then naturally given in terms of
the Legendre polynomials $\{p_n\}_{n\geq 0}$ on $(-\frac{1}{2},
\frac{1}{2})$. Since $Tensor_4\in \mathbb{E}_2$, we write the decomposition:
$$Tensor_4 = p_0(x_3) Stretching_4 + p_1(x_3) Bending_4 + p_2(x_3)Curvature_4,$$
which, as shown in section \ref{sec77}, results in:
\begin{equation*}
\begin{split}
\mathcal{I}_4(V,\mathbb{S}) & = \frac{1}{2}\Big( \big\|Stretching_4\big\|_{\mathcal{Q}_2}^2+
\big\|Bending_4\big\|_{\mathcal{Q}_2}^2 + \big\|Curvature_4\big\|_{\mathcal{Q}_2}^2\Big)\\ & 
= \frac{1}{2} \big\|\mathbb{S} +\frac{1}{2}(\nabla V)^{\trsp}\nabla V +
\frac{1}{24}(\nabla\vec b_0)^{\trsp}\nabla \vec b_0 -
\frac{1}{48}\partial_{33}G(x',0)_{2\times 2}\big\|_{\mathcal{Q}_2}^2\\
& \qquad + \frac{1}{24}\big\|\big[\langle \nabla_i\nabla_jV,\vec
b_0\rangle\big]_{i,j=1,2}\big\|_{\mathcal{Q}_2}^2 +
\frac{1}{1440}\big\|\big[R_{i3j3}(x',0)\big]_{i,j=1,2}\big\|_{\mathcal{Q}_2}^2 .
\end{split}
\end{equation*}
Above, $\nabla_i$ denotes the covariant differentiation with respect to
the metric $\bar{\mathcal{G}}$ and $R_{i3j3}$ are the potentially
non-zero curvatures of $G$ on $\omega$ at $x_3=0$.

The family of energies obtained in this manner is
parametrised by all quadruples: vector fields  $y_0,\vec
b_0\in\mathcal{C}^\infty(\bar\omega,\mathbb{R}^3)$ satisfying
$\det\big[\partial_1y_0,~\partial_2y_0, ~ \vec b_0\big]>0$, matrix
fields $T\in\mathcal{C}^\infty(\bar\omega, \mathbb{R}_{\sym}^{2\times
  2})$, and numbers $r\in\mathbb{R}$ in the range (the left parentheses
in the last interval below may be open or closed):
\begin{equation}\label{5}
\begin{split}
r\in\frac{2}{5}\Big\{\|T-\big((\nabla y_0)^{\trsp}\nabla \vec
d_0\big)_{\sym}\|_{\mathcal{Q}_2}^2; ~ \vec d_0\in
\mathcal{C}^\infty(\bar\omega, \mathbb{R}_{\sym}^{2\times 2}) \Big\}
=\Big[ \frac{2}{5}\mbox{dist}^2_{\mathcal{Q}_2}(T, \mathscr{S}_{y_0}), +\infty\Big).
\end{split}
\end{equation}
The functionals are then:
$\mathcal{I}^{y_0, \vec b_0, T, r}_4(V,\mathbb{S})=  \frac{1}{2}\|\mathbb{S}+\frac{1}{2}(\nabla
V)^{\trsp}\nabla V - T\|_{\mathcal{Q}_2}^2 +
\frac{1}{24}\|\big[\langle\nabla_i\nabla_jV, \vec
b_0\rangle\big]_{ij}\|_{\mathcal{Q}_2}^2 + r,$ defined on the linear space $\mathcal{V}_{y_0}\times
\mathscr{S}_{y_0}$.  Particular cases where the range of $r$ may be identified are:
\begin{itemize}
\item[(i)]  $y_0=id_2$. Then $\mathscr{S}_{y_0}= \{\mathbb{S}\in
L^2(\omega, \mathbb{R}^{2\times 2}_{\sym});
~\mbox{curl}^{\trsp}\mbox{curl}\,\mathbb{S}=0\}$ and the range of $r$ is
defined by the appropriate norm of: $\mbox{curl}^{\trsp}\mbox{curl}\, T$.

\item[(ii)] Gauss curvature $\kappa((\nabla y_0)^{\trsp}\nabla y_0)>0$ in $\bar
\omega$. Then in \cite{lemopa} it is shown that $\mathscr{S}_{y_0}=
L^2(\omega, \mathbb{R}^{2\times 2}_{\sym})$. The range of possible $r$
(for any $T$) is then: $[0, +\infty)$.
\end{itemize}

When $y_0=id_2$ (which occurs automatically when $\bar{\mathcal{G}}=Id_3$), then $\vec
b_0=e_3$ and the first two terms in $\mathcal{I}_4$ reduce to the stretching and
the linear bending contents of the classical von K\'arm\'an energy. The third
term is purely metric-related and measures the non-immersability of $G$
relative to the present quartic scaling. These findings generalize the
results of \cite{BLS} valid for $x_3$-independent $G$ in
(\ref{NO}). We also point out that, following the same general
principle in the $h^2$- scaling regime, one may readily decompose:
$$Tensor_2 = p_0(x_3) Stretching_2 + p_1(x_3) Bending_2;$$ 
since $Tensor_2$ is already a multiple of $x_3$, then
$Stretching_2=0$ in the ultimate form of $\mathcal{I}_2$.

It is not hard to deduce (see section \ref{sec8h4}) that the necessary
and sufficient conditions for having $\min \mathcal{I}_4=0$ are
precisely that $R_{ijkl}\equiv 0$ on $\omega\times \{0\}$, for all
$i,j,k,l=1\ldots 3$. In section \ref{example_conformal} we then
analyze the conformal non-oscillatory metric
$G=e^{2\phi(x_3)}Id_3$ and show that different orders of vanishing of
$\phi$ at $x_3=0$ correspond to different even orders of scaling of
$\mathcal{E}^h$ as $h\to 0$:
$$\phi^{(k)}(0)=0\quad \mbox{for } k=1\ldots n-1\quad \mbox{ and
}\quad  \phi^{(n)}(0)\neq 0 \quad \Leftrightarrow \quad 
ch^{2n}\leq \inf\mathcal{E}^h\leq Ch^{2n}$$
with the lower bound: $\inf \mathcal{E}^h\geq c_nh^n\big\| \big[\partial_3^{(n-2)}
R_{i3j3}(x',0)\big]_{i,j=1,2}\big\|_{\mathcal{Q}_2}^2$. These findings
are consistent with and pave the way for the follow-up paper
\cite{Lewicka_last}, which completes the scaling analysis of
$\mathcal{E}^h$ in the non-oscillatory case, including the derivation of
$\Gamma$-limits of $h^{-2n}\mathcal{E}^h$ for all
$n\geq 1$, and proving the energy quantisation in the sense that
the even powers $2n$ of $h$ are indeed the only possible ones (all of
them are also attained).

\subsection{New results of this work: Singular energies in the oscillatory case} 

We show that the analysis in the general case (\ref{O})
may follow a similar procedure, where we first project the limiting
quantity $Tensor^O$ on an appropriate polynomial space and then
decompose the projection along the respective Legendre basis. For the
$\Gamma$-limit of $\frac{1}{h^2}\mathcal{E}^h$ in section \ref{2}, we show that:
\begin{equation*}
\begin{split}
Tensor_2^O & =  x_3\big((\nabla y)^{\trsp}\nabla\vec
b\big)_{\sym}-\frac{1}{2}(\mathcal{G}_1)_{2\times 2} =
p_0(x_3) Stretching_2^O + p_1(x_3) Bending_2^O +
Excess_2, \\ &  \mbox{with }\, Excess_2 = Tensor _2^O - \mathbb{P}_1 (Tensor_2^O),
\end{split}
\end{equation*}
where $\mathbb{P}_1$ denotes the orthogonal projection on
$\mathbb{E}_1$. Consequently:
\begin{equation*}
\begin{split}
\mathcal{I}_2^O(y) & =  \frac{1}{2}\Big(\big\|Stretching_2^O\big\|_{\mathcal{Q}_2}^2+
\big\|Bending_2^O\big\|_{\mathcal{Q}_2}^2 + \big\|Excess_2\big\|_{\mathcal{Q}_2}^2\Big)\\ & =
\frac{1}{24}\big\|\big((\nabla y)^{\trsp}\nabla\vec
b\big)_{\sym}-\frac{1}{2}(\bar{\mathcal{G}}_1)_{2\times 2} \big\|_{\mathcal{Q}_2}^2+
\frac{1}{8}\dist^2_{\mathcal{Q}_2}\big((\mathcal{G}_1)_{2\times 2}, \mathbb{E}_1\big),
\end{split}
\end{equation*}
where again $Stretching_2^O=0$ in view of the assumed 
$\int_{-1/2}^{1/2}\mathcal{G}_1 \,\rmd x_3= 0$. For the same reason:
$$Excess_2=-\frac{1}{2}\Big( (\mathcal{G}_1)_{2\times
  2}-\mathbb{P}_1\big((\mathcal{G}_1)_{2\times 2}\big)\Big) = -\frac{1}{2}\Big( (\mathcal{G}_1)_{2\times
  2}-12\int_{-1/2}^{1/2}x_3(\mathcal{G}_1)_{2\times 2} \,\rmd x_3\Big) $$ 
and also: $\mathbb{P}_1\big((\mathcal{G}_1)_{2\times 2}\big)=x_3
(\bar{\mathcal{G}}_1)_{2\times 2}$ with $(\bar{\mathcal{G}}_1)_{2\times 2}$
defined in (\ref{EO}). The limiting oscillatory energy
$\mathcal{I}_2^O$ consists thus of the bending term that coincides with
$\mathcal{I}_2$ for the effective metric $\bar G$, plus the purely
metric-related excess term. A special case of $\mathcal{I}_2^O$ when $\bar{\mathcal{G}}=Id_3$
and without analyzing the excess term, has been derived in \cite{ALL},
following the case with $G^h=Id_3+h\mathcal{G}_1(\frac{x_3}{h})$
considered in \cite{Sch}. An excess term has also been
present in the work \cite{KO} on rods with misfit; we do not attempt
to compare our results with studies of dimension reduction for
rods; the literature there is abundant.

It is easy to observe that:  $\min\mathcal{I}_2^O=0$ if and only if $(\mathcal{G}_1)_{2\times 2}=x_3
(\bar{\mathcal{G}}_1)_{2\times 2}$ on $\omega\times \{0\}$. We
show in section \ref{sec5} that this automatically implies:
$\inf\mathcal{E}^h\leq Ch^4$. The $\Gamma$-limit of
$\frac{1}{h^4}\mathcal{E}^h$ is further derived in sections \ref{4} and
\ref{sec77}, by considering the decomposition:
\begin{equation*}
\begin{split}
Tensor_4^O & = p_0(x_3) Stretching_4^O + p_1(x_3) Bending_4^O +
p_2(x_3)Curvature_4^O + Excess_4, \\ &  \mbox{with }\, Excess_4 =
Tensor _4^O - \mathbb{P}_2 (Tensor_4^O). 
\end{split}
\end{equation*}
It follows that:
\begin{equation*}
\begin{split}
\mathcal{I}_4^O(V,\mathbb{S}) & = \frac{1}{2}\Big( \big\|Stretching_4^O\big\|_{\mathcal{Q}_2}^2+
\big\|Bending_4^O\big\|_{\mathcal{Q}_2}^2 +
\big\|Curvature_4^O\big\|_{\mathcal{Q}_2}^2+ \big\|Excess_4\big\|_{\mathcal{Q}_2}^2\Big)\\ & 
= \frac{1}{2} \big\|\mathbb{S} +\frac{1}{2}(\nabla V)^{\trsp}\nabla V
+ B_0 \big\|_{\mathcal{Q}_2}^2 + \frac{1}{24}\big\|\big[\langle \nabla_i\nabla_jV,\vec
b_0\rangle\big]_{i,j=1,2} + B_1\big\|_{\mathcal{Q}_2}^2 \\ & \qquad +
\frac{1}{1440}\big\|\big[R_{i3j3}(x',0)\big]_{i,j=1,2}\big\|_{\mathcal{Q}_2}^2 \\
& \qquad  +\frac{1}{2}\dist^2_{\mathcal{Q}_2}\Big(\frac{1}{4}(\mathcal{G}_2)_{2\times 2}
-\int_0^{x_3}\big[\nabla_i\big((\mathcal{G}_1e_3)
-\frac{1}{2}(\mathcal{G}_1)_{33}e_3\big)\big]_{i,j=1,2, \sym}\,\rmd t, \mathbb{E}_2\Big),
\end{split}
\end{equation*}
where $R_{1313},R_{1323},R_{2323}$ are the respective Riemann
curvatures of the effective metric $\bar G$ in (\ref{EO})
at $x_3=0$. The corrections $B_0$ and $B_1$ coincide with the same
expressions written for $\bar G$ under two extra constraints (see
Theorem \ref{Gliminf4-2}), that can be seen as the $h^4$-order
counterparts of the $h^2$-order condition $\int_{-1/2}^{1/2}\mathcal{G}_1
\,\rmd x_3= 0$ assumed throughout. In case these conditions
are valid, the functional $\mathcal{I}_4^O$ is the sum of the
effective stretching, bending and curvature in $\mathcal{I}_4$ for
$\bar G$, plus the additional purely metric-related excess term.

\subsection{New results of this work: coercivity of $\mathcal{I}_2$  and $\mathcal{I}_4$ }
We additionally analyze the derived limiting functionals by identifying their
kernels, when nonempty. In section \ref{sec_coer} we show that the
kernel of $\mathcal{I}_2$ consists of the rigid motions of a
single smooth deformation $y_0$ that solves:
$$(\nabla y_0)^{\trsp}\nabla y_0=\bar{\mathcal{G}}_{2\times 2}, \qquad 
\big((\nabla y_0)^{\trsp}\nabla\vec
b_0\big)_{\sym}=\frac{1}{2}\partial_3G(x',0)_{2\times 2}.$$
Further, $\mathcal{I}_2(y)$ bounds from above the squared distance of
an arbitrary $W^{2,2}$ isometric immersion $y$ of the midplate metric
$\bar{\mathcal{G}}_{2\times 2}$, from the indicated kernel of $\mathcal{I}_2$.

In section \ref{sec8h4} we consider the case of $\mathcal{I}_4$. We
first identify (see Theorem \ref{kernel_vK})
the zero-energy displacement-strain couples $(V, \mathbb{S})$. In
particular, we show that the minimizing displacements are exactly the
linearised rigid motions of the referential $y_0$. We then prove that
the bending term in $\mathcal{I}_4$, which is solely a function of
$V$, bounds from above the squared distance of an arbitrary $W^{2,2}$ displacement
obeying $\big((\nabla y_0)^{\trsp}\nabla V\big)_{\sym} = 0$, from the
indicated minimizing set in $V$. On the other hand, the full
coercivity result involving minimization in  both $V$ and $\mathbb{S}$ is
false. In Remark \ref{nocoer} we exhibit an example in the
setting of the classical von K\'arm\'an functional, where
$\mathcal{I}_4(V_n, \mathbb{S}_n)\to 0$ as $n\to\infty$, but the
distance of $(V_n, \mathbb{S}_n)$ from the kernel of $\mathcal{I}_4$ remains uniformly
bounded away from $0$. We note that this lack of coercivity is not prevented
by the fact that the kernel is finite dimensional.

\subsection{Other related works} 
Recently, there has been a sustained interest in studying shape formation driven by internal
prestrain, through the experimental, modelling via formal methods,
numerics, and analytical arguments \cite{sharon, JM,  ESK1,
  KM14}. General results have been derived in the abstract
setting of Riemannian manifolds: in \cite{KS14, KM14} 
$\Gamma$-convergence statements were proved for any dimension ambient
manifold and codimension midplate, in the scaling regimes
$\mathcal{O}(h^2)$ and $\mathcal{O}(1)$, respectively. In a 
work parallel to ours \cite{MS18}, the authors analyze scaling orders
$o(h^2)$, $\mathcal{O}(h^4)$ and $o(h^4)$, extending 
condition (\ref{I-II2}), Lemma \ref{higher_order_scaling_sequence}
in (\ref{NO}) case, and condition (\ref{y11}) to
arbitrary manifolds. Although they do not identify the $\Gamma$-limits
of the rescaled energies $\mathcal{E}^h$, they are able to provide the
revealing lower bounds for $\inf \mathcal{E}^h$ in terms of the appropriate curvatures.

Higher energy scalings $\inf\mathcal{E}^h\sim h^\beta$ than the ones
analyzed in the present paper may result from the interaction of the metric with boundary 
conditions or external forces, leading to the ``wrinkling-like''
effects. Indeed, our setting pertains to the ``no wrinkling'' regime where
$\beta \geq 2$ and the prestrain metric admits a $W^{2,2}$ isometric
immersion. While the systematic description of the singular limits at scalings $\beta<2$ is
not yet available, the following studies are examples of the variety
of emerging patterns. In \cite{JS, BCDM, BCDM2}, energies leading to the buckling- or
compression- driven {\em blistering}  in a thin film breaking
away from its substrate and under
clamped boundary conditions, are discussed ($\beta=1)$. Paper \cite{KoBe}
  displays dependence of the energy minimization  
on boundary conditions and classes of admissible deformations, while
\cite{Bella} discusses coarsening of {\em folds}
in  hanging drapes, where the energy identifies the number of generations of coarsening. 
In \cite{Tobasco}, {\em wrinkling} patterns are obtained, reproducing the
experimental observations when a thin shell is placed on
a liquid bath ($\beta=1$), while \cite{Obrien}  analyses wrinkling
in the center of a stretched and twisted ribbon ($\beta=4/3$).
In \cite{Maggi, Venka}, energy levels of the {\em origami patterns} in
paper crumpling are studied ($\beta=5/3$). See also \cite{Olber1, Olber2, Olber3} for an
analysis of the {\em conical singularities} 
($\mathcal{E}^h\sim h^2\log(1/h)$).  
We remark that the mentioned papers do not address the dimension
reduction, but rather analyze the chosen actual configuration of the prestrained sheet. 
Closely related is also the literature on shape selection in non-Euclidean
plates, exhibiting hierarchical {\em buckling patterns} in zero-strain
plates ($\beta=2$), where the complex morphology is due to the
non-smooth energy minimization \cite{Ge1, Ge2, Ge3}.

Various geometrically nonlinear thin plate theories have been used to
analyze the self-similar structures with metric asymptotically
flat at infinity \cite{1b}, a disk with edge-localized growth \cite{ESK1}, the shape of a
long leaf \cite{18b}, or torn plastic sheets \cite{22a}. In
\cite{DB, DCB} a variant of the F\"oppl-von K\'arm\'an equilibrium
equations has been formally derived from finite incompressible elasticity, via the
multiplicative decomposition of deformation gradient \cite{RHM}
similar to ours. See also models related to wrinkling of paper in
areas with high ink or paint coverage, grass
blades, sympatelous (meaning “with fused petals”) flowers and studies
of movement of micro-organisms that share certain characteristics of animals and
plants, called the euglenids \cite {DCB, BMT, D}.   
The forward and inverse problems in the self-folding of thin
sheets of patterned hydrogel bilayers are discussed in \cite{ADLL}.

On the frontiers of experimental modeling of shape formation,
we mention the {\em halftone gel lithography} method for polymeric materials that
can swell by imbibing fluids \cite{9,10,11, 12}. By blocking the
ability of portions of plate to swell or causing them to swell inhomogeneously, it is possible to have
the plate assume a variety of deformed shapes. Even more sophisticated
techniques of {\em biomimetic 4d printing} allow for engineering of the
3d shape-morphing systems that mimic nastic plant motions where organs
such as tendrils, leaves and flowers respond to the environmental
stimuli \cite{biomi}. Optimal control in such systems has been studied
in \cite{JM}, see also \cite{ALP}.

In \cite{lemapa2, lemapa2new, LOP}, derivations similar to the results
of the present paper were carried out under a different assumption on the 
asymptotic behavior of the prestrain, which also implied
energy scaling $h^\beta$ in non-even regimes of $\beta>2$. In
\cite{lemapa2} it was shown that the resulting Euler-Lagrange
equations of the residual energy are identical to those describing the effects of growth
in elastic plates \cite{18b}. In \cite{LOP}, a model with a
Monge-Amp\`ere constraint was derived and analysed from various aspects. 

We finally mention the paper \cite{Lewicka_last},
completed after the submission of the present article, which resolves the scaling analysis of $\mathcal{E}^h$
together with the derivation of $\Gamma$-limits of $h^{-2n}\mathcal{E}^h$, for all
$n\geq 1$. There, we identify equivalent conditions for the validity
of the scalings $h^{2n}$ in terms of vanishing of the
Riemann curvatures,  for an arbitrary non-oscillatory metric $G$, up
to appropriate orders and in terms of the matched isometry  
expansions. We also establish the asymptotic behaviour of the
minimizing immersions of $G$ as $h\to 0$ and prove 
the energy quantisation, in the sense that
the even powers $2n$ of $h$ are indeed the only possible ones (all of them are also attained).

\subsection{Notation} 
Given a matrix $F\in\mathbb{R}^{n\times n}$, we denote 
its transpose by $F^{\trsp}$ and its symmetric part by $F_{\mathrm{sym}} =
\frac{1}{2}(F + F^{\trsp})$. The space of symmetric $n\times n$ matrices is
denoted by $\R^{n\times n}_{\mathrm{sym}}$, whereas $\R^{n\times
  n}_{\mathrm{sym, pos}}$ stands for the space of symmetric, positive
definite  $n\times n$ matrices.
By $SO(n) = \{R\in\mathbb{R}^{n\times n}; ~ R^{\trsp} = R^{-1} \mbox{ and }
  \det R=1\}$ we mean the group of special rotations; its tangent
  space at $Id_n$ consists of skew-symmetric matrices:  $T_{Id_n}SO(n)=so(n) =
\{F\in \mathbb{R}^{n\times n}; ~ F_{\mathrm{sym}}=0\}$.
We use the matrix norm $|F| = (\mbox{trace}(F^{\trsp} F))^{1/2}$, which
is induced by the inner product $\langle F_1 : F_2\rangle =
\mbox{trace}(F_1^{\trsp} F_2)$. The $2 \times 2$ principal minor of 
$F\in\mathbb{R}^{3\times 3}$ is denoted by $F_{2\times 2}$.
Conversely, for a given $F_{2\times 2} \in \mathbb{R}^{2\times 2}$, the $3 \times 3$ matrix with
principal minor equal $F_{2\times 2}$ and all other entries equal to
$0$, is denoted by $F_{2\times 2}^*$. Unless specified otherwise, all limits are
taken as the thickness parameter $h$ vanishes: $h\to 0$. By $C$ we denote any
universal positive constant, independent of $h$. We use the Einstein
summation convention over repeated lower and upper indices running from $1$ to $3$.

\subsection{Acknowledgments}
M.L. was supported by the NSF grant DMS-1613153. D.L. acknowledges the
support of V. Agostiniani and discussions related to $h^2$-scaling regime. 
We are grateful to Y. Grabovsky for discussions on Remark
\ref{nocoer} and to R.V. Kohn and C. Maor for pointing us to some relevant literature.

\section{Compactness and $ \Gamma$-limit under $Ch^2$ energy bound}\label{2}

Define the matrix fields $\bar A\in \mathcal{C}^\infty(\bar\omega, \mathbb{R}^{3\times
  3}_{\sym,\mathrm{pos}})$ and $A^h, A_1, A_2\in \mathcal{C}^\infty(\bar\Omega, \mathbb{R}^{3\times
  3}_{\sym})$ so that, uniformly for all $(x', x_3)\in\Omega^h$ there holds:
$$A^h(x', x_3) = \mathcal{G}^h(x', x_3)^{1/2} = \bar A(x') + h A_1(x', \frac{x_3}{h}) +
\frac{h^2}{2}A_2(x', \frac{x_3}{h}) + o(h^2).$$
Equivalently, $\bar A, A_1, A_2$ solve the following system of equations:
\begin{equation}\label{Adef}
\bar A^2=\bar{\mathcal{G}}, \qquad 2(\bar A
A_1)_{\sym}=\mathcal{G}_1, \qquad 2A_1^2 + 2(\bar A
A_2)_{\sym}=\mathcal{G}_2\qquad\mbox{in }\; \bar\Omega.
\end{equation}

Under the assumption (\ref{O}), condition (iii) on $W$ easily implies:
\begin{equation*}
\begin{split}
\frac{1}{h}\int_{\Omega^h}\mathrm{dist}^2\big(\nabla u^h(x) \bar A(x')^{-1}, \SO(3)\big)
\,\rmd x & \leq \frac{C}{h}\int_{\Omega^h}\mathrm{dist}^2\big(\nabla u^h(x) A^h(x)^{-1},
\SO(3)\big) +  h^2\,\rmd x\\ &
\leq C\big(\mathcal{E}^h(u^h) + h^2\big). 
\end{split}
\end{equation*}
Consequently, the results of \cite{BLS} automatically give the following
compactness properties of any sequence of deformations with the
quadratic energy scaling:

\begin{thm} \label{compactness_thm}
Assume (\ref{O}). Let $\{u^h\in \WW^{1,2}(\Omega^h,\R^3)\}_{h\to 0}$
be a sequence of deformations satisfying:
\begin{equation}\label{Ch2}
\mathcal{E}^h(u^h)\leq Ch^2.
\end{equation}
Then the following properties hold for the rescalings
$y^h\in \WW^{1,2}(\Omega, \mathbb{R}^3)$ given by:
$$y^h(x', x_3) = u^h(x', hx_3) - \fint_{\Omega^h}u^h\,\rmd x.$$
\begin{itemize}
\item [(i)] There exist $y\in \WW^{2,2}(\omega, \R^3)$ and $\vec b\in
  \WW^{1,2}\cap\LL^{\infty}(\omega,\R^3)$ such that, up to a subsequence:
$$ y^h\to y\quad  \mbox{strongly in }\WW^{1,2}(\Omega, \R^3) \quad \mbox {and} \quad 
\frac{1}{h}\partial_3 y^h\to \vec b \quad  \mbox{strongly in
}\LL^{2}(\Omega, \R^3),\quad \mbox{ as }\,h\to 0.$$
\item [(ii)]The limit deformation $y$ realizes the reduced midplate metric on $\omega$:
\begin{equation}\label{iso}
(\nabla y)^{\trsp}\nabla y=\bar{\mathcal{G}}_{2\times 2}.
\end{equation}
In particular $\partial_1y$, $\partial_2y\in L^\infty(\omega,
\mathbb{R}^{3})$ and the unit normal
${\vec \nu=\frac{\partial_1y\times \partial_2y}{|\partial_1y\times \partial_2y|}}$
to the surface $y(\omega)$ satisfies: $\vec \nu\in \WW^{1,2}\cap\LL^{\infty}(\omega,\R^3)$.
The limit displacement $\vec b$ is the Cosserat field defined via:
\begin{equation} \label{cosserat}
\vec b=(\nabla y) (\bar{\mathcal{G}}_{2\times 2})^{-1}\left[\begin{matrix}\bar{\mathcal{G}}_{13}\\
\bar{\mathcal{G}}_{23}\end{matrix}\right] + \frac{\sqrt{\det \bar{\mathcal{
G}}}}{\sqrt{\det\bar{\mathcal{G}}_{2\times 2}}}\,\vec \nu.
\end{equation}
\end{itemize}
\end{thm}
Recall that the results in \cite{BLS} also give:
\begin{equation}\label{old_Kirch}
\liminf_{h\to 0}\frac{1}{h^2}\frac{1}{h}\int_{\Omega^h}W\big(\nabla
u^h\bar{\mathcal{G}}^{-1/2}\big)\geq \frac{1}{24} 
\int_\omega\mathcal{Q}_2\Big(x', \nabla y(x')^{\trsp}\nabla \vec b(x')\Big) \,\rmd x',
\end{equation}
with the curvature integrand $(\nabla y)^{\trsp}\nabla \vec b$ quantified by the quadratic forms:
\begin{equation}\label{Qform}
\begin{split}
& \mathcal{Q}_2(x', F_{2\times 2}) = \min\left\{
  \mathcal{Q}_3\big(\bar  A(x')^{-1}\tilde F \bar A(x')^{-1}\big);
  ~ \tilde F\in\mathbb{R}^{3\times 3} \mbox{ with }\tilde F_{2\times 2} = F_{2\times 2}\right\},\\
& \mathcal{Q}_3(F) = D^2 W(Id_3)(F,F).
\end{split}
\end{equation}
The form $\mathcal{Q}_3$ is defined for  all $F\in\mathbb{R}^{3\times 3}$, 
while $\mathcal{Q}_2(x', \cdot)$ are defined on $F_{2\times
  2}\in\mathbb{R}^{2\times 2}$. Both forms $\mathcal{Q}_3$ and all $\mathcal{Q}_2$ are nonnegative
definite and depend only on the  symmetric parts of their arguments,
in view of the assumptions on the elastic energy density $W$.
Clearly, the minimization problem in (\ref{Qform}) has a unique
solution among symmetric matrices $\tilde F$ which for each $x'\in
\omega$ is described by the linear function $F_{2\times 2}\mapsto c(x', F_{2\times
  2})\in \mathbb{R}^3$ in:
\begin{equation}\label{cdef}
\mathcal{Q}_2(x', F_{2\times 2}) = \min\Big\{\mathcal{Q}_3\big(\bar
A(x')^{-1} (F_{2\times 2}^* + c\otimes e_3)\bar A(x')^{-1}\big); ~
c\in\mathbb{R}^3\Big\}.
\end{equation}
The energy in the right hand side of (\ref{old_Kirch}) is a Kirchhoff-like fully nonlinear
bending, which in case of $\bar Ae_3 = e_3$ reduces to the classical
bending content relative to the second fundamental form $(\nabla y)^{\trsp}\nabla
\vec b= (\nabla y)^{\trsp}\nabla \vec \nu$ on the deformed surface $y(\omega)$.

\medskip

In the present setting, we start with an observation about projections on
polynomial subspaces of $L^2$. Consider the following Hilbert space, with its norm:
\begin{equation}\label{poly_norm} 
\mathbb{E}\doteq \big(L^2(\Omega,
\mathbb{R}^{2\times 2}_{\sym}), \|\cdot\|_{\mathcal{Q}_2}\big), \qquad
\|F\|_{\mathcal{Q}_2} = \Big(\int_{\Omega}\mathcal{Q}_2(x', F(x))
\,\rmd x\Big)^{1/2},
\end{equation}
associated to the scalar product (with obvious notation):
$$\langle F_1, F_2\rangle_{\mathcal{Q}_2} =
\int_{\Omega}\mathcal{L}_{2,x'} \big(F_1(x), F_2(x)\big) \,\rmd x.$$ 
We define $\mathbb{P}_1$ and $\mathbb{P}_2$, respectively,
as the orthogonal projections onto the following subspaces of $\mathbb{E}$:
\begin{equation}\label{poly_space}
\begin{split}
\mathbb{E}_1 & = \Big\{x_3\mathcal{F}_1(x') +\mathcal{F}_0(x'); ~
\mathcal{F}_1, \mathcal{F}_0\in L^2(\omega, \mathbb{R}^{2\times 2}_{\sym})\Big\},\\
\mathbb{E}_2 & = \Big\{x_3^2 \mathcal{F}_2(x') + x_3\mathcal{F}_1(x') +\mathcal{F}_0(x'); ~
\mathcal{F}_2, \mathcal{F}_1, \mathcal{F}_0\in L^2(\omega, \mathbb{R}^{2\times 2}_{\sym})\Big\},
\end{split}
\end{equation}
obtained by projecting each $F(x', \cdot)$ on
the appropriate polynomial subspaces of $L^2(-1/2, 1/2)$ whose
orthonormal bases consist of the Legendre 
polynomials $\{p_i\}_{i=0}^\infty$. The first three polynomials are:
$$p_0(t) = 1,\qquad p_1(t) = \sqrt{12}t, \qquad p_2(t) = \sqrt{5}\big(6t^2 - \frac{1}{2}\big).$$

\begin{lemma}\label{proj}
For every $F\in\mathbb{E}$, we have:
\begin{equation*}
\begin{split}
\mathbb{P}_1(F) & = 12\Big(\int_{-1/2}^{1/2}x_3F \,\rmd x_3\Big)x_3 + 
\Big(\int_{-1/2}^{1/2}F \,\rmd x_3\Big),\\
\mathbb{P}_2(F) &= \Big(\int_{-1/2}^{1/2}(180x_3^2-15)F \,\rmd x_3 \Big)x_3^2 +
  12\Big(\int_{-1/2}^{1/2}x_3F \,\rmd x_3\Big)x_3 + 
\Big(\int_{-1/2}^{1/2}(-15x_3^2 + \frac{9}{4})F \,\rmd x_3 \Big)
\end{split}
\end{equation*}
Moreover, the distances from spaces $\mathbb{E}_1$ and $\mathbb{E}_2$ are given by: 
\begin{equation*}
\begin{split}
\dist^2_{\mathcal{Q}_2}(F, \mathbb{E}_1) & =
\int_\omega\bigg(\int_{-1/2}^{1/2}\mathcal{Q}_2\big(x', F\big)
\,\rmd x_3 - 12 \mathcal{Q}_2\big(x', \int_{-1/2}^{1/2}x_3 F \,\rmd
x_3 \big) - \mathcal{Q}_2\big(x', \int_{-1/2}^{1/2}F\,\rmd x_3 \big) \bigg)\,\rmd x',\\
\dist^2_{\mathcal{Q}_2}(F, \mathbb{E}_2) & = \int_\omega\bigg(\int_{-1/2}^{1/2}\mathcal{Q}_2\big(x', F\big)
\,\rmd x_3 - 180 \mathcal{Q}_2\big(x', \int_{-1/2}^{1/2}\big(x_3^2-\frac{1}{12}\big) F \,\rmd
x_3 \big) \\ & \qquad \qquad  \qquad \quad
- 12 \mathcal{Q}_2\big(x', \int_{-1/2}^{1/2}x_3 F \,\rmd x_3 \big)
- \mathcal{Q}_2\big(x', \int_{-1/2}^{1/2}F \,\rmd x_3 \big) \bigg)\,\rmd x'.
\end{split}
\end{equation*}
\end{lemma}
\begin{proof}
The Lemma results by a straightforward calculation:
\begin{equation*}
\begin{split}
\dist^2_{\mathcal{Q}_2}(F, \mathbb{E}_1) & = \|F\|_{\mathcal{Q}_2}^2 - \|\mathbb{P}_1(F)\|_{\mathcal{Q}_2}^2 =
\|F\|_{\mathcal{Q}_2}^2 - \Big(\big\|\int_{-1/2}^{1/2}p_1 F \,\rmd x_3 \big\|_{\mathcal{Q}_2}^2 
+ \big\|\int_{-1/2}^{1/2}p_0 F \,\rmd x_3 \big\|_{\mathcal{Q}_2}^2 \Big),\\
\dist^2_{\mathcal{Q}_2}(F, \mathbb{E}_2) & = \|F\|_{\mathcal{Q}_2}^2 -
\|\mathbb{P}_2(F)\|_{\mathcal{Q}_2}^2 \\ & =
\|F\|_{\mathcal{Q}_2}^2 - \Big(\big\|\int_{-1/2}^{1/2}p_2 F \,\rmd
x_3 \big\|_{\mathcal{Q}_2}^2 + \big\|\int_{-1/2}^{1/2}p_1 F \,\rmd x_3 \big\|_{\mathcal{Q}_2}^2 
+ \big\|\int_{-1/2}^{1/2}p_0 F \,\rmd x_3 \big\|_{\mathcal{Q}_2}^2 \Big),
\end{split}
\end{equation*}
where we have used that: $ \mathbb{P}_1(F)  = p_1\int_{-1/2}^{1/2}p_1F
\,\rmd x_3 + p_0 \int_{-1/2}^{1/2}p_0 F \,\rmd x_3$ and similarly:
$\mathbb{P}_2(F) = p_2\int_{-1/2}^{1/2}p_2F \,\rmd x_3 +
p_1 \int_{-1/2}^{1/2}p_1F \,\rmd x_3 + p_0\int_{-1/2}^{1/2}p_0 F \,\rmd x_3$.
\end{proof}

\medskip

\begin{thm}\label{Gliminf}
In the setting of Theorem \ref{compactness_thm}, 
$\liminf_{h\to 0} \frac{1}{h^2}\mathcal{E}^h(u^h)$ is bounded from below by:
\begin{equation*}
\begin{split}
 \mathcal I_2^O(y) & =\frac{1}{2}\int_{\Omega}\mathcal
Q_2\Big(x',x_3\nabla y(x')^{\trsp}\nabla \vec b(x') -\frac{1}{2}\mathcal{G}_1(x)_{2\times 2}\Big)\,\rmd x \\
& = \frac{1}{24}\int_{\omega} \mathcal Q_2\Big(x', \big(\nabla y(x')^{\trsp}\nabla
\vec b(x')\big)_{\sym}-\frac{1}{2}\bar{\mathcal{G}}_1(x')_{2\times
  2}\Big)\,\rmd x'+ \frac{1}{8}\dist^2_{\mathcal{Q}_2}\Big((\mathcal{G}_1)_{2\times 2}, \mathbb{E}_1\Big),
\end{split}
\end{equation*}
where $\bar{\mathcal{G}}_1$ is as in (\ref{EO}).
In the non-oscillatory case (\ref{NO}) this formula becomes:
\begin{equation*}
\begin{split}
 \mathcal I_2(y) = \frac{1}{24}\int_{\omega} \mathcal Q_2\Big(x', \big(\nabla y(x')^{\trsp}\nabla
\vec b(x')\big)_{\sym}-\frac{1}{2}\partial_3 G(x',0)_{2\times 2}\Big)\,\rmd x'.
\end{split}
\end{equation*}
The first term in $\mathcal{I}_2^O$ coincides with $\mathcal{I}_2$
for the effective metric $\bar{G}$ in (\ref{EO}).
\end{thm}

\begin{proof}
The argument follows the proof of \cite[Theorem 2.1]{BLS} and thus we only
indicate its new ingredients. Applying the compactness analysis for
the $x_3$-independent metric $\bar G$, one obtains the sequence
$\{R^h\in L^2(\omega, SO(3))\}_{h\to 0}$ of approximating rotation-valued fields, satisfying:
\begin{equation}\label{Rh}
\frac{1}{h} \int_{\Omega^h} |\nabla u^h(x) \bar A(x')^{-1} -
R^h(x')|^2 \,\rmd x \leq Ch^2.
\end{equation}
Define now the family $\{S^h\in L^2(\Omega, \mathbb{R}^{3\times
  3})\}_{h\to 0}$ by:
$$S^h(x', x_3) = \frac{1}{h}\Big( R^h(x')^{\trsp}\nabla u^h(x', hx_3)
A^h(x', hx_3)^{-1}- Id_3\Big). $$
According to \cite{BLS}, the same quantities, written for the metric
$\bar G$ rather than $G^h$:
$$\bar S^h(x', x_3) = \frac{1}{h}\Big( R^h(x')^{\trsp}\nabla u^h(x', hx_3)
\bar A(x')^{-1}- Id_3\Big), $$
converge weakly in $L^2(\Omega, \mathbb{R}^{3\times 3})$ to $\bar S$, such that:
\begin{equation}\label{22}
\big(\bar A(x') \bar S(x', x_3) \bar A(x')\big)_{2\times 2} = \bar
s(x') + x_3 \nabla y(x')^{\trsp}\nabla\vec b(x'),
\end{equation}
with some appropriate $\bar s\in L^2(\omega, \mathbb{R}^{2\times 2})$. Observe that:
$$S^h(x', x_3) = \bar S^h(x', x_3) + \frac{1}{h}R^h(x')^{\trsp}\nabla ^hu(x', hx_3)
\big(A^h(x', hx_3)^{-1} - \bar A(x')^{-1}\big)$$
and that the term $R^h(x')^{\trsp}\nabla u^h(x', hx_3)$ converges
strongly in $L^2(\Omega,\mathbb{R}^{3\times 3})$ to $\bar A(x')$.
On the other hand, the remaining factor converges uniformly on $\Omega$ as $h\to 0$, because:
\begin{equation}\label{noga}
\frac{1}{h}\big(A^h(x', hx_3)^{-1} - \bar A(x')^{-1}\big) =
\displaystyle{-\bar A(x')^{-1} A_1(x',x_3)\bar A(x')^{-1} + \mathcal{O}(h) } 
\end{equation}
Concluding, $S^h$ converge weakly in $L^2(\Omega, \mathbb{R}^{3\times
  3})$ to $ S$, satisfying by (\ref{22}):
\begin{equation}\label{29}
\big(\bar A(x') S(x', x_3) \bar A(x')\big)_{2\times 2} 
= \bar s(x') + x_3 \nabla y(x')^{\trsp}\nabla \vec b(x') - \big(\bar
A(x')A_1(x',x_3)\big)_{2\times 2}.
\end{equation}

\smallskip

Consequently, using the definition of $S^h$ and frame invariance of $W$
and Taylor expanding $W$ at $Id_3$ on the set $\{|S^h|^2\leq 1/h\}$, we obtain:
\begin{equation*}
\begin{split}
\liminf_{h\to 0} \mathcal{E}^h(u^h) & = \liminf_{h\to 0}
\frac{1}{h^2}\int_{\Omega} W\big(Id_3 + hS^h(x)\big) \,\rmd x \\ & \geq \liminf_{h\to 0}
\frac{1}{2}\int_{\{|S^h|^2\leq 1/h\}}\mathcal{Q}_3(S^h(x)) +
o(|S^h|^2)   \,\rmd x \\ & \geq 
\frac{1}{2}\int_{\Omega}\mathcal{Q}_3\big(S(x)\big) \,\rmd x
\geq \frac{1}{2}\int_{\Omega}\mathcal{Q}_2\Big(x', \big(\bar
A(x')S(x)\bar A(x')\big)_{2\times 2}\Big) \,\rmd x. 
\end{split}
\end{equation*}
Further, recalling (\ref{29}) and (\ref{Adef}) we get:
\begin{equation*}
\begin{split}
\liminf_{h\to 0} \mathcal{E}^h&(u^h)  \geq  \frac{1}{2}\big\|\bar
s_{\sym} + x_3\big((\nabla y)^{\trsp}\nabla \vec b\big)_{\sym}
- \frac{1}{2}(\mathcal{G}_1)_{2\times 2} \big\|_{\mathcal{Q}_2}^2 \\ &
= \frac{1}{2}\big\|\mathbb{P}_1\big(\bar
s_{\sym} + x_3\big((\nabla y)^{\trsp}\nabla \vec b\big)_{\sym}
- \frac{1}{2}(\mathcal{G}_1)_{2\times 2} \big)\big\|_{\mathcal{Q}_2}^2
+ \frac{1}{8}\big\|(\mathcal{G}_1)_{2\times 2} -
\mathbb{P}_1((\mathcal{G}_1)_{2\times 2})\big\|_{\mathcal{Q}_2}^2  
\\ & = \frac{1}{2}\big\|\bar s_{\sym} \big\|_{\mathcal{Q}_2}^2
+ \frac{1}{2}\big\|x_3\Big(\big((\nabla y)^{\trsp}\nabla \vec b\big)_{\sym}
-  6 \int_{-1/2}^{1/2}t(\mathcal{G}_1)_{2\times 2} \,\rmd t\Big) \big\|_{\mathcal{Q}_2}^2  
+  \frac{1}{8}\dist^2_{\mathcal{Q}_2}\big((\mathcal{G}_1)_{2\times 2}, \mathbb{E}_1\big) \\ & \geq 
\frac{1}{24} \big\|\big((\nabla y)^{\trsp}\nabla \vec b\big)_{\sym}
-  6 \int_{-1/2}^{1/2}t(\mathcal{G}_1)_{2\times 2} \,\rmd t \big\|_{\mathcal{Q}_2}^2  
+  \frac{1}{8}\dist^2_{\mathcal{Q}_2}\big((\mathcal{G}_1)_{2\times 2}, \mathbb{E}_1\big)
=\mathcal{I}_2^{O}(y),
\end{split}
\end{equation*}
where we have used the fact that $\mathcal{Q}_2(x', \cdot)$ is a
function of its symmetrized argument and Lemma \ref{proj}. The formula for
$\mathcal{I}_2$ in case (\ref{NO}) is immediate.
\end{proof}

Our next result is the upper bound, parallel to the lower bound in Theorem \ref{Gliminf}:

\begin{thm} \label{Glimsup}
Assume (\ref{O}). For every isometric immersion $y\in
W^{2,2}(\omega,\mathbb{R}^3)$ of the reduced midplate metric $\bar{\mathcal 
{G}}_{2\times 2}$ as in (\ref{iso}), there exists a sequence $\{u^h\in
W^{1,2}(\Omega^h, \R^3)\}_{h\to 0}$ such that the sequence $\{y^h(x', x_3) =
u^h(x', hx_3)\}_{h\to 0}$ converges in $\WW^{1,2}(\Omega, \mathbb{R}^3)$ to $y$ and: 
\begin{equation}\label{up_bd2} 
\lim_{h\to 0} \frac{1}{h^2}\mathcal E^h(u^h) = \displaystyle{\mathcal{I}_2^O(y) } 
\end{equation}
Automatically, $\frac{1}{h}\partial_3y^h$ converges in
$L^2(\Omega,\mathbb{R}^3)$ to $\vec b\in W^{1,2}\cap L^\infty(\omega,
\mathbb{R}^3)$ as in (\ref{cosserat}). 
\end{thm}
\begin{proof} 
Given an admissible $y$, we define $\vec b$ by (\ref{cosserat}) and also
define the matrix field:
\begin{equation}\label{Qyb}
Q=\big[ \partial_1y,~ \partial_2 y,~ \vec b \big]\in
W^{1,2}\cap L^\infty(\omega, \mathbb{R}^{3\times 3}).
\end{equation}
It follows that $Q(x') \bar A(x')^{-1}\in SO(3)$ on $\omega$. 
The recovery sequence $u^h$ satisfying (\ref{up_bd2}) is then constructed via a diagonal
argument, applied to the explicit deformation fields below. Again,
we only indicate the new ingredients with respect to the proof in \cite[Theorem 3.1]{BLS}.

\smallskip

Recalling the notion of the linear vector association
$c(x',F_{2\times 2})$ that is minimizing in the right hand side of the
formula (\ref{cdef}),  we define the vector field $\vec d\in L^2(\Omega, \mathbb{R}^3)$ by:
\begin{equation}\label{d-o}
\begin{split}
\vec d(x', x_3) = Q(x')^{{\trsp}, -1}&\bigg( \frac{x_3^2}{2}\Big(c\big(x', \nabla y(x')^{\trsp}\nabla \vec
b(x')\big) - \frac{1}{2} \left[ \begin{array}{c}\nabla |\vec
    b|^2(x')\\ 0\end{array}\right]\Big) \\ & \quad - \frac{1}{2}
c\Big(x', \int_0^{x_3}\mathcal{G}_1(x', t)_{2\times 2}\,\rmd t\Big)
\\ & \quad  + \int_0^{x_3}\mathcal{G}_1(x', t)\,\rmd t \, e_3 -\frac{1}{2}
\int_0^{x_3} \mathcal{G}_1(x', t)_{33}\,\rmd t \, e_3 \bigg).
\end{split}
\end{equation}
In view of (\ref{cdef}), the above definition is
equivalent to the vector field $\partial_3\vec d\in L^2(\Omega, \mathbb{R}^3)$ being,
for each $(x', x_3)\in \Omega$, the unique solution to:
\begin{equation*} 
\begin{split}
\mathcal{Q}_2&\Big(x',x_3\nabla y(x')^{\trsp}\nabla \vec b(x') -
\frac{1}{2}\mathcal{G}_1(x', x_3)_{2\times 2}\Big)  \\ & \qquad =
\mathcal{Q}_3\bigg(\bar A(x')^{-1}\Big(Q(x')^{\trsp} \left[x_3\partial_1\vec b(x'),
  ~ x_3\partial_2\vec b(x'), ~ \partial_3\vec d(x', x_3)\right] -
\frac{1}{2}\mathcal{G}_1(x', x_3) \Big)\bar A(x')^{-1}\bigg).
\end{split}
\end{equation*}
One then approximates $y, \vec b$ by sequences $\{y^h\in
W^{2, \infty}(\omega, \mathbb{R}^3)\}_{h\to 0}$, $\{b^h\in
W^{1, \infty}(\omega, \mathbb{R}^3)\}_{h\to 0}$ respectively, and request them to satisfy conditions exactly
as in the proof of \cite[Theorem 3.1]{BLS}. The warping field $\vec d$
is approximated by $d^h(x', x_3)=\int_0^{x_3} \bar d^h(x', t)\,\rmd t$, where:
$$\bar d^h\to \bar d = \partial_3 \vec d \quad \mbox{strongly in }
L^2(\Omega, \mathbb{R}^3) \quad \mbox{ and } \quad h\|\bar
d^h\|_{W^{1,\infty}(\Omega, \mathbb{R}^3)}\to 0 \quad  \mbox{ as } h\to 0.$$
Finally, we define:
\begin{equation}\label{rs-o}
u^h(x', x_3) = y^h(x')+x_3b^h(x') + h^2 d^h\big(x', \frac{x_3}{h}\big),
\end{equation}
so that, with the right approximation error, there holds:
$$\nabla u^h(x', x_3) \approx Q(x') + h \bigg[\frac{x_3}{h}\partial_1\vec
  b(x'),~ \frac{x_3}{h}\partial_2\vec b(x'),~ \partial_3\vec d\big(x', \frac{x_3}{h}\big)\bigg]. $$
Using Taylor's expansion of $W$, the definition (\ref{d-o}) and the
controlled blow-up rates of the approximating sequences, the construction is done.
\end{proof}

\medskip

We conclude this section by noting the following easy direct consequence of Theorems \ref{Gliminf}
and \ref{Glimsup}:

\begin{cor}\label{limit_scaled}
If the set of $W^{2,2}(\omega,\mathbb{R}^3)$ isometric immersions of $\bar{\mathcal{G}}_{2\times 2}$ is
nonempty, then the functional $\mathcal{I}_2^O$ attains its infimum and:
$$\lim_{h\to 0} \frac{1}{h^2}\inf\mathcal{E}^h = \min \;\mathcal{I}_2^0. $$
The infima  in the left hand side are taken over 
$W^{1,2}(\Omega^h, \mathbb{R}^3)$ deformations $u^h$, whereas the minima in
the right hand side are taken over $W^{2,2}(\omega,\mathbb{R}^3)$
isometric immersions $y$ of $\bar{\mathcal{G}}_{2\times 2}$.
\end{cor}

\section{Identification of the $Ch^2$ scaling regime}\label{conditions_flateness} 

In this section, we identify the equivalent conditions for $\inf
\mathcal{E}^h \sim h^2$ in terms of curvatures of the metric tensor
$\bar G$ in case (\ref{NO}). We begin by expressing the integrand
tensor in the residual energy $\mathcal{I}_2$
in terms of the shape operator on the deformed midplate. Recall that
we always use the Einstein summation convention over repeated indices
running from $1$ to $3$.

\begin{lemma}\label{lem_term}
In the the non-oscillatory setting (\ref{NO}), let $y\in
W^{2,2}(\omega,\mathbb{R}^3)$ be an isometric immersion of the metric
$\bar{\mathcal{G}}_{2\times 2}$, so that (\ref{iso}) holds on $\omega$. Define
the Cosserat vector $\vec b$ according to (\ref{cosserat}). Then:
\begin{equation}\label{tensor1}
\big((\nabla y)^{\trsp}\nabla \vec b\big)_{\sym} -
\frac{1}{2}\partial_3G(x',0)_{2\times 2} =
\frac{1}{\sqrt{\bar{\mathcal{G}}^{33}}} \Pi_y
+\frac{1}{\bar{\mathcal{G}}^{33}}\left[\begin{array}{cc} \Gamma_{11}^3 &
    \Gamma_{12}^3\\ \Gamma_{12}^3 & \Gamma_{22}^3\end{array}\right](x',0),
\end{equation}
for all $x' \in \omega$. Above, $\bar{\mathcal{G}}^{33}= \langle \bar{\mathcal{G}}^{-1}e_3, e_3\rangle$,
whereas $\Pi_y = (\nabla y)^{\trsp}\nabla \vec \nu\in W^{1,2}(\omega,
\mathbb{R}^{2\times 2}_{\sym})$ is the second
fundamental form of the surface $y(\omega)\subset \mathbb{R}^3$, and
$\{\Gamma^i_{kl}\}_{i,k,l=1\ldots 3}$ are the Christoffel symbols of $G$:
$$\Gamma^i_{kl} = \frac{1}{2} G^{im}\big( \partial_l
G_{mk} + \partial_k G_{ml} - \partial_m G_{kl}\big).$$
\end{lemma}
\begin{proof}
The proof is an extension of the arguments in \cite[Theorem 5.3]{BLS},
which we modify for the case of $x_3$-dependent metric $G$. Firstly, the
fact that $Q^{\trsp}Q=\bar{\mathcal{G}}$ with $Q$ defined in (\ref{Qyb}), yields:
\begin{equation}\label{y0}
\big((\nabla y)^{\trsp}\nabla \vec b\big)_{\sym} = \Big(
\big[\partial_i\bar{\mathcal{G}}_{j3}\big]_{i,j=1, 2} \Big)_{\sym} -
\big[\langle \partial_{ij}y, \vec b\rangle\big]_{i,j=1, 2}. 
\end{equation}
Also, $\partial_i\bar{\mathcal{G}} = 2\big((\partial_i Q)^{\trsp} Q\big)_{\sym}$
for $i=1,2$, results in: 
\begin{equation}\label{y2}
\langle \partial_{ij}y, \partial_ky\rangle = \frac{1}{2}\big( \partial_iG_{kj} + \partial_j
G_{ik}-\partial_kG_{ij}\big)
\end{equation}
and:
$$(\nabla y)^{\trsp}\partial_{ij}y = \Gamma_{ij}^m(x',0) \left[\begin{array}{c}
\bar{\mathcal{G}}_{m1}\\ \bar{\mathcal{G}}_{m2}\end{array}\right] \qquad\mbox{ for }\;  i,j=1,2.$$
Consequently, we obtain the formula:
\begin{equation*}
\begin{split}
\left[ \bar{\mathcal{G}}_{13}, ~ \bar{\mathcal{G}}_{23}\right] (\bar{\mathcal{G}}_{2\times 2})^{-1}
(\nabla y)^{\trsp}\partial_{ij}y & = \Bigg[ \bar{\mathcal{G}}_{13}, ~ \bar{\mathcal{G}}_{23},
  ~ \left[\bar{\mathcal{G}}_{13}, ~ \bar{\mathcal{G}}_{23}\right] (\bar{\mathcal{G}}_{2\times
    2})^{-1}\left[\begin{array}{c} \bar{\mathcal{G}}_{13}\\ \bar
      {\mathcal{G}}_{23}\end{array}\right]\Bigg] \left[\begin{array}{c}
    {\Gamma}_{ij}^1\\  {\Gamma}_{ij}^2 \\  {\Gamma}_{ij}^3\end{array}\right](x',0)
\\ & = \bar{\mathcal{G}}_{m3}\Gamma_{ij}^m(x',0) -
\frac{1}{\bar{\mathcal{G}}^{33}}\Gamma_{ij}^3(x',0). 
\end{split}
\end{equation*}
Computing the normal vector $\vec \nu$ from (\ref{cosserat}) and noting that $\det
  \bar{\mathcal{G}}_{2\times 2}/ \det\bar{\mathcal{G}} = \bar{\mathcal{G}}^{33}$, we get:
\begin{equation*}
\begin{split}\Pi_{ij} & = -\langle \partial_{ij}y, \vec \nu\rangle = -\sqrt{\bar
  {\mathcal{G}}^{33}} \Big(\langle \partial_{ij}y, \vec b\rangle -
\left[ \bar{\mathcal{G}}_{13}, ~ \bar{\mathcal{G}}_{23}\right] (\bar{\mathcal{G}}_{2\times 2})^{-1}
(\nabla y)^{\trsp}\partial_{ij} y\Big) \\ & = 
\sqrt{\bar{\mathcal{G}}^{33}}\big((\nabla y)^{\trsp}\nabla \vec b\big)_{\sym, ij}
- \frac{1}{\sqrt{\bar{\mathcal{G}}^{33}}}\Gamma_{ij}^3(x', 0) - \frac{\sqrt{\bar
  {\mathcal{G}}^{33}}}{2}\partial_3 G_{ij}(x',0), \qquad\mbox{ for }\;  i,j=1,2,
\end{split}
\end{equation*}
which completes the proof of (\ref{tensor1}).
\end{proof}

\smallskip

The key result of this section is the following:

\begin{thm}\label{Kirchhoff_optimal}
The energy scaling beyond the Kirchhoff regime:
$$\lim_{h\to 0}\frac{1}{h^2}\inf \mathcal{E}^h = 0$$
is equivalent to the following conditions:
\begin{itemize}
\item[(i)] in the oscillatory case (\ref{O})
\begin{equation}\label{I-IIO}
\left[~\mbox{\begin{minipage}{15cm} \vspace{1mm}
$(\mathcal{G}_1)_{2\times 2}\in\mathbb{E}_1$ or equivalently there holds:
\begin{equation*}
\mathcal{G}_1(x', x_3)_{2\times 2} =
x_3\bar{\mathcal{G}}_1(x')_{2\times 2} \qquad\mbox{ for all }\; (x', x_3)\in\bar\Omega.
\end{equation*} 
Moreover, 
condition  (\ref{I-II}) below must be satisfied with $G$ replaced by the effective
metric $\bar{G}$ in (\ref{EO}). This condition involves only
$\bar{\mathcal{G}}$  and $(\bar{\mathcal{G}}_1)_{2\times 2}$ terms of $\bar{G}$. \vspace{1mm}
\end{minipage}}\right.
\end{equation}

\item[(ii)]  in the non-oscillatory case (\ref{NO})
\begin{equation}\label{I-II}
\left[~\mbox{\begin{minipage}{15cm} There exists $y_0\in W^{2,2}(\omega,
    \mathbb{R}^3)$ satisfying (\ref{iso}) and such that:
$$\Pi_{y_0}(x') = -\frac{1}{\sqrt{\bar{\mathcal{G}}^{33}}}
\left[\begin{array}{cc}\Gamma_{11}^3 & \Gamma_{12}^3\\ \Gamma_{12}^3 &
  \Gamma_{22}^3\end{array}\right] (x',0) \qquad \mbox{for all } x'\in\omega, $$
where $\Pi_{y_0}$ is the second fundamental form of the surface $y_{0}(\omega)$ and
$\{\Gamma_{jk}^i\}$ are the Christoffel symbols of the metric $G$.
\end{minipage}}\right.
\end{equation}
\end{itemize}
The isometric immersion $y_0$ in (\ref{I-II}) is automatically smooth
(up to the boundary) and it is unique up to rigid motions. Further, on
a simply connected midplate $\omega$, condition (\ref{I-II}) is equivalent to:
\begin{equation}\label{I-II2}
\left[~\mbox{\begin{minipage}{15cm} The following Riemann curvatures
      of the metric $G$ vanish on $\omega\times \{0\}$:
$$R_{1212}(x',0) = R_{1213}(x', 0) = R_{1223}(x',0) = 0\qquad
\mbox{for all } x'\in \omega.$$
\end{minipage}}\right.
\end{equation}
Above, the Riemann curvatures of a given metric $G$ are:
$$R_{iklm}= \frac{1}{2}\Big(\partial_{kl}G_{im} + \partial_{im}G_{kl}
- \partial_{km}G_{il} - \partial_{il}G_{km}\Big) + 
G_{np}\big(\Gamma_{kl}^n \Gamma_{im}^p - \Gamma_{km}^n \Gamma_{il}^p\big).$$
\end{thm}
\begin{proof}
By Corollary \ref{limit_scaled}, it suffices to determine the
equivalent conditions for $\min \mathcal{I}_2^O=0$ and $\min
\mathcal{I}_2=0$. In case (\ref{O}), the linearity of $x_3\mapsto \mathcal{G}_1(x', x_3)_{2\times 2}$
is immediate, while condition (\ref{I-II})
follows in both cases (\ref{O}) and (\ref{NO})  by Lemma
\ref{lem_term}. Note that the Christoffel symbols $\{\Gamma_{jk}^i\}$
depend only on $\bar{\mathcal{G}}$ and $\partial_3G(x',0)_{2\times 2}$ in the Taylor expansion
of $G$. This completes the proof of (i) and (ii).

\smallskip

Regularity of $y_0$ is an easy consequence, via the bootstrap argument, of the continuity equation:
\begin{equation}\label{continuity}
\partial_{ij}y_0 = \sum_{m=1}^2\gamma_{ij}^m\partial_my_0 -
(\Pi_{y_0})_{ij}\vec\nu_0 \qquad \mbox{for } i,j=1, 2,
\end{equation}
where $\{\gamma_{ij}^m\}_{i,j,m=1\ldots 2}$ denote the Christoffel
symbols of $\bar{\mathcal{G}}_{2\times 2}$ on $\omega$. Uniqueness of
$y_0$ is a consequence of (\ref{I-II}), due to uniqueness of isometric immersion with prescribed
second fundamental form.

\smallskip

To show (\ref{I-II2}), we argue as in the proof of \cite[Theorem
5.5]{BLS}. The compatibility of $\bar G_{2\times 2}$ and $\Pi_{y_0}$ is equivalent to the satisfaction of the related
Gauss-Codazzi-Mainardi equations. By an explicit calculation, we see
that the two Codazzi-Mainardi equations become:
\begin{equation*}
\begin{split}
\big(\partial_2\Gamma_{11}^3 - \partial_1 &\Gamma_{12}^3\big) - \frac{1}{2}
\left(\frac{\partial_2G^{33}}{G^{33}} \Gamma_{11}^3
- \frac{\partial_1G^{33}}{G^{33}} \Gamma_{12}^3\right) 
+ \frac{1}{G^{33}}G^{m3}\big(\Gamma_{2m}^3\Gamma_{11}^3-\Gamma_{1m}^3\Gamma_{12}^3\big) \\ & 
\qquad\qquad = \left( \sum_{m=1}^2\Gamma^3_{1m}\Gamma^m_{12} -
\sum_{m=1}^2\Gamma^3_{2m}\Gamma^m_{11} \right) +
\frac{G^{32}}{G^{33}}(\Gamma^3_{11}\Gamma^3_{22} -
(\Gamma_{12}^3)^2),\\
\big(\partial_2\Gamma_{12}^3 - \partial_1 &\Gamma_{22}^3\big) - \frac{1}{2}
\left(\frac{\partial_2G^{33}}{G^{33}} \Gamma_{12}^3
- \frac{\partial_1G^{33}}{G^{33}} \Gamma_{22}^3\right) 
+ \frac{1}{G^{33}} G^{m3}\big(\Gamma_{2m}^3\Gamma_{12}^3-\Gamma_{1m}^3\Gamma_{22}^3\big)\\ & 
\qquad\qquad = \left( \sum_{m=1}^2\Gamma^3_{1m}\Gamma^m_{22} -
\sum_{m=1}^2\Gamma^3_{2m}\Gamma^m_{12} \right) -
\frac{G^{31}}{G^{33}}(\Gamma^3_{11}\Gamma^3_{22} - (\Gamma_{12}^3)^2),
\end{split}
\end{equation*}
and are equivalent to $R_{121}^3=R_{221}^3=0$ on $\omega\times\{0\}$.
The Gauss equation is, in turn, equivalent to $R_{1212}=0$ exactly as
in \cite{BLS}. The simultaneous vanishing of $R_{121}^3, R_{221}^3,
R_{1212}$ is equivalent with the vanishing of $R_{1212}, R_{1213}$ and
$R_{1223}$, which proves the claim in (\ref{I-II2}).
\end{proof}

\section{Coercivity of the limiting energy $\mathcal{I}_2$}\label{sec_coer}

In this section we quantify the statement in Theorem
\ref{Kirchhoff_optimal} and prove that when either of $\mathcal{I}_2$
or $\mathcal{I}_2^O$ can be minimized to 
zero, the effective energy $\mathcal{I}_2(y)$ measures then the distance of
a given isometric immersion $y$ from the kernel:
$\mbox{ker}\;\mathcal{I}_2 = \big\{Qy_0 + d; ~ Q\in SO(3),~ d\in\mathbb{R}^3\big\}$.

\smallskip

Assume that the set of $W^{2,2}(\omega,\mathbb{R}^3)$ isometric
immersions $y$ of $\bar{\mathcal{G}}_{2\times 2}$ is nonempty, which in view of
Theorems \ref{Gliminf} and \ref{Glimsup} is equivalent to:
$\inf\mathcal{E}^h\leq Ch^2$. For each such $y$, the continuity equation
(\ref{continuity}) combined with Lemma \ref{lem_term} gives the
following formula, valid for all $i,j=1, 2$:
\begin{equation}\label{zero}
\partial_{ij}y = \sum_{m=1, 2}\gamma_{ij}^m\partial_m y -
\sqrt{\bar{\mathcal{G}}^{33}}\Big(\big((\nabla y)^{\trsp}\nabla\vec b\big)_{\sym}
- \frac{1}{2}\partial_3G(x',0)_{2\times 2}\Big)_{ij}\vec \nu +
\frac{\Gamma_{ij}^3}{\sqrt{\bar{\mathcal{G}}^{33}}}\vec \nu \quad \mbox{on } \omega.
\end{equation}
Another consequence of (\ref{continuity}) is:
$$|\nabla^2 y|^2 = |\Pi_y|^2 + \sum_{i,j=1, 2} \big\langle \bar
{\mathcal{G}}_{2\times 2} : [\gamma_{ij}^1,~\gamma_{ij}^2]^{\otimes 2}\big\rangle \quad \mbox{on } \omega. $$
By Lemma \ref{lem_term} and since $|\nabla y|^2 =
\mathrm{trace} \;\bar{\mathcal{G}}_{2\times 2}$, this yields the bound:
\begin{equation}\label{jeden}
\big\|y-\fint_{\omega} y\big\|^2_{W^{2,2}(\omega,\mathbb{R}^3)}\leq C \big(\mathcal{I}_2(y) + 1\big),
\end{equation}
where $C$ is a constant independent of $y$.  Clearly, when condition
(\ref{I-II2}) does not hold, so that $\min \mathcal{I}_2>0$, the right
hand side $C \big(\mathcal{I}_2(y) + 1\big) $ above may be replaced by $C\mathcal{I}_2(y)$. On the other
hand, in presence of (\ref{I-II2}), the  bound (\ref{jeden}) can be
refined to the following coercivity result:

\begin{thm}\label{coercive2}
Assume the curvature condition (\ref{I-II2}) on a metric $G$ as in
(\ref{NO}), and let $y_0$ be the unique (up to rigid motions
in $\mathbb{R}^3$) isometric immersion
of $\bar{\mathcal{G}}_{2\times 2}$ satisfying (\ref{I-II}). Then, for all $y\in W^{2,2}(\omega,
\mathbb{R}^3)$ such that $(\nabla y)^{\trsp}\nabla y = \bar{\mathcal{G}}_{2\times 2}$, there holds:
\begin{equation}\label{coer}
\dist^2_{W^{2,2}(\omega, \mathbb{R}^3)} \Big(y, ~ \big\{Ry_0 + c; ~
R\in SO(3),~ c\in\mathbb{R}^3\big\}\Big) \leq C \mathcal{I}_2(y),
\end{equation}
with a constant $C>0$ that depends on $G,\omega$ and $W$ but is independent of $y$.
\end{thm}
\begin{proof}
Without loss of generality, we set $\fint_\omega y =\fint_\omega y_0=0$. 
For any $R\in SO(3)$, identity (\ref{zero}) implies:
\begin{equation*}
\begin{split}
\int_\omega\big|\nabla^2y - \nabla^2(R y_0)\big|^2 \,\rmd x' & \leq
C\Big(\int_\omega \big|\nabla y - \nabla (Ry_0)\big|^2 \,\rmd x' \\ &
\qquad\quad  + \int_\omega \big|\big((\nabla y)^{\trsp}\nabla \vec
b\big)_{\sym}-\frac{1}{2}\partial_3G(x',0)_{2\times 2}\big|^2 \,\rmd x' +
\int_\omega |\vec \nu - R\vec\nu_0|^2 \,\rmd x'\Big) \\ & \leq
C \Big( \int_\omega \big|\nabla y - \nabla (Ry_0)\big|^2 \,\rmd x'  +
\int_\omega \big|\big((\nabla y)^{\trsp}\nabla \vec 
b\big)_{\sym}-\frac{1}{2}\partial_3G(x',0)_{2\times 2}\big|^2 \,\rmd x'\Big),
\end{split}
\end{equation*}
where we used $\mathcal{I}_2(Ry_0) = 0$ and the fact that $\int_\omega
|\vec \nu - R\vec\nu_0|^2 \,\rmd x'\leq C\int_\omega \big|\nabla y -
\nabla (Ry_0)\big|^2 \,\rmd x'$ following, in particular, from
$|\partial_1y\times \partial_2y| = |\partial_1(Ry_0)\times \partial_2(Ry_0)| =
\sqrt{\det\bar{\mathcal{G}}_{2\times 2}}$. 
Also, the non-degeneracy of quadratic forms $\mathcal{Q}_2(x', \cdot)$
in (\ref{Qform}), implies the uniform bound:
$$\int_\omega \big|\big((\nabla y)^{\trsp}\nabla \vec
b\big)_{\sym}-\frac{1}{2}\partial_3G(x',0)_{2\times 2}\big|^2 \,\rmd x'\leq C\mathcal{I}_2(y).$$
Taking $R\in SO(3)$ as in Lemma \ref{weak_coercive}
below, (\ref{coer}) directly follows in view of  (\ref{weak_bd}).
\end{proof}

The next weak coercivity estimate has been the essential part of Theorem \ref{coercive2}:

\begin{lemma}\label{weak_coercive}
Let $y$ and $y_0$ be as in Theorem \ref{coercive2}. Then there exists $R\in SO(3)$ such that:
\begin{equation}\label{weak_bd}
\int_\omega|\nabla y -R\nabla y_0|^2 \,\rmd x'\leq C \int_\omega \big|\big((\nabla y)^{\trsp}\nabla \vec
b\big)_{\sym}-\frac{1}{2}\partial_3G(x',0)_{2\times 2}\big|^2 \,\rmd x',
\end{equation}
with a constant $C>0$ that depends on $G,\omega$ but it is independent of $y$.
\end{lemma}
\begin{proof}
Consider the natural extensions $u$ and $u_0$ of $y$ and $y_0$, namely:
$$u(x', x_3) = y(x') + x_3\vec b(x'), \quad u_0(x', x_3) = y_0(x') +
x_3\vec b_0(x') \qquad \mbox{for all } (x', x_3)\in \Omega^h.$$
Clearly, $u\in W^{1,2}(\Omega^h, \mathbb{R}^3)$ and
$u_0\in\mathcal{C}^1(\bar\Omega^h, \mathbb{R}^3)$ satisfies
$\det\nabla u_0>0$ for $h$ sufficiently small. Write: 
$$\omega=\bigcup_{k=1}^N \omega_k,\qquad \Omega^h=\bigcup_{k=1}^N \Omega^h_k$$
as the union of $N\geq 1$ open, bounded, connected domains
with Lipschitz boundary, such that on each $\{\Omega_k^h=\omega_k\times
(-\frac{h}{2}, \frac{h}{2})\}_{k=1}^N$, the deformation ${u_0}_{\mid {\Omega_k^h}}$ is a
$\mathcal{C}^1$ diffeomorphism onto its image $\mathcal{U}_k^h\subset\mathbb{R}^3$.

\smallskip

{\bf 1.} We first prove (\ref{weak_bd}) under the assumption
$N=1$. Call $v=u\circ u_0^{-1}\in W^{1,2}(\mathcal{U}^h,
\mathbb{R}^3)$ and apply the geometric rigidity estimate
\cite{FJM02} for the existence of $R\in SO(3)$ satisfying: 
\begin{equation}\label{dwa}
\int_{\mathcal{U}^h}|\nabla v-R|^2 \,\rmd z \leq
C\int_{\mathcal{U}^h}\dist^2\big(\nabla v, SO(3)\big) \,\rmd z,
\end{equation}
with a constant $C$ depending on a particular choice of $h$ (and
ultimately $k$, when $N>1$), but
independent of $v$. Since $\nabla v(u_0(x)) =
\nabla u(x) \big(\nabla u_0(x)\big)^{-1}$ for all $x\in\Omega^h$, we get:
\begin{equation}\label{trzy}
\begin{split}
\int_{\mathcal{U}^h}|\nabla v-R|^2 \,\rmd z &= \int_{\Omega^h}(\det\nabla u_0)
\big|(\nabla u-R\nabla u_0)(\nabla u_0)^{-1}\big|^2 \,\rmd x \geq
C\int_{\Omega^h}|\nabla u- R\nabla u_0|^2 \,\rmd x \\ & = C\int_{\Omega^h}
\Big|\big[\partial_1 y, ~ \partial_2y,~ \vec b\big] - R\big[\partial_1 y_0,
~ \partial_2y_0,~ \vec b_0\big] \Big|^2 + x_3^2 |\nabla \vec b -
R\nabla \vec b_0|^2 \,\rmd x\\ & \geq Ch \int_{\omega}|\nabla y - R\nabla y_0|^2 \,\rmd x'.
\end{split}
\end{equation}
Likewise, the change of variables in the right hand side of (\ref{dwa}) gives:
\begin{equation}\label{cztery}
\int_{\mathcal{U}^h}\dist^2\big(\nabla v, SO(3)\big) \,\rmd z\leq C
\int_{\Omega^h}\dist^2\big((\nabla u)(\nabla u_0)^{-1}, SO(3)\big) \,\rmd x.
\end{equation}

\smallskip

Since $(\nabla u)^{\trsp}\nabla u(x',0) = (\nabla u_0)^{\trsp}\nabla u_0(x',0) = \bar{\mathcal{G}}(x')$,
by polar decomposition it follows that: $\nabla u(x',0)=Q(x')=\bar R\bar{\mathcal{G}}^{1/2}$ and $\nabla
u_0(x',0)=Q_0(x')=\bar R_0\bar{\mathcal{G}}^{1/2}$ for some $\bar R,
\bar R_0\in SO(3)$. The notation $Q$, $Q_0$ is consistent with that
introduced in (\ref{Qyb}). Observe further:
\begin{equation*}
\begin{split}
\nabla u(x', x_3) & = Q+x_3 \big[\partial_1 \vec b, ~\partial_2\vec b,~ 0\big] 
=\bar R\bar{\mathcal{G}}^{1/2}\Big(Id_3+x_3
\bar{\mathcal{G}}^{-1}Q^{\trsp}\big[\partial_1 \vec b, ~\partial_2\vec b,~ 0\big]\Big) \\ & 
= \bar R\bar{\mathcal{G}}^{1/2}\Big(Id_3+x_3 \bar{\mathcal{G}}^{-1}\Big(\big((\nabla y)^{\trsp}\nabla
    \vec b\big)^* + e_3\otimes \big[\nabla \vec b, ~ 0\big]^{\trsp} \vec b\Big)\Big),
\end{split}
\end{equation*}
and similarly:
$$\nabla u_0(x', x_3) =  \bar R_0\bar{\mathcal{G}}^{1/2}\Big(Id_3+x_3
\bar{\mathcal{G}}^{-1}\Big(\big((\nabla y_0)^{\trsp}\nabla
    \vec b_0\big)^* + e_3\otimes \big[\nabla \vec b_0, ~ 0\big]^{\trsp} \vec b_0 \Big)\Big).$$
Consequently, the integrand in the right hand side of (\ref{cztery}) becomes:
\begin{equation}\label{piec}
\begin{split}
(\nabla& u)(\nabla u_0)^{-1} \\ & = \bar
R\bar{\mathcal{G}}^{1/2}\bigg(
Id_3 + x_3 \bar{\mathcal{G}}^{-1} S \Big(Id_3+x_3\bar{\mathcal{G}}^{-1}\big((\nabla y_0)^{\trsp}\nabla
    \vec b_0)^* + e_3\otimes \big[\nabla \vec b_0, ~0\big]^{\trsp} \vec b_0
    \Big)^{-1}\bigg)\bar{\mathcal{G}}^{-1/2}\bar R_0^{\trsp}, 
\end{split}
\end{equation}
where:
$$S= \Big((\nabla y)^{\trsp}\nabla 
    \vec b - (\nabla y_0)^{\trsp}\nabla \vec b_0\Big)^* + e_3\otimes
    \big[\nabla \vec b, ~ 0\big]^{\trsp} \vec b - e_3\otimes \big[\nabla \vec
    b_0, ~0\big]^{\trsp} \vec b_0 = \Big((\nabla y)^{\trsp}\nabla 
    \vec b - (\nabla y_0)^{\trsp}\nabla \vec b_0\Big)_{\sym}^*.$$
The last equality follows from the easy facts that, for $i,j=1,2$, we have:
\begin{equation*}
\begin{split}
& \langle \partial_i\vec b, \vec b\rangle = \langle \partial_i\vec b_0,
\vec b_0\rangle = \frac{1}{2} \partial_i\bar{\mathcal{G}}_{33}\\
& \langle\partial_i y, \partial_j\vec b\rangle - \langle\partial_j y, \partial_i\vec b\rangle =
\langle\partial_i y_0, \partial_j\vec b_0\rangle - \langle\partial_j y_0, \partial_i\vec b_0\rangle 
=  \partial_j\bar{\mathcal{G}}_{i3} - \partial_i\bar{\mathcal{G}}_{j3}. 
\end{split}
\end{equation*}

\smallskip

Thus, (\ref{cztery}) and (\ref{piec}) imply:
\begin{equation}\label{szesc}
\begin{split}
\int_{\mathcal{U}^h}\dist^2\big(\nabla v, SO(3)\big) \,\rmd z& \leq C
\int_{\Omega^h}\big|(\nabla u)(\nabla u_0)^{-1} - \bar R \bar R_0^{\trsp}\big|^2\,\rmd x
\leq C \int_{\Omega^h}\big|x_3 S(x', x_3)\big|^2\,\rmd x \\ & \leq C
\int_\omega \Big|\big((\nabla y)^{\trsp}\nabla 
    \vec b\big)_{\sym} - \big((\nabla y_0)^{\trsp}\nabla \vec b_0\big)_{\sym}
\Big|^2 \,\rmd x' \\ & = C \int_\omega \Big|\big((\nabla y)^{\trsp}\nabla 
    \vec b\big)_{\sym} - \frac{1}{2}\partial_3G(x',0)_{2\times 2}\Big|^2 \,\rmd x'
\end{split}
\end{equation}
with a constant $C$ that depends on $G, \omega$ and $h$,
but not on $y$. We conclude (\ref{weak_bd}) in view of (\ref{dwa}), (\ref{trzy}) and (\ref{szesc}).

\smallskip

{\bf 2.} To prove  (\ref{weak_bd}) in case $N>1$, let
$k,s:1\ldots N$ be such that 
$\omega_k\cap\omega_s\neq\emptyset$. Define:
$$F = \Big(\int_{\Omega^h_k\cap\Omega_s^h}\det\nabla u_0 \,\rmd
  x\Big)^{-1}\int_{\Omega^h_k\cap\Omega_s^h}(\det\nabla u_0)(\nabla
  u)(\nabla u_0)^{-1} \,\rmd x\in \mathbb{R}^{3\times 3}.$$ 
Denote by $R_k, R_s\in SO(3)$ the corresponding rotations in
(\ref{weak_bd}) on $\omega_k,\omega_s$. For $i\in\{k, s\}$ we have: 
\begin{equation*}
\begin{split}
|F-R_i|^2 & = \Big|\Big(\int_{\Omega^h_k\cap\Omega_s^h}\det\nabla u_0 \,\rmd
  x\Big)^{-1} \int_{\Omega^h_k\cap\Omega_s^h}(\det\nabla
  u_0)\big(\nabla u-R_i\nabla u_0\big)(\nabla u_0)^{-1} \,\rmd
  x\Big|^2 \\ & \leq C \int_{\Omega^h_k\cap\Omega_s^h}|\nabla u- R_i\nabla
  u_0|^2 \,\rmd x\leq C \int_{\Omega^h_i}|\nabla u- R_i\nabla
  u_0|^2 \,\rmd x \\ & \leq \int_{\omega_i} \big|\big((\nabla y)^{\trsp}\nabla \vec b\big)_{\sym}
-\frac{1}{2}\partial_3G(x',0)_{2\times 2}\Big|^2 \,\rmd x',
\end{split}
\end{equation*}
where for the sake of the last bound we applied the intermediate estimate in (\ref{trzy}) to the left
hand side of (\ref{dwa}), as discussed in the previous step. Consequently:
$$|R_k-R_s|^2 \leq C\int_{\omega} \big|\big((\nabla y)^{\trsp}\nabla \vec b\big)_{\sym}
-\frac{1}{2}\partial_3G(x',0)_{2\times 2}\Big|^2 \,\rmd x',$$
and thus:
\begin{equation*}
\begin{split}
\int_{\omega_k}|\nabla y - R_s\nabla y_0|^2 \,\rmd x' & \leq 2\Big(
\int_{\omega_k}|\nabla y - R_k\nabla y_0|^2 \,\rmd x' +
\int_{\omega_k}|R_k-R_s|^2|\nabla y_0|^2 \,\rmd x'\Big) \\ & \leq 
 C\int_{\omega} \big|\big((\nabla y)^{\trsp}\nabla \vec b\big)_{\sym}
-\frac{1}{2}\partial_3G(x',0)_{2\times 2}\Big|^2 \,\rmd x'.
\end{split}
\end{equation*}
This shows that one can take one and the same $R=R_1$ on each
$\{\omega_k\}_{k=1}^N$, at the expense of possibly increasing the
constant $C$ by a controlled factor depending only on $N$. The proof
of (\ref{weak_bd}) is done.
\end{proof}

\begin{remark}
A similar reasoning as in the proof of Lemma \ref{weak_coercive},
yields a quantitative version of the uniqueness of isometric immersion
with a prescribed second fundamental form compatible to the metric by the 
Gauss-Codazzi-Mainardi equations. More precisely, given a smooth
metric $g$ in $\omega\subset\mathbb{R}^2$, for every two isometric
immersions $y_1, y_2\in W^{2,2}(\omega, \mathbb{R}^3)$ of $g$, there holds:
$$\min_{R\in SO(3)}\int_\omega |\nabla y_1 - R\nabla y_2|^2  \,\rmd x'
\leq C\int_\omega |\Pi_{y_1}-\Pi_{y_2}|^2 \,\rmd x',$$
with a constant $C>0$, depending on $g$ and $\omega$ but independent of $y_1$ and $y_2$. 
\end{remark}

\section{Higher order energy scalings}\label{sec5}

In this and the next sections we assume that:  
\begin{equation}\label{vonKa}
\lim_{h\to 0}\frac{1}{h^2}\inf {\mathcal E^h}=0.
\end{equation}
Recall that by Theorem \ref{Kirchhoff_optimal} this condition is
equivalent to the existence of a (automatically smooth and unique up to rigid motions)
vector field $y_0:\bar\omega\to\mathbb{R}^3$ satisfying:
\begin{equation}\label{system}
(\nabla y_0)^{\trsp}\nabla y_0 = \bar{\mathcal{G}}_{2\times 2} \qquad\mbox{ and }\qquad
\big((\nabla y_0)^{\trsp}\nabla \vec b_0\big)_\sym =
\frac{1}{2}(\bar{\mathcal{G}}_1)_{2\times 2} \quad\mbox{ on }\;\omega,
\end{equation}
where in the oscillatory case (\ref{O}) the symmetric $x'$-dependent
matrix $\mathcal{G}_1$ is given in
(\ref{EO}) and there must be $({\mathcal{G}}_1)_{2\times 2}= x_3(\bar{\mathcal{G}}_1)_{2\times 2}$,
whereas in the non-oscillatory  (\ref{NO})  case
$\bar{\mathcal{G}}_1(x')$ is simply $\partial_3G(x', 0)$. 
The (smooth) Cosserat field
$\vec b_0:\bar\omega\to\mathbb{R}^3$ in (\ref{cosserat}) is uniquely
given by requesting that:
$$Q_0 \doteq \left[\partial_1 y_0,~ \partial_2 y_0, ~ \vec b_0\right] \qquad
  \mbox{satisfies: } \qquad Q_0^{\trsp} Q_0 = \bar{\mathcal{G}}, \quad \det Q_0 >0
  \quad \mbox{ on } \;\omega, $$
with notation similar to (\ref{Qyb}). We now introduce the new vector field $\vec
d_0:\bar\Omega\to\mathbb{R}^3$ through:
\begin{equation}\label{d0}
Q_0^{\trsp} \left[ x_3\partial_1\vec b_0(x'), ~ x_3\partial_2\vec b_0(x'),
  ~ \partial_3\vec d_0(x', x_3)\right] - \frac{1}{2}{\mathcal{G}}_1(x', x_3) \in so(3),
\end{equation}
justified by (\ref{system}) and in agreement with the construction (\ref{d-o}) of 
second order terms in the recovery sequence for the Kirchhoff limiting
energies. Explicitly, we have:
\begin{equation*}
\vec d_0(x',x_3) = Q_0(x')^{\trsp, -1}\Big(\int_0^{x_3}
{\mathcal{G}}_1(x',t)\,\rmd t \, e_3 - \frac{1}{2}\int_0^{x_3}
{\mathcal{G}}_1(x',t)_{33}\,\rmd t \, e_3 - \frac{x_3^2}{2}\left[\begin{array}{c}(\nabla \vec 
    b_0)^{\trsp}\vec b_0(x')\\ 0\end{array}\right]\Big).
\end{equation*}
In what follows, the smooth matrix field in (\ref{d0})
will be referred to as $P_0:\bar\Omega\to\mathbb{R}^{3\times 3}$, namely:
\begin{equation}\label{P0}
P_0(x', x_3) = \left[x_3\partial_1\vec b_0(x'), ~ x_3\partial_2\vec b_0(x'), ~ \partial_3\vec d_0(x', x_3)\right].
\end{equation}
In the non-oscillatory case (\ref{NO}), the above formulas become:
\begin{equation}\label{d0no}
\begin{split}
& \vec d_0 = \frac{x_3^2}{2}\tilde d_0(x'), \qquad P_0(x', x_3) =
x_3\Big[\partial_1\vec b_0, ~ \partial_2\vec b_0, ~ \tilde d_0\Big](x'),\\ 
& \mbox{where: } \quad \tilde d_0(x') = Q_0(x')^{\trsp,
  -1}\Big(\partial_3 G(x',0)e_3 - \frac{1}{2}\partial_3G(x',0)_{33}e_3
-  \left[\begin{array}{c}(\nabla \vec b_0)^{\trsp}\vec b_0(x')\\ 0\end{array}\right]\Big).
\end{split}
\end{equation}
We also note that the assumption $\int_{-1/2}^{1/2}\mathcal{G}_1(x', t) \,\rmd t=0$ implies:
\begin{equation}\label{av0}
\int_{-1/2}^{1/2}P_0(x', x_3) \,\rmd x_3= 0 \qquad\mbox{for all } x'\in\bar\omega.
\end{equation}
With the aid of $\vec d_0$ we now construct the sequence of deformations with low energy:

\begin{lemma}\label{higher_order_scaling_sequence}
Assume (\ref{O}). Then (\ref{vonKa}) implies:
$$ \inf\mathcal E^h \leq Ch^4.$$
\end{lemma}
\begin{proof}
Define the sequence of smooth maps $\{u^h:\bar\Omega^h\to\R^3\}_{h\to  0}$ by:
\begin{equation}\label{changeva}
u^h(x', x_3)=  y_0(x') + x_3\vec b_0(x') + \displaystyle{h^2\vec d_0\big(x',\frac{x_3}{h}\big) } 
\end{equation}
In order to compute $\nabla u^h (A^h)^{-1}$, recall the expansion of $(A^h)^{-1}$, so that:
\begin{equation}\label{try}
\nabla u^h(x)A^h(x)^{-1} = Q_0(x')\bar A(x')^{-1}\big(Id_3 + hS^h(x) +\mathcal{O}(h^2)\big),
\end{equation}
where for every $x=(x', x_3)\in \Omega^h$:
\begin{equation*}
S^h(x) = \displaystyle{\bar A(x')^{-1}\Big(Q_0(x')^{\trsp}P_0\big(x', \frac{x_3}{h}\big)
- \bar A(x') A_1\big(x',\frac{x_3}{h}\big) \Big)\bar A(x')^{-1}. }
\end{equation*}
By frame invariance of the energy density $W$ and since
$Q_0(x')\bar A(x')^{-1}\in SO(3)$, we obtain:
\begin{equation*}
\begin{split}
W\big(\nabla u^h(x)A^h(x)^{-1}\big) & = W\big(Id_3 +h S^h(x) +\mathcal{O}(h^2)\big)
= W\big(Id_3 + hS^h(x)_{\sym}+\mathcal{O}(h^2)\big) \\ & = W\big(Id_3
+\mathcal{O}(h^2)\big) = \mathcal{O}(h^4),
\end{split}
\end{equation*}
where we also used the fact that $S^h(x)_{\sym}=0$ following directly from the definition (\ref{d0}).
This implies that $\mathcal{E}^h(u^h)= \mathcal{O}(h^4)$ as well,
proving the claim. 
\end{proof}

\begin{lemma}\label{approx}
Assume (\ref{O}) and (\ref{vonKa}). For an
open, Lipschitz subset  $\mathcal{V}\subset \omega$, denote:
$$\mathcal V^h=\mathcal V\times \big(-\frac{h}{2}, \frac{h}{2}\big),
\qquad \mathcal{E}^h(u^h, \mathcal{V}^h) =
\frac{1}{h}\int_{\mathcal{V}^h}W\big(\nabla u^h(A^h)^{-1}\big) \,\rmd x.$$
If $y_0$ is injective on $\mathcal{V}$, then for every $u^h\in
W^{1,2}(\mathcal V^h,\R^3)$ there exists $\bar R^h\in \SO(3)$ such that:  
\begin{equation}\label{approx_ineq}
\frac{1}{h}\int_{\mathcal V^h}\Big|\nabla u^h(x)-\bar R^h\Big(Q_0(x')
+ h P_0\big(x', \frac{x_3}{h}\big)\Big)\Big|^2\,\rmd x\leq C\big({\mathcal E}^h(u^h,\mathcal V^h) +
h^3|\mathcal V^h|\big),
\end{equation}
with the smooth correction matrix field $P_0$ in (\ref{P0}). The constant $C$ in
(\ref{approx_ineq}) is uniform for all subdomains 
$\mathcal{V}^h\subset \Omega^h$ which are bi-Lipschitz equivalent with
controlled Lipschitz constants. 
\end{lemma}
\begin{proof}
The proof,  similar to \cite[Lemma 2.2]{LRR}, is a combination of the
change of variable argument in Lemma \ref{weak_coercive} 
and the low energy deformation construction in Lemma
\ref{higher_order_scaling_sequence}. Observe first that:
$$Q_0(x') + hP_0\big(x', \frac{x_3}{h}\big) = \nabla Y^h(x', x_3) + \mathcal{O}(h^2),$$
where by $Y^h:\bar \Omega^h\to\mathbb{R}^3$ we denote the smooth vector
fields in (\ref{changeva}). It is clear that for sufficiently small
$h>0$, each $Y^h_{\mid\mathcal{V}^h}$ is a smooth 
diffeomorphism onto its image $\mathcal{U}^h\subset\mathbb{R}^3$,
satisfying uniformly: $\det\nabla Y^h>c>0$.  We now consider $v^h=u^h\circ
(Y^h)^{-1}\in W^{1,2}(\mathcal{U}^h, \mathbb{R}^3)$. By the rigidity
estimate \cite{FJM02}:
\begin{equation}\label{bzz}
\int_{\mathcal{U}^h}|\nabla v^h-\bar R^h|^2\,\rmd z\leq C
\int_{\mathcal{U}^h}\dist^2\big(\nabla v^h, SO(3)\big)\,\rmd z, 
\end{equation}
for some rotation $\bar R^h\in SO(3)$. Noting that: $(\nabla v^h)\circ
Y^h = (\nabla u^h)(\nabla Y^h)^{-1}$ in the set $\mathcal{V}^h$, the change of
variable formula yields for the left hand side in (\ref{bzz}): 
\begin{equation*}
\begin{split}
\int_{\mathcal{U}^h}|\nabla v^h-\bar R^h|^2\,\rmd z & =
\int_{\mathcal{V}^h}(\det \nabla Y^h)\big|(\nabla u^h)(\nabla
Y^h)^{-1}-\bar R^h\big|^2\,\rmd x\\ & \geq c \int_{\mathcal{V}^h}\Big|\nabla u^h
-\bar R^h \Big(Q_0(x') + hP_0\big(x', \frac{x_3}{h}\big)
+\mathcal{O}(h^2)\Big)\Big|^2\,\rmd x \\ & \geq  c \int_{\mathcal{V}^h}\Big|\nabla u^h
-\bar R^h \Big(Q_0(x') + hP_0\big(x',
\frac{x_3}{h}\big)\Big)\Big|^2\,\rmd x - c \int_{\mathcal{V}^h}\mathcal{O}(h^4)\,\rmd x. 
\end{split}
\end{equation*}
Similarly, the right hand side in (\ref{bzz}) can be estimated by:
\begin{equation*}
\begin{split}
\int_{\mathcal{U}^h}\dist^2\big(\nabla v^h, SO(3)\big)\,\rmd z & =
\int_{\mathcal{V}^h}(\det \nabla Y^h)\, \dist^2\big((\nabla u^h)(\nabla Y^h)^{-1}, SO(3)\big)\,\rmd x
\\ & \leq C \int_{\mathcal{V}^h} \dist^2\big((\nabla u^h) (A^h)^{-1}\cdot A^h(\nabla
Y^h)^{-1}, SO(3)\big)\,\rmd x\\ & \leq C \int_{\mathcal{V}^h}
\dist^2\Big((\nabla u^h) (A^h)^{-1}, SO(3) (\nabla
Y^h)(A^h)^{-1}\Big)\,\rmd x.
\end{split}
\end{equation*}
Recall that from (\ref{try}) we have: $(\nabla
Y^h)(A^h)^{-1}\in SO(3)\big(Id_3 + hS^h+\mathcal{O}(h^2)\big)\subset
SO(3)\big(Id_3 + \mathcal{O}(h^2)\big)$, since $S^h\in
so(3)$. Consequently, the above bound becomes:
\begin{equation*}
\begin{split}
\int_{\mathcal{U}^h}\dist^2\big(\nabla v^h, SO(3)\big)\,\rmd z & \leq C \int_{\mathcal{V}^h}
\dist^2\Big((\nabla u^h) (A^h)^{-1}, SO(3) (Id_3 +
\mathcal{O}(h^2))\Big)\,\rmd x \\ & \leq 
C \int_{\mathcal{V}^h} \dist^2\Big((\nabla u^h) (A^h)^{-1}, SO(3)\Big) + \mathcal{O}(h^4)\,\rmd x. 
\end{split}
\end{equation*}
The estimate (\ref{approx_ineq}) follows now in view of (\ref{bzz})
and by the lower bound on energy density $W$.
\end{proof}

The well-known approximation technique \cite{FJM02} combined with the arguments in 
\cite[Corollary 2.3]{LRR}, yield the following approximation result that can be seen
as a higher order counterpart of (\ref{Rh}): 

\begin{cor}\label{cor_approx}
Assume (\ref{O}) and (\ref{vonKa}). Then, for any
sequence $\{u^h\in W^{1,2}(\Omega^h,\R^3)\}_{h\to 0}$ satisfying:
$\mathcal{E}^h(u^h)\leq Ch^4,$
there exists a sequence of rotation-valued maps $R^h\in
W^{1,2}(\omega, SO(3))$, such that with $P_0$ defined in (\ref{P0}) we\ have:
\begin{equation}\label{eq: approx}
\begin{split}
& \frac{1}{h}\int_{\Omega^h}\Big|\nabla u^h(x)-R^h(x')\Big(Q_0(x')+hP_0\big(x',
\frac{x_3}{h}\big)\Big)\Big|^2\,\rmd x\leq Ch^4,\\
& \int_{\omega}|\nabla R^h(x')|^2\,\rmd x'\leq Ch^2.
\end{split}
\end{equation}
\end{cor}

\section{Compactness and $\Gamma$-limit under $Ch^4$ energy bound}\label{4}

In this section, we derive the $\Gamma$-convergence result for the
energy functionals $\mathcal{E}^h$ in the von K\'arm\'an scaling
regime. The general form of the limiting energy $\mathcal{I}_4^O$ will
be further discussed and split into the stretching, bending, curvature
and excess components in section \ref{sec77}. We begin by stating the
compactness result, that is the higher order version of Theorem \ref{compactness_thm}.

\begin{thm}\label{thm: lower_bound_4}
Assume (\ref{O}) and (\ref{vonKa}). Fix
$y_0$ solving (\ref{system}) and normalize it to have:
$\int_\omega y_0 \,\rmd x'=0$. Then, for any sequence of
deformations $\{u^h\in W^{1,2}(\Omega^h,\R^3)\}_{h\to 0}$ satisfying:
\begin{equation}\label{h4}
{\mathcal E}^h(u^h)\leq Ch^4,
\end{equation}
there exists a sequence $\{\bar R^h\in SO(3)\}_{h\to
  0}$ such that the following convergences (up to a subsequence) below,
hold for $y^h\in W^{1,2}(\Omega, \R^3)$:
$$y^h(x', x_3) = (\bar R^h)^{\trsp} \Big( u^h(x',hx_3) - \fint_{\Omega^h} u^h \,\rmd x\Big).$$
\begin{itemize}
\item [(i)] $y^h\to y_0$ strongly in $W^{1,2}(\Omega, \R^3)$ and
  $\displaystyle{\frac{1}{h}\partial_3y^h\to \vec b_0}$ strongly in $L^2(\Omega,
  \R^3)$, as $h\to 0$.\vspace{1mm}
\item [$(ii)$] There exists $V \in W^{2,2}(\omega, \R^3)$ and $\mathbb
  S\in L^2(\omega,\mathbb{R}^{2\times 2}_\sym)$ such that, as $h\to 0$:
\begin{equation*}
\begin{split}
& V^h(x')=\frac{1}{h}\int_{-1/2}^{1/2} y^h(x',x_3) - \big(y_0(x')+h
x_3\vec b_0(x')\big)\,\rmd x_3 \to V \quad \mbox{strongly in }
W^{1,2}(\omega, \R^3)\\
& \frac{1}{h}\big((\nabla y_0)^{\trsp}\nabla
V^h\big)_{\sym}\rightharpoonup\mathbb S\quad\mbox{ weakly in
} L^2(\omega, \R^{2\times 2}).
\end{split}
\end{equation*}
\item [$(iii)$] The limiting displacement $V$ satisfies: $\big((\nabla
  y_0)^{\trsp}\nabla V\big)_{\sym}=0$ in $\omega$.
\end{itemize}
\end{thm}

We omit the proof because it follows as in \cite[Theorem 3.1]{LRR} in view of
condition (\ref{av0}). We only recall the definitions used in the
sequel. The rotations $\bar R^h$ are given by:
$$\bar R^h = \mathbb{P}_{SO(3)}\fint_{\Omega^h}\nabla u^h(x) Q_0(x')^{-1}\,\rmd x$$
and (\ref{eq: approx}) implies that they satisfy, for some limiting  rotation $\bar R$:
\begin{equation}\label{r1}
\int_\omega |R^h(x') -\bar R^h|^2 \,\rmd x' \leq Ch^2\qquad
\mbox{and}\qquad \bar R^h\to \bar R\in SO(3).
\end{equation}
Consequently:
\begin{equation}\label{r2}
S^h = \frac{1}{h}\big((\bar R^h)^{\trsp}R^h(x') -Id_3\big)
\rightharpoonup S \qquad
\mbox{weakly in }\;W^{1,2}(\omega, \mathbb{R}^{3\times 3})
\end{equation}
The field $S\in W^{1,2}(\omega, so(3))$ is such that $(\nabla
y_0)^{\trsp}\nabla V = \big(Q_0^{\trsp}SQ_0\big)_{2\times 2}\in
so(2)$, which allows for defining a new vector field $\vec p\in W^{1,2}(\omega,
\mathbb{R}^3)$ through:
\begin{equation}\label{r3}
\left[\nabla V,~ \vec p\right] = S Q_0\quad
\mbox{ or equivalently: }\; \vec p(x') = -Q_0(x')^{\trsp,-1
}\left[\begin{array}{c}\nabla V(x')^{\trsp}\vec b_0(x') \\
    0\end{array}\right] \quad \mbox{for all } x'\in\omega.
\end{equation}
Finally, by (\ref{eq: approx}) we note the uniform 
boundedness of the fields $\{Z^h\in L^2(\Omega, \mathbb{R}^{3\times
  3})\}_{h\to 0}$ below, together with their convergence (up to as
subsequence) as $h\to 0$:
\begin{equation}\label{m-1}
{Z}^h(x) = \frac{1}{h^2} \Big(\nabla u^h(x', hx_3) - R^h(x')
\big(Q_0(x') + hP_0(x', x_3)\big)\Big) \rightharpoonup Z \qquad\mbox{
  weakly in } L^2(\Omega,\mathbb{R}^{3\times 3}).
\end{equation}
Rearranging terms and using the previously established convergences,
it can be shown that:
\begin{equation}\label{m7}
\mathbb S(x') =\left(Q_0^{\trsp}(x')\bar R^{\trsp}\int_{-1/2}^{1/2}Z(x',x_3)\,\rmd x_3\right)_{2\times
  2,\sym} -\frac{1}{2}\nabla V(x')^{\trsp}\nabla V(x') \qquad\mbox{
  for all } x'\in \omega. 
\end{equation}

\medskip

\begin{thm}\label{Gliminf4}
In the setting of Theorem \ref{thm: lower_bound_4}, $\liminf_{h\to 0}\frac{1}{h^4}\mathcal E^h(u^h)$ 
is bounded below by:
\begin{equation*}
\mathcal{I}^O_4(V,\mathbb S) = 
\frac{1}{2}\int_{\Omega}\mathcal{Q}_2\Big(x',I(x') +x_3III(x') +
II(x) \Big) \,\rmd x = \frac{1}{2} \|I+ x_3 III+ II\|_{\mathcal{Q}_2}^2,
\end{equation*}
where: 
\begin{equation}\label{123}
\begin{split}
I(x') & =  \mathbb{S}(x') + \frac{1}{2}\nabla
      V(x')^{\trsp}\nabla V(x') - \nabla y_0(x')^{\trsp} \nabla \int_{-1/2}^{1/2} \vec
      d_0(x) \,\rmd x_3, \\
III(x') & = \nabla y_0(x')^{\trsp}\nabla \vec p(x') + \nabla V(x')^{\trsp}\nabla\vec b_0 ,\\
II(x) & = \frac{x_3^2}{2}\nabla \vec b_0(x')^{\trsp}\nabla \vec b_0(x') +
\nabla y_0(x')^{\trsp}\nabla_{\mathrm{tan}}\vec d_0(x) - \frac{1}{4}\mathcal{G}_2(x)_{2\times 2}.
\end{split}
\end{equation}
\end{thm}
\begin{proof} 
{\bf 1.} Towards estimating the energy $\mathcal{E}^h(u^h)$, we
replace the argument $\nabla u^h(x) A^h(x)^{-1}$  
of the frame invariant density $W$ by:
\begin{equation}\label{m1}
\begin{split}
& \big(Q_0\bar A^{-1}\big)(x')^{\trsp}R^h(x')^{\trsp}\nabla u^h(x) A^h(x)^{-1}
\\ & \qquad\qquad = \big(Q_0\bar A^{-1}\big)(x')^{\trsp}Q_0(x') A^h(x)^{-1} + h \bar
A(x')^{-1}Q_0(x')^{\trsp} P_0\big(x', \frac{x_3}{h}\big)  A^h(x)^{-1}
\\ & \qquad\qquad\quad \; + h^2 I_3^h\big(x',\frac{x_3}{h}\big)\qquad
\mbox{for all } x\in\Omega^h,
\end{split}
\end{equation}
where $I^h_3$ is given in (\ref{m0}). Calculating the higher order expansion of (\ref{noga}):
\begin{equation}\label{Ah_expansion}
\begin{split}
A^h(x)^{-1} =  \bar A(x')^{-1} + \bar
    A(x')^{-1}\Big(-hA_1\big(x', \frac{x_3}{h}\big) &+ h^2 A_1\big(x',
      \frac{x_3}{h}\big)   \bar A(x')^{-1}A_1\big(x',
      \frac{x_3}{h}\big) \\ & -\frac{h^2}{2}A_2\big(x',\frac{x_3}{h}\big)
\Big) \bar A(x')^{-1} + o(h^2), 
\end{split}
\end{equation}
the expressions in (\ref{m1}) can be written as:
\begin{equation}\label{m2}
\begin{split}
\big(Q_0\bar A^{-1}\big)(x')^{\trsp}& R^h(x')^{\trsp}\nabla u^h(x',
hx_3) A^h(x', hx_3)^{-1} \\ & = Id_3 + h I_1(x', x_3) + h^2\Big(I_2(x', x_3)
+ I_3^h(x', x_3)\Big) + o(h^2) \qquad \mbox{for all } x\in\Omega,
\end{split}
\end{equation}
where $I_1:\Omega\to so(3)$ and $I_2:\Omega\to \mathbb{R}^{3\times 3}$
are smooth matrix fields, given by:
\begin{equation*}
\begin{split}
& I_1(x)  = \bar A(x')^{-1}\displaystyle{\Big(Q_0(x')^{\trsp}P_0(x) - \bar
      A(x') A_1(x) \Big) }\bar A(x')^{-1} \\
&  I_2(x)  = \bar A(x')^{-1}\displaystyle{\Big(\bar A(x') A_1(x) \bar
  A(x')^{-1}A_1(x) - \frac{1}{2}\bar A(x') A_2(x) - Q_0(x')^{\trsp}P_0(x)\bar
      A(x')^{-1} A_1(x) \Big) }\bar A(x')^{-1}.
\end{split}
\end{equation*}
The fact that $I_1(x)\in so(3)$ follows from (\ref{d0}). Also, we have:
\begin{equation}\label{m0}
\begin{split}
I_3^h(x) = &~\bar A(x')^{-1}Q_0(x')^{\trsp}R^h(x')^{\trsp}Z^h(x) A^h(x',
hx_3)^{-1}\\ & \qquad \rightharpoonup I_3 (x) = \bar
A(x')^{-1}Q_0(x')^{\trsp}\bar R^{\trsp}Z(x) \bar A(x)^{-1}
\qquad\mbox{weakly in } L^2(\Omega, \mathbb{R}^{3\times 3}),
\end{split}
\end{equation}
where we used (\ref{m1}) and (\ref{r1}) to pass to the limit with $(R^h)^{\trsp}$.
As in the proof of Theorem \ref{Gliminf}, we now identify the ``good''
sets $\{|I_3^h|^2\leq 1/h\}\subset \Omega$ and employ (\ref{m2}) to
write there the following Taylor's expansion of $W(\nabla u^h(A^h)^{-1})$:
\begin{equation}\label{m3}
\begin{split}
W\Big(\nabla u^h(x', \,&hx_3)  A^h(x', hx_3)^{-1}\Big) = 
W\Big(Id_3 + h I_1(x) + h^2(I_2(x) + I_3^h(x)) + o(h^2)
\Big) \\ & = W\Big(e^{-hI_1(x)}\big(Id_3 + h I_1(x) + h^2(I_2(x) + I_3^h(x)) \big) 
+ o(h^2)\Big) \\ & 
= W\Big( Id_3+ h^2\big(I_2 -\frac{1}{2}I_1^2+ I^h_3\big) + o(h^{2})\Big)
\\ & =\frac{h^4}{2} \mathcal{Q}_3\Big(\big(I_2 -\frac{1}{2}I_1^2+
I^h_3\big)_\sym\Big) + o(h^4).
\end{split}
\end{equation}
Above, we repeatedly used the frame invariance of $W$ and the
exponential formula: 
$$e^{-h I_1}=Id_3 - h I_1 + \frac{h^2}{2}I_1^2 + \mathcal{O}(h^3).$$
Since the weak convergence in (\ref{m0}) implies convergence of
measures $\big||I_3^h|^2\geq 1/h\big|\to 0$ as $h\to 0$, with the help
of (\ref{m3}) we finally arrive at:
\begin{equation}\label{m4}
\begin{split}
\liminf_{h\to 0}\frac{1}{h^4}\frac{1}{h}\int_{\Omega^h}&W\big(\nabla u^h(x)
A^h(x)^{-1}\big) \,\rmd x  \geq \liminf_{h\to
  0}\frac{1}{2}\int_{|I_3^h|^2\leq 1/h}\mathcal{Q}_3\big(
I_2 -\frac{1}{2}I_1^2+ I^h_3\big) \,\rmd x \\ & \geq \frac{1}{2}\int_{\Omega}\mathcal{Q}_3\big(
I_2 -\frac{1}{2}I_1^2+ I_3\big) \,\rmd x =
\frac{1}{2}\int_{\Omega}\mathcal{Q}_2\Big(x', \big(\bar A
(I_2 -\frac{1}{2}I_1^2+ I_3)\bar A\big)_{2\times 2}\Big) \,\rmd x.
\end{split}
\end{equation}

\smallskip

{\bf 2.}  We now compute the effective integrand in (\ref{m4}). Firstly, by (\ref{Adef}) a direct calculation yields:
\begin{equation}\label{m5}
\big(I_2(x)-\frac{1}{2} I_1(x)^2\big)_{\sym}  =\big(I_2\big)_{\sym} 
+\frac{1}{2} I_1^{\trsp}I_1 =\displaystyle{\frac{1}{2}\bar
    A(x')^{-1}\Big(P_0(x)^{\trsp}P_0(x) - \frac{1}{2}\mathcal{G}_2(x)\Big) }\bar A(x')^{-1} 
\end{equation}

Secondly, to address the symmetric part of the limit $I_3$ in (\ref{m0}),  consider functions
$f^{s,h}:\Omega\to\mathbb{R}^3$:
$$f^{s,h}(x)= \fint_{0}^s \Big(h(\bar R^h)^{\trsp}Z^h(x', x_3+t) +S^h(x')
\big(Q_0(x') + hP_0(x', x_3+t)\big)\Big)e_3 \,\rmd t.$$
By (\ref{r2}) it easily follows that:
\begin{equation}\label{m6}
f^{s,h}\to S\vec b_0 = \vec p \qquad \mbox{ strongly in } \; L^2(\Omega,
\mathbb{R}^3), \quad \mbox{ as } h\to 0.
\end{equation}
On the other hand, we write an equivalent form of $f^{s,h}$ and
compute the tangential derivatives:
\begin{equation*}
\begin{split} 
& f^{s,h}(x)= \frac{1}{h^2s} \big(y^h(x', x_3+s) - y^h(x', x_3)\big) -
\frac{1}{h}\vec b_0(x') - \displaystyle{\frac{1}{s}\Big(\vec
      d_0(x', x_3+s) - \vec d_0(x', x_3)\Big)},\\
& \partial_if^{s,h}(x)= ~\frac{1}{s}(\bar R^h)^{\trsp}\big(Z^h(x',
x_3+s) - Z^h(x', x_3)\big)e_i + S^h(x')\partial_i\vec b_0(x') \\ & 
\qquad\qquad \qquad\qquad \qquad\qquad \qquad\qquad \qquad\qquad \qquad
-\displaystyle{\frac{1}{s}\Big(\partial_i\vec
      d_0(x', x_3+s) - \partial_i\vec d_0(x', x_3)\Big) } 
\end{split}
\end{equation*}
for $i=1,2$. In view of (\ref{r1}) and (\ref{r2}) and equating the
tangential derivatives $\partial_1, \partial_2$, result in:
$$ \bar R^{\trsp}\big(Z(x', x_3)
- Z(x', 0) \big)e_i = x_3\big( \partial_i\vec p(x') -
S(x') \partial_i\vec b_0(x') \big) + \displaystyle{\partial_i\vec
      d_0(x', x_3) - \partial_i\vec d_0(x', 0)}.$$
Further, by (\ref{m0}), (\ref{r3}) and since $S\in so(3)$, it follows that:
\begin{equation}\label{m8}
\begin{split}
\big(\bar A(x') I_3(x) &\bar A(x')\big)_{2\times 2, \sym}  = 
\Big(Q_0(x')^{\trsp}\bar R^{\trsp}Z(x)\Big)_{2\times 2,\sym} \\ &
= ~\Big(Q_0(x')^{\trsp}\bar R^{\trsp}Z(x',0)\Big)_{2\times 2,\sym}+ x_3\Big( \nabla y_0(x')^{\trsp}\nabla \vec p(x')
+ \nabla V(x')^{\trsp}\nabla\vec b_0 \Big)_{\sym}
\\ & \qquad\qquad \qquad\qquad \qquad\qquad \qquad
+ \displaystyle{\Big(Q_0(x')^{\trsp}\nabla \vec d_0(x) \Big)_{2\times 2, \sym}}
\end{split}
\end{equation}
On the other hand, taking the $x_3$-average and recalling (\ref{m7}), we get:
\begin{equation}\label{m9}
\begin{split}
\Big(Q_0(x')^{\trsp}&\bar R^{\trsp}Z(x',0)\Big)_{2\times 2,\sym}  = 
\mathbb{S}(x')+\frac{1}{2}\nabla V(x')^{\trsp}\nabla V(x')
-\Big(\nabla y_0(x')^{\trsp}\nabla\int_{-1/2}^{1/2}\vec d_0(x) \,\rmd x_3\Big)_{\sym}
\end{split}
\end{equation}

\smallskip

{\bf 3.} We now finish the proof of Theorem \ref{Gliminf4}. Combining
(\ref{m5}), (\ref{m8}) and (\ref{m9}), we see that:
$$\Big(\bar A(x')(I_2 - \frac{1}{2}I_1^2 + I_3)\bar
A(x')\Big)_{2\times 2, \sym} = \Big(I(x') +x_3III(x') + II(x)\Big)_{\sym} \qquad\mbox{on } \;\Omega,$$
where $I, II, III$ are as in (\ref{123}). In virtue of (\ref{m4}), we obtain:
\begin{equation*}
\begin{split}
\liminf_{h\to 0}\frac{1}{h^4}\frac{1}{h}&\int_{\Omega^h}W\big(\nabla u^h(x)
A^h(x)^{-1}\big) \,\rmd x  \geq 
\frac{1}{2}\int_{\Omega}\mathcal{Q}_2\Big(x',I(x') +x_3III(x') + II(x) \Big) \,\rmd x.
\end{split}
\end{equation*}
This yields the claimed lower bound by $\mathcal{I}_4^O(V,\mathbb{S})$.
\end{proof}

\medskip

For the upper bound statement, define the linear spaces:
\begin{equation}\label{VS}
\begin{split}
\mathscr V & = \Big\{V\in \WW^{2,2}(\omega, \R^3); ~ \big(\nabla
y_0(x')^{\trsp}\nabla V(x')\big)_{\sym}=0\,\mbox{ for all } x'\in\omega\Big\},\\ 
\mathscr S & = {\rm cl}_{L^2(\omega, \R^{2\times
    2})}\Big\{\big((\nabla y_0)^{\trsp}\nabla w\big)_{\sym};~ w\in W^{1,2}(\omega, \R^3)\Big\}.
\end{split}
\end{equation}
We see that  the limiting quantities $V$ and $\mathbb{S}$ in Theorem
\ref{Gliminf4} satisfy: $V\in\mathscr{V}$, $\mathbb{S}\in\mathscr{S}$.
The space $\mathscr{V}$ consists of the first order infinitesimal
isometries on the smooth minimizing immersion surface $y_0(\omega)$,
i.e. those Sobolev-regular displacements $V$ that preserve the metric on $y_0(\omega)$ up
to first order. The tensor fields $\mathbb{S}\in\mathscr{S}$ are the
finite strains on $y_0(\omega)$, eventually forcing the stretching
term in the von K\'arm\'an energy $\mathcal{I}_4^O$ to be of second order.

\begin{thm} \label{Glimsuph4}
Assume that $y_0$ solves (\ref{system}). Then, for every
$(V, \mathbb{S})\in \mathscr{V}\times \mathscr{S}$
there exists a sequence $\{u^h\in W^{1,2}(\Omega^h, \R^3)\}_{h\to 0}$
such that the rescaled sequence $\{y^h(x', x_3) = u^h(x', hx_3)\}_{h\to 0}$ satisfies (i) and (ii) of Theorem
\ref{thm: lower_bound_4}, together with:
\begin{equation}\label{up_bd} 
\lim_{h\to 0} \frac{1}{h^4}\mathcal E^h(u^h) = \mathcal{I}_4^O(V, \mathbb{S}).
\end{equation}
\end{thm}
\begin{proof} 
{\bf 1.} Given admissible $V$ and  $\mathbb{S}$, we first define the
$\varepsilon$-recovery sequence $\{u^h\in W^{1,\infty}(\Omega^h,
\R^3)\}_{h\to 0}$. The ultimate argument for (\ref{up_bd}) will be
obtained via a diagonal argument. We set:
\[
\begin{aligned}
u^h(x', x_3) = & \, \, y_0(x')+ h v^h(x')+h^2w^h(x')+x_3\vec b_0(x')+
h^2\vec d_0\big(x', \frac{x_3}{h}\big) \\
& + h^3 \vec k_0\big(x', \frac{x_3}{h}\big) + hx_3\vec
p^{\,h} (x') + h^2x_3\vec q^{\,h}(x') + h^3\vec r^{\,h}\big(x',
\frac{x_3}{h}\big) \qquad \mbox{for all }\, (x',x_3)\in \Omega^h.
\end{aligned}
\]
The smooth vector fields $\vec b_0$ and $\vec d_0$ are as in
(\ref{system}), (\ref{d0}). We now introduce other terms in the above expansion.
The sequence $\{w^h\in \mathcal{C}^{\infty}(\overline \omega, \R^3)\}_{h\to 0}$ is such that: 
\begin{equation}\label{wh}
\begin{split}
& \Big((\nabla y_0)^{\trsp}\nabla \big(w^h+\int_{-1/2}^{1/2}\vec
d_0(\cdot,t)\,\rmd t\big)\Big)_{\sym}\to \mathbb S \quad\mbox{
  strongly in }\, \LL^2(\omega, \R^{2\times 2}) ~~\mbox{ as } h\to 0, \\
& \lim_{h\to 0}\sqrt{h}\|w^h\|_{\WW^{2,\infty}(\omega, \R^3)}=0.
\end{split}
\end{equation}
Existence of such a sequence is guaranteed by the fact that $\mathbb
S\in \mathscr S$, where we ``slow down'' the approximations $\{w^h\}$
to guarantee the blow-up rate of order less that $h^{-1/2}$.
Further, for a fixed small $\varepsilon >0$, the truncated sequence $\{v^h\in
W^{2,\infty}(\omega, \R^3)\}_{h\to 0}$ is chosen according to the
standard construction in \cite{FJM02} (see also references therein), in a way that: 
\begin{equation}\label{vh}
\begin{split}
& v^h\to V\quad\mbox{ strongly in }\,W^{2,2}(\omega,\R^3) ~~ \mbox{ as }\,h\to 0,\\
& h\|v^h\|_{W^{2,\infty}(\omega, \R^3)}\leq \varepsilon\quad\mbox{ and
}\quad\lim_{h\to 0}\frac{1}{h^2}\,\big|\{x'\in \omega;~ v^h(x')\neq
V(x')\}\big|=0. 
\end{split}
\end{equation}
The vector field $\vec k_0\in \mathcal{C}^{\infty}(\bar\Omega, \R^3)$
and sequences $\{\vec p^{\,h}, \vec q^{\,h}\in 
W^{1,\infty}(\omega, \R^3)\}_{h\to 0}$, $\{\tilde r^h\in \LL^{\infty}(\Omega, \R^3)\}_{h\to 0}$ are defined by:
\begin{equation}\label{phqh}
\begin{aligned}
Q_0^{\trsp}\vec p^{\,h} =&\,\left[\begin{matrix}
-(\nabla v^h)^{\trsp}\vec b_0\\ 0\end{matrix}\right],\\
Q_0^{\trsp}\vec q^{\,h} = &\, c\Big(x',\big((\nabla
y_0)^{\trsp}\nabla w^h\big)_{\sym}+\frac{1}{2}(\nabla
v^h)^{\trsp}\nabla v^h\Big)-\left[\begin{matrix}
(\nabla v^h)^{\trsp}\vec p^{\,h}\\ \frac{1}{2}|\vec p^{\,h}|^2\end{matrix}\right]-\left[\begin{matrix}
(\nabla w^h)^{\trsp}\vec b_0\\ 0\end{matrix}\right],\\
Q_0^{\trsp} \partial_3 \vec k_0=&\, c\Big(x',  \big((\nabla
y_0)^{\trsp}\nabla_{\rm tan} \vec d_0\big)_{\sym} + \frac{x_3^2}{2}(\nabla \vec b_0)^{\trsp}\nabla 
\vec b_0 - \frac{1}{4}(\mathcal G_2)_{2\times 2}\Big) \vspace{1mm} \\  & - \left[\begin{matrix}
x_3(\nabla \vec b_0)^{\trsp}\partial_3 \vec d_0\\
\frac{1}{2}|\partial_3 \vec d_0|^2\end{matrix}\right]
+ \left[\begin{matrix} (\nabla_{\rm tan} \vec d_0)^{\trsp}\vec b_0\\
0\end{matrix}\right]  + \frac{1}{2}\mathcal G_2e_3-\frac{1}{4} ({\mathcal G}_2)_{33}e_3,\\
Q_0^{\trsp}\tilde r^h =&\, x_3 c\Big(x', \big((\nabla
y_0)^{\trsp}\nabla \vec p^{\,h}+ (\nabla v^h)^{\trsp}\nabla
\vec b_0\big)_{\sym}\Big)-\left[\begin{matrix} 
(\nabla v^h)^{\trsp}\partial_3 \vec d_0\\
\langle \vec p^{\,h}, \partial_3 \vec d_0\rangle\end{matrix}\right].
\end{aligned}
\end{equation}
Finally, we choose $\{\vec r^{\,h}\in W^{1,\infty}(\Omega, \R^3)\}_{h\to  0}$ to satisfy:
\begin{equation}\label{rhh}
\lim_{h\to 0}\|\partial_3 \vec r^{\,h}-\tilde r^h\|_{\LL^2(\Omega,
  \R^3)}=0\quad\mbox{ and }\quad \lim_{h\to
  0}\sqrt{h}\|\vec r^{\,h}\|_{\WW^{1,\infty}(\Omega,\R^3)}=0. 
\end{equation}

\medskip

{\bf 2.} Observe that for all $(x',x_3)\in\Omega$ there holds::
\[
\begin{aligned}
\nabla u^h(x',hx_3)= & \, Q_0 + h\Big(\big[\nabla v^h,~
\vec p^{\,h}\big] + P_0\Big) + h^2\Big(\big[\nabla w^h, ~\vec
q^{\,h}\big]+\big[x_3\nabla \vec p^{\,h},~ \partial_3\vec r^{\,h}\big] +
\big(\nabla_{\rm tan}\vec d_0,~\partial_3\vec k_0\big]\Big) \\ & + \mathcal{O}(h^3)\Big(
|\nabla \vec k_0| + |\nabla \vec q^{\,h}| + |\nabla \vec r^{\,h}|\Big).
\end{aligned}
\]
Consequently, by (\ref{Ah_expansion}) it follows that:
\[
\big((\nabla u^h)(A^h)^{-1}\big)(x',hx_3) = Q_0\bar A^{-1}\Big(Id_3 + h
\bar A^{-1} J_1^h \bar A^{-1} + h^2 \bar A^{-1}J_2^h \bar A^{-1} + J_h^3\Big),
\]
where:
\[
\begin{aligned}
J_1^h & = Q_0^{\trsp}\Big(\big[\nabla v^h, ~\vec p^{\,h}\big] + P_0\Big) -
\bar A A_1,\\
J_2^h & = Q_0^{\trsp}\Big(\big[\nabla w^h, ~\vec q^{\,h}\big] +
\big[x_3\nabla \vec p^{\,h}, ~\partial_3\vec r^{\,h}\big] + \big[\nabla \vec d_0, ~\partial_3\vec k_0\big] \Big) -
J_1^h \bar A^{-1} A_1 - \frac{1}{2}\bar A A_2,
\end{aligned}
\]
and where $J_1^h$, $J_2^h$, $J_3^h$ satisfy the uniform bounds (independent of $\varepsilon$):
\[
\begin{aligned}
|J_1^h| & \leq C \big( 1+ |\nabla v^h|\big), \\
|J_2^h| & \leq C \big(1+ |\nabla w^h| + |\nabla v^h|^2 + |\nabla^2 v^h| + |\nabla \vec r^{\,h}|\big),\\
|J_3^h| & \leq C h^3 \big(1+ |\nabla w^h| + |\nabla^2 w^h| + |\nabla v^h|^2 + |\nabla^2 v^h|
+ |\nabla v^h|\cdot|\nabla^2 v^h|+|\nabla \vec r^{\,h}|\big) + o(h^2).
\end{aligned}
\]
In particular, the distance $\dist\big((\nabla u^h) (A^h)^{-1}, SO(3)\big)\leq |
(\nabla u^h) (A^h)^{-1}- Q_0\bar A^{-1}|$ is as small as one wishes,
uniformly in $x\in\Omega$, for $h$ sufficiently small. Thus, the
argument $ (\nabla u^h) (A^h)^{-1}$ of the frame invariant density $W$ in $\mathcal{E}^h(u^h)$
may be replaced by its polar decomposition  factor:
\[
\begin{aligned}
\Big(\big(\nabla u^h(A^h)^{-1}\big)^{\trsp}\big(\nabla u^h(A^h)^{-1}\big)\Big)^{1/2}
& = \Big( Id_3 + 2h^2 \bar A^{-1} \big( (J_2^h)_{\sym} + \frac{1}{2}
(J_1^h)^{\trsp}\bar A^{-2} J_1^h\big) \bar A^{-1} +\mathcal{R}^h\Big)^{1/2}\\ 
& =  Id_3 + h^2 \bar A^{-1} \big( (J_2^h)_{\sym} + \frac{1}{2}
(J_1^h)^{\trsp}\bar A^{-2} J_1^h\big) \bar A^{-1} + \mathcal{R}^h,
\end{aligned}
\]
where $\mathcal{R}^h$ stands for any quantity obeying the following bound:
\[
\begin{split}
\mathcal R^h = & ~\mathcal{O}(h) |(J_1^h)_{\sym}| + \mathcal{O}(h^3)\big(
1+|\nabla v^h|\big)\big(1+|\nabla w^h| + |\nabla v^h|^2 +|\nabla^2
v^h|+|\nabla \vec r^{\,h}|\big) \\ & + \mathcal{O}(h^3)|\nabla^2w^h| + o(h^2).
\end{split}
\]

In conclusion, Taylor's expansion of $W$ at $Id_3$ gives:
\begin{equation}\label{TE1}
\begin{aligned}
\frac{1}{h^4}\int_{\Omega}W\Big((\nabla
u^h)&(A^h)^{-1}(x',hx_3)\Big)\,\rmd x \\ = & \,
\frac{1}{h^4}\int_{\Omega}W\Big(\big(\big(\nabla
u^h(A^h)^{-1}\big)^{\trsp}\big(\nabla u^h(A^h)^{-1}\big)\big)^{1/2}(x', hx_3)\Big) 
\,\rmd x  \\ \leq & \, \frac{1}{2}\int_{\Omega}\mathcal Q_3\Big(\bar A^{-1} \big(
(J_2^h)_{\sym} + \frac{1}{2}(J_1^h)^{\trsp}\bar A^{-2}J_1^h\big)\bar
A^{-1} + \frac{1}{h^2}\mathcal{R}^h\Big)\,\rmd x \\ & + \mathcal{O}(h^2) \int_\Omega
|J_2^h|^3 + |J_1^h|^6 \,\rmd x + \frac{\mathcal{O}(1)}{h^4}\int_\Omega |\mathcal{R}^h|^3 \,\rmd x.
\end{aligned}
\end{equation}
The residual terms above are estimated as in \cite{LRR}, using
(\ref{wh}), (\ref{vh}), (\ref{rhh}). We have:
$$h^2 \int_\Omega|J_2^h|^3 + |J_1^h|^6 \,\rmd x \leq h^2 \int_\Omega
1+ |\nabla w^h|^3 + |\nabla v^h|^6 + |\nabla^2 v^h|^3 + |\nabla \vec r^{\,h}|^3
\,\rmd x  \leq o(1),$$
as $ h^2 \int_\Omega|\nabla v^h|^6 \,\rmd x \leq C h^2 \|\nabla
v^h\|_{W^{1,2}}^6 = o(1)$ and $
h^2 \int_\Omega|\nabla^2 v^h|^3 \,\rmd x \leq \epsilon h\int_\Omega
|\nabla^2 v^h|^2 \,\rmd x= o(1)$. Further:
\begin{equation*}
\begin{aligned}
\frac{1}{h^4}\int_\Omega |\mathcal{R}^h|^2 \,\rmd x & \leq 
\frac{1}{h^2}\int_\Omega \big|\big((\nabla y_0)^{\trsp}\nabla
v^h\big)_{\sym}\big|^2 \,\rmd x \\ & \qquad + \mathcal{O}(h^2) \int_\Omega
\big(1+|\nabla v^h|^2\big)\big(1+|\nabla w^h|^2+|\nabla v^h|^4 +
|\nabla^2 v^h|^2+|\nabla \vec r^{\,h}|^2\big) \,\rmd x \\ & \qquad + \mathcal{O}(h^2)
\int_\Omega |\nabla^2w^h|^2 \,\rmd x + o(1) \\ & = o(1) + \mathcal{O}(h^2)
\int_\Omega |\nabla v^h|\cdot |\nabla^2 v^h|^2 \leq C\varepsilon,
\end{aligned}
\end{equation*}
because the last condition in (\ref{vh}) implies:
\begin{equation}\label{before}
\begin{split}
\frac{1}{h^2}\int_\Omega \big|\big((\nabla y_0)^{\trsp}&\nabla
v^h\big)_{\sym}\big|^2 \,\rmd x  \leq \frac{C}{h^2}\|\nabla^2
v^h\|_{L^\infty}\int_{\{v^h\neq V\}} \dist^2(x', \{v^h=V\}) \,\rmd x'
\\ & \leq \frac{C\varepsilon^2}{h^4} \int_{\{v^h\neq V\}} \dist^2(x',
\{v^h=V\}) \,\rmd x'\leq C\varepsilon^2 \frac{1}{h^2}\big|\{v^h\neq
V\}\big| = o(1).
\end{split}
\end{equation}
From the two estimates above  it also follows that
$\frac{1}{h^4}\int_\Omega |\mathcal{R}^h|^3 \,\rmd x  =
o(1)$. Consequently, (\ref{TE1}) yields:
\begin{equation}\label{TE2}
\limsup_{h\to 0}\frac{1}{h^4}\mathcal E^h(u^h)\leq C\varepsilon + 
\limsup_{h\to 0}\frac{1}{2}\int_\Omega \mathcal Q_3\Big(\bar A^{-1} \big(
(J_2^h)_{\sym} + \frac{1}{2}(J_1^h)^{\trsp}\big)\bar
A^{-2}J_1^h\big)\bar A^{-1}\Big)\,\rmd x.
\end{equation}

\medskip

{\bf 3.} Observe now that:
\[
\begin{aligned}
(J_2^h)_{\sym} + \frac{1}{2}(J_1^h)^{\trsp}\bar A^{-2}J_1^h = &  
- \Big(\big((\nabla y_0)^{\trsp}\nabla v^h\big)_{\sym}^*\bar
A^{-1}A_1\Big)_{\sym} \\ & + \Big(Q_0^{\trsp}\big[\nabla w^h,~\vec q^{\,h}\big]
+ Q_0^{\trsp}\big[x_3\nabla \vec p^{\,h},~ \partial_3\vec r^{\,h}\big] + Q_0^{\trsp}\big[\nabla \vec
d_0,~\partial_3\vec k_0\big]\Big)_{\sym} \\ & + \frac{1}{2}\big[\nabla v^h,
~ \vec p^{\,h}\big]^{\trsp}\big[\nabla v^h,~ \vec p^{\,h}\big] + \Big(\big[\nabla v^h,
~ \vec p^{\,h}\big]^{\trsp} P_0\Big)_{\sym} + \frac{1}{2}P_0^{\trsp}P_0 - \frac{1}{4}\mathcal{G}_2.
\end{aligned}
\]
Replacing $\partial_3\vec r^{\,h}$ by $\tilde r^h$ and using (\ref{phqh}), it follows that:
\[
\begin{aligned}
\bigg(\int_\Omega \mathcal{Q}_3\Big( &\bar A^{-1}\big((J_2^h)_{\sym} +
\frac{1}{2}(J_1^h)^{\trsp}\bar A^{-2}J_1^h\big)\bar A^{-1}\Big) \,\rmd x \bigg)^{1/2}
\\ \leq &   \bigg(\int_\Omega \mathcal{Q}_2\Big(x', \big((\nabla y_0)^{\trsp}\nabla
w^h\big)_{\sym} + \frac{1}{2}(\nabla v^h)^{\trsp}\nabla v^h +
x_3\big((\nabla y_0)^{\trsp}\nabla \vec p^{\,h}+(\nabla
v^h)^{\trsp}\nabla\vec b_0\big)_{\sym} \\ & \qquad\qquad\qquad
\qquad\qquad\qquad  \qquad\qquad\qquad \qquad\qquad +\frac{x_3^2}{2}(\nabla \vec
b_0)^{\trsp}\nabla \vec b_0 - \frac{1}{4}(\mathcal{G}_2)_{2\times 2}
\Big) \,\rmd x \bigg)^{1/2} \\ & + \|\big((\nabla y_0)^{\trsp}\nabla
v^h\big)_{\sym} \|_{L^2(\Omega)} + \|\partial_3\vec r^{\,h} - \tilde r^h\|_{L^2(\Omega)}.
\end{aligned}
\]
The second term above converges to $0$ by (\ref{before}) and the third
term also converges to $0$, by (\ref{rhh}). On the other hand, the
first term can be split into the integral on the set $\{v^h=V\}$, whose
limit as $h\to 0$ is estimated by $\mathcal{I}_4^O(V, \mathbb{S})$, and
the remaining integral that is bounded by:
\begin{equation*}
\begin{split}
C&\int_{\{v^h\neq V\}\times (-\frac{1}{2}, \frac{1}{2})}  1+ |\nabla w^h|^2 + |\nabla v^h|^4 + |\nabla^2
v^h|^2 + |\nabla \vec r^{\,h}|^3 \,\rmd x \\ & \qquad \leq C \epsilon^2 \frac{1}{h^2}
|\{v^h\neq V\}| + C\int_{\{v^h\neq V\}}|\nabla
v^h|^4 \,\rmd x' \leq o(1) + C |\{v^h\neq V\}|^{1/2} \|\nabla
v^h\|_{L^8}^4 = o(1).
\end{split}
\end{equation*}

In conclusion, (\ref{TE2}) becomes (with a uniform constant $C$ that
does not depend on $\varepsilon$):
\begin{equation*}
\limsup_{h\to 0}\frac{1}{h^4}\mathcal E^h(u^h)\leq C\varepsilon + 
\mathcal{I}_4^{O}(V, \mathbb{S}).
\end{equation*}
A diagonal argument applied to the indicated $\varepsilon$-recovery sequence
$\{u^h\}_{h\to 0}$ completes the proof.
\end{proof}

\begin{cor}\label{26new}
The functional $\mathcal{I}_4^O$  attains its infimum and there holds:
$$\lim_{h\to 0} \frac{1}{h^4}\inf\mathcal{E}^h =\min \;\mathcal{I}_4^O.$$
The infima  in the left hand side are taken over 
$W^{1,2}(\Omega, \mathbb{R}^3)$ deformations $u^h$, whereas the minimum in
the right hand side is taken over admissible displacement-strain couples $(V,\mathbb{S})\in
\mathscr V\times \mathscr S$ in (\ref{VS}).
\end{cor}

\section{Further discussion of $\mathcal{I}_4^O$ and reduction to the
  non-oscillatory case (\ref{NO})}\label{sec77}

In this section, we identify the appropriate components of the integrand in
the energy $\mathcal{I}_4^O$ as: stretching, bending, curvature and
the order-$4$ excess, the latter quantity being the projection of the
entire integrand on the orthogonal complement of
$\mathbb{E}_2$ in $\mathbb{E}$. This superposition is in the same spirit, as the integrand of
$\mathcal{I}_2^O$ in Theorem \ref{Gliminf} decoupling
into bending and the order-$2$ excess, defined as the projection on
the orthogonal complement of $\mathbb{E}_1$. 
There, the assumed condition $\int_{-1/2}^{1/2}\mathcal{G}_1
\,\rmd x_3=0$ served as the compatibility criterion, assuring that the
$2$-excess being null results in ${\mathcal{I}}_4^O$ coinciding  
with the non-oscillatory limiting energy ${\mathcal{I}}_4$, written for the
effective metric $\bar G$ in (\ref{EO}).  
Below, we likewise derive the parallel version $\mathcal{I}_4$  of 
${\mathcal I}_4^O$, corresponding to the non-oscillatory
case, and show that the vanishing of the $4$-excess reduces
${\mathcal{I}}_4^O$ to ${\mathcal{I}}_4$ (for the effective metric
(\ref{EO})), under two new further compatibility conditions (\ref{constr}) on $(\mathcal{G}_2)_{2\times 2}$.

\medskip

The following formulas will be useful in the sequel:
\begin{lemma}\label{new_basis}
In the non-oscillatory setting (\ref{NO}), let $y_0$, $\vec b_0$ be as
in (\ref{system}) and $\tilde d_0$ as in (\ref{d0no}). Then:
\begin{equation}\label{newba}
\big[\partial_{ij}y_0, ~ \partial_i\vec b_0, ~ \tilde d_0\big](x') =
\big[\partial_1y_0, ~ \partial_2\vec y_0, ~ \vec b_0\big](x')\cdot
\left[\begin{array}{ccc}\Gamma_{ij}^1 & \Gamma_{i3}^1 &  \Gamma_{33}^1 \\
\Gamma_{ij}^2 & \Gamma_{i3}^2 &  \Gamma_{33}^2 \\
\Gamma_{ij}^3 & \Gamma_{i3}^3 &  \Gamma_{33}^3
\\ \end{array}\right](x',0) \qquad\mbox{for }\; i,j=1,2,
\end{equation}
for all $x'\in\omega$.
Consequently, for any smooth vector field $\vec q:\omega\to\R^3$ there holds:
$$\Big[\nabla y_0(x')^{\trsp}\nabla \big(Q_0(x')^{\trsp, -1}\vec q(x')\big)\Big]_{i,j=1,2}
= \nabla \vec q(x')_{2\times 2} - \Big[\Big\langle \vec q(x'), [\Gamma_{ij}^1,
~ \Gamma_{ij}^2, ~ \Gamma_{ij}^3](x',0)\Big\rangle \Big]_{i,j=1,2}.$$
Above, $\{\Gamma_{ij}^k\}$ are the Christoffel symbols of the metric
$G$ and the expression in the right and side represents the tangential
part of the covariant derivative of the $(0,1)$ tensor field $\vec q$
with respect to $G$.
\end{lemma}
\begin{proof}
In view of $\big((\nabla y_0)^{\trsp}\nabla \vec b_0\big)_{\sym} = \frac{1}{2}\partial_3G(x',0)_{2\times
  2}$ in (\ref{system}) and recalling (\ref{y0}), we get:
$$\langle \partial_{ij}y_0, \vec b_0\rangle =
\frac{1}{2}\big(\partial_iG_{j3}+\partial_jG_{i3}-\partial_3G_{ij}\big)(x',0)\qquad\mbox{for
 all } i,j=1,2,$$
which easily results in:
\begin{equation*}
\langle \partial_i\vec b_0, \partial_jy_0\rangle =
\frac{1}{2}\big(\partial_iG_{j3}-\partial_jG_{i3}+\partial_3G_{ij}\big)(x',0)
\quad\mbox{ and } \quad \langle \partial_i\vec b_0, \vec b_0\rangle =
\frac{1}{2}\partial_iG_{33}(x',0).
\end{equation*}
Thus (\ref{y2}) and the above allow for computing the
coordinates in the basis $\partial_1y_0, \partial_2y_0, \vec b_0$ as
claimed in (\ref{newba}); see also \cite[Theorem 6.2]{LRR} for more
details. The second formula results from:
\begin{equation*}
\begin{split}
\big\langle \partial_iy_0, \partial_j \big(Q_0^{\trsp, -1}\vec q\big)\big\rangle
& = \big\langle \partial_iy_0, \partial_j \big(Q_0^{\trsp, -1}\big)\vec q\big\rangle
+ \big\langle \partial_iy_0, Q_0^{\trsp, -1}\partial_j\vec q\big\rangle \\ & = 
- \big\langle \partial_iy_0, Q_0^{\trsp, -1}\partial_j \big(Q_0^{\trsp}\big) Q_0^{\trsp, -1}\vec q\big\rangle
+ \big\langle Q_0^{-1}\partial_iy_0, \partial_j\vec q\big\rangle \\ & = 
- \big\langle Q_0^{-1}\partial_j \big(Q_0\big)e_i, \vec q\big\rangle
+ \big\langle e_i, \partial_j\vec q\big\rangle,
\end{split}
\end{equation*}
which together with (\ref{newba}) yields the Lemma.
\end{proof}

\begin{lemma}\label{pomoc_4}
In the non-oscillatory setting (\ref{NO}), let $y_0$, $\vec b_0$ be as
in (\ref{system}) and $\tilde d_0$ as in (\ref{d0no}). Then
the metric-related term $II$ in (\ref{123}) has the form
$II = \frac{x_3^2}{2}\bar{II}(x')$ and for all $x'\in \omega$ we have:
\begin{equation}\label{tensor2}
\bar{II}_{\sym} = (\nabla \vec b_0)^{\trsp}\nabla \vec b_0+
\big((\nabla y_0)^{\trsp}\nabla\tilde d_0\big)_{\sym} - \frac{1}{2}\partial_{33}G(x', 0)_{2\times 2}
  = \left[\begin{array}{cc} R_{1313} & R_{1323} \\ R_{1323} & R_{2323} \end{array}\right](x',0).
\end{equation}
Above, $R_{ijkl}$ are the Riemann curvatures of the metric $G$,
evaluated at the midplate points $x'\in\omega$.
\end{lemma}
\begin{proof}
We argue as in the proof of \cite[Theorem 6.2]{LRR}. Using (\ref{d0}) we arrive at:
\begin{equation}\label{y1}
\begin{split}
\big((\nabla y_0)^{\trsp}\nabla \tilde d_0\big)_\sym = & -
\big[\langle\partial_{ij}y_0, \vec d_0\rangle \big]_{i,j=1,2} 
+ \frac{1}{2}\partial_{33}G(x',0) _{2\times 2}\\ & +
\Big[R_{i2j3} - G_{np}\big(\Gamma_{i3}^n\Gamma_{j3}^p -
\Gamma_{ij}^n\Gamma_{33}^p\big)\Big]_{i,j=1, 2} (x',0).
\end{split}
\end{equation}
Directly from (\ref{newba}) we hence obtain:
\begin{equation}\label{y22}
\langle \partial_{ij}\vec y_0, \tilde d_0\rangle =
G_{np}\Gamma_{ij}^n\Gamma_{33}^{p}, \qquad 
\langle \partial_{i}\vec b_0, \partial_j\vec b_0\rangle = G_{np}\Gamma_{i3}^n\Gamma_{j3}^{p},
\end{equation}
which together with (\ref{y1}) yields (\ref{tensor2}).
\end{proof}

\medskip

With the use of Lemma \ref{pomoc_4}, it is quite straightforward to
derive the ultimate form of the energy $\mathcal{I}_4^O$ in the
non-oscillatory setting. In particular, the proof of the following
result is a special case of the proof of Theorem \ref{Gliminf4-2} below.

\begin{thm}\label{novk}
Assume (\ref{NO}) and (\ref{vonKa}). The expression (\ref{I4O-new}) becomes:
\begin{equation*}
\begin{split}
\mathcal{I}_4(V,\mathbb S)  = &~\frac{1}{2}\int_{\omega}\mathcal{Q}_2\Big(x', 
\mathbb{S}(x') + \frac{1}{2}\nabla V(x')^{\trsp}\nabla V(x') +
  \frac{1}{24} \nabla \vec b_0(x')^T\nabla\vec b_0(x')-
  \frac{1}{48} \partial_{33}G(x', 0)_{2\times 2}\Big) \,\rmd x' \\ & 
+ \frac{1}{24}\int_{\omega}\mathcal{Q}_2\Big(x', \nabla y_0(x')^{\trsp}\nabla \vec p(x')+ \nabla
V(x')^{\trsp}\nabla\vec b_0(x') \Big) \,\rmd x' \\ & 
+ \frac{1}{1440}\int_{\omega}\mathcal{Q}_2\Big(x', 
\left[\begin{array}{cc} R_{1313} & R_{1323} \\ R_{1323} & R_{2323} \end{array}\right](x',0)\Big) \,\rmd x',
\end{split}
\end{equation*}
where $R_{ijkl}$ stand for the Riemann curvatures of the metric $G$.
\end{thm}

\begin{remark}\label{flat4}
In the particular, ``flat'' case of $G=Id_3$ the functional
$\mathcal{I}_4$ reduces to the classical von K\'arm\'an energy below.
Indeed, the unique solution to
(\ref{system}) is: $y_0=id$, $\vec b_0=e_3$ and further:
\begin{equation*}
\begin{split}
& \mathscr V = \big\{V(x) = (\alpha x^\perp +\vec \beta, v(x)); ~
\alpha\in\mathbb{R}, ~ \vec \beta\in\mathbb{R}^2, ~ v\in W^{2,2}
(\omega)\big\},\\
& \mathscr S = \big\{\sym\nabla w; ~ w\in W^{1,2} (\omega, \mathbb{R}^2)\big\}.
\end{split}
\end{equation*}
Given $V\in \mathscr V$, we have $\vec p = (-\nabla
v, 0)$ and thus: 
$$\mathcal{I}_4(V, \mathbb{S}) = \frac{1}{2}\int_\omega
  \mathcal{Q}_2\big(x', \sym\nabla w+\frac{1}{2}(\alpha^2Id_2 + \nabla
  v\otimes\nabla v)\big) \,\rmd x'+\frac{1}{24}\int_\omega
  \mathcal{Q}_2\big(x', \nabla^2 v\big) \,\rmd x'.$$ 
Absorbing the stretching $\alpha^2Id_2$ into $\sym\nabla w$, the above
energy can be expressed in a familiar form:
\begin{equation}\label{vkflat}
\mathcal{I}_4(v,w) = \frac{1}{2}\int_\omega
  \mathcal{Q}_2\big(x', \sym\nabla w+\frac{1}{2} \nabla
  v\otimes\nabla v\big) \,\rmd x'+\frac{1}{24}\int_\omega
  \mathcal{Q}_2\big(x', \nabla^2 v\big) \,\rmd x',
\end{equation} 
as a function of the out-of-plane scalar displacement
$v$ and the in-plane vector displacement $w$.
\end{remark}

\medskip

As done for the Kirchhoff energy $\mathcal{I}_2^O$ in Theorem
\ref{Gliminf}, we now identify conditions allowing $\mathcal{I}_4^O$ to
coincide with $\mathcal{I}_4$ of the effective metric $\bar G$, modulo
the introduced below order-$4$ excess term. 

\begin{thm}\label{Gliminf4-2}
In the setting of Theorem \ref{Gliminf4}, we have:
\begin{equation}\label{I4O-new}
\begin{split}
\mathcal{I}^O_4(V,\mathbb S) = & \,\frac{1}{2}
\int_\omega\mathcal{Q}_2\Big(x', \mathbb{S}+\frac{1}{2}(\nabla 
V)^{\trsp}\nabla V + B_0\Big) \,\rmd x' \\ & + \frac{1}{24}\int_\omega\mathcal{Q}_2\Big(x', (\nabla
V)^{\trsp}\nabla \vec b_0 + (\nabla y_0)^{\trsp}\nabla \vec p  + 12 B_1 \Big)
\,\rmd x' \\ & + \frac{1}{1440}\int_\omega\mathcal{Q}_2\Big(x', (\nabla
\vec b_0)^{\trsp}\nabla \vec b_0 + (\nabla y_0)^{\trsp}\nabla \tilde d_0
- \frac{1}{2}(\bar{\mathcal{G}}_2)_{2\times  2}\Big) \,\rmd x' 
\\ & + \frac{1}{2}\dist^2_{\mathcal{Q}_2}\Big(II_{\sym}, \mathbb{E}_2\Big),
\end{split}
\end{equation}
where $\bar{\mathcal{G}}_1$ and $\bar{\mathcal{G}}_2$ are given in
(\ref{EO}), inducing $\tilde d_0$ via (\ref{d0no}) for
$\partial_3G=\bar{\mathcal{G}}_1$, and where we introduce the
following purely metric-related quantities:
\begin{equation}\label{g2bending}
\begin{split}
& \dist^2_{\mathcal{Q}_2}\big(II_{\sym}, \mathbb{E}_2\big)  =
\dist^2\bigg( \int_0^{x_3}\nabla(\mathcal G_1e_3)_{2\times 2, 
  \sym}\,\rmd t \\ & \qquad \qquad \qquad \qquad\qquad \qquad
- \Big[\Big\langle \int_0^{x_3}\mathcal G_1e_3 \,\rmd t
, [\Gamma_{ij}^1,~ \Gamma_{ij}^2, ~ \Gamma_{ij}^3](x',0)\Big\rangle \Big]_{i,j=1,2}\\
& \qquad \qquad \qquad \qquad\qquad \qquad
+ \frac{1}{2}\int_0^{x_3}(\mathcal G_1)_{33} \,\rmd t \,\big[\Gamma^3_{ij}(x',0)\big]_{i,j=1,2}
 - \frac{1}{4}(\mathcal{G}_2)_{2\times 2}, \,\mathbb{E}_2\bigg),
\end{split}
\end{equation}
\begin{equation}\label{g1bending}
\begin{split}
& B_0 = \frac{1}{24} (\nabla\vec b_0)^{\trsp}\nabla\vec b_0
- \frac{1}{4}\int_{-1/2}^{1/2}(\mathcal{G}_2)_{2\times 2}\,\rmd x_3 
= \frac{1}{24}\Big[\bar{\mathcal{G}}_{np}\Gamma_{i3}^n\Gamma_{j3}^{p}\Big]_{i,j=1,2} 
- \frac{1}{4}\int_{-1/2}^{1/2}(\mathcal{G}_2)_{2\times 2}\,\rmd x_3, \\
& B_1 =  (\nabla y_0)^{\trsp}\nabla \big(\int_{-1/2}^{1/2}x_3\vec d_0 \,\rmd
x_3\big) -\frac{1}{4}\int_{-1/2}^{1/2}x_3(\mathcal{G}_2)_{2\times
  2}\,\rmd x_3 \\ & \quad \, = -\nabla \big(\int_{-1/2}^{1/2}\frac{x_3^2}{2} 
{\mathcal{G}}_1e_3\,\rmd x_3\big)_{2\times 2} + 
\Big[\Big\langle \int_{-1/2}^{1/2}\frac{x_3^2}{2}
{\mathcal{G}}_1e_3\,\rmd x_3, [\Gamma_{ij}^1,
~ \Gamma_{ij}^2, ~ \Gamma_{ij}^3](x',0)\Big\rangle \Big]_{i,j=1,2} \\
& \qquad \; - \frac{1}{2} \int_{-1/2}^{1/2}\frac{x_3^2}{2}
(\mathcal{G}_1)_{33}\,\rmd x_3 \Big[\Gamma_{ij}^3(x',0)\Big]_{i,j=1,2}
-\frac{1}{4}\int_{-1/2}^{1/2}x_3(\mathcal{G}_2)_{2\times 2}\,\rmd x_3,
\end{split}
\end{equation}
By $\{\Gamma_{ij}^k\}$ we denote the Christoffel symbols of the
metric $\bar G$ in (\ref{EO}). The third term in (\ref{I4O-new})
equals  the scaled norm of the Riemann curvatures of the effective metric $\bar G$:
$$\frac{1}{1440}\int_{\omega}\mathcal{Q}_2\Big(x', 
\left[\begin{array}{cc} R_{1313} & R_{1323} \\ R_{1323} & R_{2323} \end{array}\right](x',0)\Big) \,\rmd x'.$$
The first three terms in $\mathcal{I}_4^{O}$ coincide with
$\mathcal{I}_4$ in Theorem \ref{novk} for the
effective metric $\bar G$ in (\ref{EO}), provided that the following
compatibility conditions hold:
\begin{equation}\label{constr}
\begin{split}
& \int_{-1/2}^{1/2} \big(15x_3^2 - \frac{9}{4}\big)\mathcal{G}_2(x', x_3)_{2\times 2} \,\rmd x_3 = 0,\\
& \frac{1}{4}\int_{-1/2}^{1/2} x_3\mathcal{G}_2(x', x_3)_{2\times 2} \,\rmd x_3  +
\nabla \big(\int_{-1/2}^{1/2}\frac{x_3^2}{2} {\mathcal{G}}_1e_3\,\rmd
x_3\big)_{2\times 2, \sym} \\ & \qquad \qquad \qquad \qquad \qquad -
\Big[\Big\langle \int_{-1/2}^{1/2}\frac{x_3^2}{2}
{\mathcal{G}}_1e_3\,\rmd x_3, [\Gamma_{ij}^1,
~ \Gamma_{ij}^2, ~ \Gamma_{ij}^3](x',0)\Big\rangle \Big]_{i,j=1,2} 
\\ & \qquad \qquad \qquad \qquad \qquad \qquad \qquad \qquad 
+ \frac{1}{2} \int_{-1/2}^{1/2}\frac{x_3^2}{2}
(\mathcal{G}_1)_{33}\,\rmd x_3 \Big[\Gamma_{ij}^3(x',0)\Big]_{i,j=1,2} = 0.
\end{split}
\end{equation}
\end{thm}
\begin{proof}
We write:
\begin{equation*}
\begin{split}
\mathcal{I}_4^O(V, \mathbb{S})= \frac{1}{2} \|I+ x_3 III+ II\|_{\mathcal{Q}_2}^2
= \frac{1}{2} \|I+ x_3 III+ \mathbb{P}_2(II)\|_{\mathcal{Q}_2}^2+
\frac{1}{2}\dist^2 _{\mathcal{Q}_2}\big(II_{\sym}, \mathbb{E}_2\big),
\end{split}
\end{equation*}
and further decompose the first term above along the Legendre projections:
\begin{equation*}
\begin{split}
\|I+ x_3& III+ \mathbb{P}_2(II)\|_{\mathcal{Q}_2}^2 =
\big\|\int_{-1/2}^{1/2} (I+ x_3 III+ II) p_0(x_3)  \,\rmd x_3
\big\|_{\mathcal{Q}_2}^2 \\ & \qquad + 
\big\|\int_{-1/2}^{1/2} (I+ x_3 III+ II) p_1(x_3)  \,\rmd x_3 \big\|_{\mathcal{Q}_2}^2 + 
\big\|\int_{-1/2}^{1/2} (I+ x_3 III+ II) p_2(x_3)  \,\rmd x_3 \big\|_{\mathcal{Q}_2}^2 \\ & =
\underbrace{\big\|I+ \int_{-1/2}^{1/2} II  \,\rmd x_3 \big\|_{\mathcal{Q}_2}^2}_{\Large{Stretching}} +
\frac{1}{12} \underbrace{\big\|III + 12\int_{-1/2}^{1/2}x_3 II \,\rmd
  x_3 \big\|_{\mathcal{Q}_2}^2}_{\Large{Bending}}  + 
\underbrace{\big\|\int_{-1/2}^{1/2} p_2(x_3)II  \,\rmd x_3
\big\|_{\mathcal{Q}_2}^2}_{\Large{Curvature}}.
\end{split}
\end{equation*}
To identify the four indicated terms in $\mathcal{I}_4^O$, observe that
$\int_{-1/2}^{1/2}x_3\int_0^{x_3}\mathcal{G}_1 \,\rmd x_3 = - 
\int_{-1/2}^{1/2}\frac{x_3^2}{2}\mathcal{G}_1 \,\rmd x_3$ and:
\begin{equation*}
\begin{split}
\dist^2_{\mathcal{Q}_2}&\big(II_{\sym}, \mathbb{E}_2\big) \\ & = \dist^2_{\mathcal{Q}_2}\bigg(\Big((\nabla
y_0)^{\trsp}\nabla\big(Q_0^{\trsp, -1}\int_0^{x_3}\mathcal G_1e_3 \,\rmd t
- \frac{1}{2}Q_0^{\trsp, -1}\int_0^{x_3}(\mathcal G_1)_{33} \,\rmd t
\,e_3\big)\Big)_{\sym} - \frac{1}{4}(\mathcal{G}_2)_{2\times 2},
\mathbb{E}_2\bigg).
\end{split}
\end{equation*}
Thus the formulas in (\ref{g2bending}) and (\ref{g1bending}) follow directly from Lemma
\ref{new_basis} and (\ref{y22}). There also holds:
\begin{equation*}
\begin{split}
Stretching = & \int_\omega\mathcal{Q}_2\Big(x', \mathbb{S}+\frac{1}{2}(\nabla
V)^{\trsp}\nabla V + \frac{1}{24}(\nabla \vec b_0)^{\trsp}\nabla \vec b_0 
- \frac{1}{4}\int_{-1/2}^{1/2}(\mathcal{G}_2)_{2\times 2}\,\rmd
x_3\Big) \,\rmd x', \\
Bending = & \, \frac{1}{12}\int_\omega\mathcal{Q}_2\Big(x', (\nabla
V)^{\trsp}\nabla \vec b_0 + (\nabla y_0)^{\trsp}\nabla \vec p 
\\ & \qquad\qquad \qquad 
+ 12 (\nabla y_0)^{\trsp}\nabla \big(\int_{-1/2}^{1/2}x_3\vec d_0 \,\rmd x_3\big)
- 3\int_{-1/2}^{1/2}x_3(\mathcal{G}_2)_{2\times 2}\,\rmd x_3\Big),
\,\rmd x' \\ Curvature = & \, \frac{1}{720}\int_\omega\mathcal{Q}_2\Big(x', (\nabla
\vec b_0)^{\trsp}\nabla \vec b_0 + 60 (\nabla y_0)^{\trsp}\nabla 
\big(\int_{-1/2}^{1/2}(6x_3^2-\frac{1}{2})\vec d_0 \,\rmd x_3\big) \\
& \qquad\qquad\qquad \;\;
- 15\int_{-1/2}^{1/2}(6x_3^2-\frac{1}{2})(\mathcal{G}_2)_{2\times  2}\,\rmd x_3\Big) \,\rmd x'.
\end{split}
\end{equation*}
It is easy to check that with the choice of the effective metric
components $\bar{\mathcal{G}}_1e_3$ and
$(\bar{\mathcal{G}}_2)_{2\times 2}$ and denoting $\tilde d_0$ the
corresponding vector in (\ref{d0no}), we have:
\begin{equation*}
\begin{split}
Curvature = & \, \frac{1}{720}\int_\omega\mathcal{Q}_2\Big(x', (\nabla
\vec b_0)^{\trsp}\nabla \vec b_0 + (\nabla y_0)^{\trsp}\nabla \tilde d_0
- \frac{1}{2}(\bar{\mathcal{G}}_2)_{2\times  2}\Big) \,\rmd x'.
\end{split}
\end{equation*}
This proves (\ref{I4O-new}).  Equivalence of the constraints (\ref{constr}) with:
$$\int_{-1/2}^{1/2}(\mathcal{G}_2)_{2\times 2}\,\rmd x_3 =
\frac{1}{12}(\bar{\mathcal{G}}_2)_{2\times 2} \quad \mbox{ and } \quad
(B_1)_{\sym} =0 \qquad\mbox{in }\, \omega,$$ 
follows by a direct inspection. We now invoke Lemma \ref{pomoc_4} to
complete the proof.
\end{proof}

\medskip

\begin{remark}
Observe that the vanishing of the $4$-excess and curvature terms in $\mathcal{I}_4^O$:
$$II_{\sym}\in\mathbb{E}_2\quad\mbox{ and } \quad Curvature =0,$$ 
are the necessary conditions for $\min\mathcal{I}_4^O=0$ and they are
equivalent to $II_{\sym}\in \mathbb{E}_1$. Consider now a particular case
scenario of $\bar{\mathcal{G}}=Id_3$ and $\mathcal{G}_1=0$, where the
spaces $\mathscr V$ and $\mathscr S$ are given in Remark
\ref{flat4}, together with $\vec d_0=0$. Then, the above necessary
condition reduces to: $(\mathcal{G}_2)_{2\times 2}\in\mathbb{E}_1$, namely:
$$(\mathcal{G}_2)_{2\times 2}(x', x_3)=x_3\mathcal{F}_1(x') +
\mathcal{F}_0(x') \qquad \mbox{for all }\, x=(x', x_3)\in \bar\Omega.$$
It is straightforward that, on a simply connected midplate $\omega$, both terms:
\begin{equation*}
\begin{split}
& Stretching = \int_\omega\mathcal{Q}_2\Big(x',\sym\nabla w +\frac{1}{2}\nabla v\otimes \nabla v -
\frac{1}{4}\mathcal{F}_0\Big) \,\rmd x', \quad  Bending =
\int_\omega\mathcal{Q}_2\Big(x', \nabla^2v + \frac{1}{4}\mathcal{F}_1 \Big) \,\rmd x',
\end{split}
\end{equation*}
can be equated to $0$ by choosing appropriate
displacements $v$ and $w$, if and only if there holds:
\begin{equation}\label{mama}
\mbox{curl}\,\mathcal{F}_1=0,\qquad \mbox{curl}^{\trsp}\mbox{curl}\, \mathcal{F}_0
+ \frac{1}{4}\det\mathcal{F}_1 = 0\quad\mbox{ in }\, \omega.
\end{equation}
Note that these are precisely the linearised Gauss-Codazzi-Mainardi
equations corresponding to the metric $Id_2+2h^2\mathcal{F}_0$ and
shape operator $\frac{1}{2}h\mathcal{F}_1$ on $\omega$. We see that 
these conditions are automatically satisfied in presence of
(\ref{constr}), when $(\mathcal{G}_2)_{2\times 2}\in\mathbb{E}_1$
actually results in $(\mathcal{G}_2)_{2\times 2}=0$.  
An integrability criterion similar to (\ref{mama}) can be derived also
in the general case, under $II_{\sym}\in\mathbb{E}_1$ and again it automatically
holds with (\ref{constr}). This last statement will be pursued in the next section.
\end{remark}

\section{Identification of the $Ch^4$ scaling regime and coercivity of
the limiting energy $\mathcal{I}_4$} \label{sec8h4}

\begin{thm}\label{vK_optimal}
The energy scaling beyond the von K\'arm\'an regime:
$$\lim_{h\to 0}\frac{1}{h^4}\inf \mathcal{E}^h = 0$$
is equivalent to the following condition, on a simply connected $\omega$:
\begin{itemize}
\item[(i)] in the oscillatory case (\ref{O}), in presence of the
  compatibility conditions (\ref{constr})
\begin{equation}\label{y10}
\left[~\mbox{\begin{minipage}{15cm} \vspace{1mm}
$II_{\sym}\in\mathbb{E}_2$ and (\ref{y11}) holds with $G$ replaced by the effective
metric $\bar{G}$ in (\ref{EO}). This condition involves 
$\bar{\mathcal{G}}$, $\bar{\mathcal{G}}_1$ and
$(\bar{\mathcal{G}}_2)_{2\times 2}$ terms of $\bar{G}$. \vspace{1mm} 
\end{minipage}}\right.
\end{equation}
\item[(ii)]  in the non-oscillatory case (\ref{NO})
\begin{equation}\label{y11}
\left[~\mbox{\begin{minipage}{15cm} All the Riemann curvatures
      of the metric $G$ vanish on $\omega\times \{0\}$:
$$R_{ijkl}(x',0) = 0\qquad
\mbox{for all } \; x'\in \omega\quad \mbox{ and all } \;i,j,k,l=1\ldots 3.$$
\end{minipage}}\right.
\end{equation}
\end{itemize}
\end{thm}
\begin{proof}
By Corollary \ref{26new}, it suffices to determine the equivalent
conditions for $\min \mathcal{I}_4=0$. Clearly, $\min \mathcal{I}_4=0$ implies
(\ref{y11}). Vice versa, if (\ref{y11}) holds, then:
$$\frac{1}{24}(\nabla \vec
b_0)^{\trsp}\nabla \vec b_0-\frac{1}{48}\partial_{33}G(x',0) = -\frac{1}{24}\big((\nabla
y_0)^{\trsp}\nabla \tilde d_0\big)_{\sym},$$
by Lemma \ref{pomoc_4}. Taking $V=\vec p=0$ and
$\mathbb{S}=\frac{1}{24}\big((\nabla y_0)^{\trsp}\nabla \tilde d_0\big)_{\sym}\in
\mathscr S$, we get $\mathcal{I}_4(V,\mathbb{S})=0$. 
\end{proof}

\smallskip

We further have the following counterpart of the essential uniqueness of the
minimizing isometric immersion $y_0$ statement
in Theorem \ref{Kirchhoff_optimal}:

\begin{thm}\label{kernel_vK}
In the non-oscillatory setting (\ref{NO}), assume (\ref{y11}). Then
$\mathcal{I}_4(V,\mathbb{S})=0$ if and only if:
\begin{equation}\label{ml0}
V=Sy_0+c \quad \mbox{and} \quad \mathbb{S}=\frac{1}{2}\Big((\nabla
  y_0)^{\trsp}\nabla \big(S^2y_0 + \frac{1}{12}\tilde
  d_0\big)\Big)_{\sym} \quad\mbox{ on }\; \omega,
\end{equation}
for some skew-symmetric matrix $S\in so(3)$ and a vector $c\in\mathbb{R}^3$.
\end{thm}
\begin{proof}
We first observe that the bending term $III$ in (\ref{123}) is already symmetric, because:
\begin{equation*}
\begin{split}
\Big[\langle\partial_iy_0, \partial_j\vec p\rangle +
\langle\partial_iV, \partial_j\vec b_0\rangle\Big]_{i,j=1,2} & = \Big[\partial_j\big( 
\langle\partial_iy_0, \vec p\rangle + \langle\partial_iV, \vec b_0\rangle\big)\Big]_{i,j=1,2}
- \Big[\langle\partial_{ij}y_0, \vec p\rangle +
\langle\partial_{ij}V, \vec b_0\rangle\Big]_{i,j=1,2}\\ & = - \Big[\langle\partial_{ij}y_0, \vec p\rangle +
\langle\partial_{ij}V, \vec
b_0\rangle\Big]_{i,j=1,2}\in\mathbb{R}^{2\times 2}_{\sym},
\end{split}
\end{equation*}
where we used the definition of $\vec p$ in (\ref{r3}).
Recalling (\ref{tensor2}), we see that $\mathcal{I}_4(V,\mathbb{S})=0$ if and only if:
\begin{equation}\label{ml1}
\begin{split}
&\mathbb{S} + \frac{1}{2}(\nabla V)^{\trsp}\nabla V -
 \frac{1}{24}\big((\nabla y_0)^{\trsp}\nabla \tilde d_0\big)_{\sym}=0,\\
&(\nabla y_0)^{\trsp}\nabla\vec p + (\nabla V)^{\trsp}\nabla \vec b_0 =0.
\end{split}
\end{equation}

Consider the matrix field $S=\big[\nabla V,~\vec p\big]Q_0^{-1}\in
W^{1,2}(\omega, so(3))$ as in (\ref{r3}). Note that:
\begin{equation}\label{ml2}
\begin{split}
\partial_i S = & ~\big[\nabla \partial_i V, ~ \partial_i\vec p\big]Q_0^{-1}-
\big[\nabla V, ~ \vec p\big]Q_0^{-1}(\partial_iQ_0)Q_0^{-1}=
Q_0^{-1,\trsp}\bar S^i Q_0^{-1}\qquad\mbox{for }\; i=1,2\\
&\mbox{ where } \quad \bar S^i = Q_0^{\trsp}\big[\nabla \partial_iV,
~ \partial_i\vec p\big]+ \big[\nabla V, ~ \vec p\big]^{\trsp}(\partial_iQ_0)\in L^2(\omega, so(3)).
\end{split}
\end{equation}
Then we have:
$$\langle \bar S^ie_1, e_2\rangle
= \partial_i\big(\langle \partial_2y_0, \partial_iV\rangle +
\langle \partial_2 V, \partial_iy_0\rangle\big) - \big(\langle \partial_{12}y_0, \partial_iV\rangle +
\langle \partial_{12}V, \partial_iy_0\rangle\big) =0,$$
because the first term in the right hand side above equals $0$ in view of
$V\in \mathscr V$, whereas the second
term equals $\partial_2\langle\partial_1y_0, \partial_1V\rangle$ for $i=1$ and
$\partial_1\langle\partial_2y_0, \partial_2V\rangle$ for $i=2$, both
expression being null again in view of $V\in\mathscr V$.
We now claim that $\{\bar S^i\}_{i=1,2}=0$ is actually equivalent to the second
condition in (\ref{ml1}). It suffices to examine the only possibly
nonzero components:
\begin{equation}\label{ml3}
\langle \bar S^ie_3, e_j\rangle = \langle \partial_
jy_0, \partial_i\vec p\rangle +\langle \partial_jV, \partial_i\vec b_0\rangle = \big((\nabla
y_0)^{\trsp}\nabla\vec p + (\nabla V)^{\trsp}\nabla \vec
b_0\big)_{ij} \qquad \mbox{for all }\; i,j=1,2,
\end{equation}
proving the claim.

Consequently, the second condition in (\ref{ml1}) is equivalent to $S$
being constant, to the effect that $\nabla V = \nabla (Sy_0)$, or
equivalently that $V-Sy_0$ is a constant vector. In this case:
$$\mathbb{S}=\frac{1}{2}(\nabla y_0)^{\trsp}\nabla\big(S^2 y_0) +
\frac{1}{24}\big((\nabla y_0)^{\trsp}\nabla \tilde d_0 \big)_{\sym} =
\frac{1}{2}\Big((\nabla y_0)^{\trsp}\nabla \big(S^2y_0+
\frac{1}{12}\tilde d_0 \big)\Big)_{\sym}$$
is equivalent to  the first condition in (\ref{ml1}),  as $(\nabla V)^{\trsp}\nabla
V = -(\nabla y_0)^{\trsp}S^2\nabla y_0$. The proof is done.
\end{proof}

\medskip

From Theorem \ref{kernel_vK} we deduce its quantitative version, that
is a counterpart of Theorem \ref{coercive2} in the present von
K\'arm\'an regime:

\begin{thm}\label{coercive4}
In the non-oscillatory setting (\ref{NO}), assume (\ref{y11}). Then
for all $V\in \mathscr V$ there holds:
\begin{equation}\label{8.1}
\dist^2_{W^{2,2}(\omega, \mathbb{R}^3)} \Big(V, ~ \big\{Sy_0 + c; ~
S\in so(3),~ c\in\mathbb{R}^3\big\}\Big) \leq
C\int_\omega\mathcal{Q}_2\big(x', (\nabla y_0)^{\trsp}\nabla \vec p +
(\nabla V)^{\trsp}\nabla \vec b_0\big) \,\rmd x'
\end{equation}
with a constant $C>0$ that depends on $G,\omega$ and $W$ but it is
independent of $V$.
\end{thm}
\begin{proof}
We argue by contradiction. Since $\mathscr{V}_{lin}\doteq \{Sy_0 + c; ~ S\in so(3),
~ c\in\mathbb{R}^3\}$ is a linear subspace of $\mathscr{V}$ and
likewise the expression $III$ in (\ref{123}) is linear in $V$, with
its  kernel equal to $\mathscr{V}_{lin}$ in virtue of Theorem \ref{kernel_vK}, it
suffices to take a sequence $\{V_n\in\mathscr{V}\}_{n\to\infty}$ such  that:
\begin{equation}\label{ml4}
\begin{split}
& \|V_n\|_{W^{2,2}(\omega, \mathbb{R}^3)}=1, \qquad 
V_n\perp_{W^{2,2}(\omega, \mathbb{R}^3)}\mathscr{V}_{lin}\qquad\mbox{for all
} n, \\ & \mbox{and: }\quad (\nabla
y_0)^{\trsp}\nabla \vec p_n +(\nabla V_n)^{\trsp}\nabla \vec b_0\to
0\quad \mbox{strongly in } L^2(\omega,\mathbb{R}^{2\times 2}), \quad\mbox{as }\; n\to\infty.
\end{split}
\end{equation}
Passing to a subsequence if necessary and using the definition of
$\vec p$ in (\ref{r3}), it follows that:
\begin{equation}\label{ml4.5}
V_n\rightharpoonup V \quad\mbox{weakly in }
W^{2,2}(\omega,\mathbb{R}^3), \qquad \vec p_n\rightharpoonup \vec p \quad\mbox{weakly in }
W^{1,2}(\omega,\mathbb{R}^3). 
\end{equation}
Clearly, $Q_0^{\trsp}\big[\nabla V, ~ \vec p\big]\in L^2(\omega,
so(3))$ so that $V\in \mathscr{V}$, but also $(\nabla y_0)^{\trsp}\nabla \vec p +
(\nabla V)^{\trsp}\nabla \vec b_0=0$. Thus, Theorem \ref{kernel_vK}
and the perpendicularity assumption in (\ref{ml4}) imply:
$V=\vec p = 0$. We will now prove:
\begin{equation}\label{ml4.55}
V_n\to 0 \quad \mbox{ strongly in }\; W^{2,2}(\omega, \mathbb{R}^3),
\end{equation}
which will contradict the first (normalisation) condition in (\ref{8.1}).

As in (\ref{ml2}), the assumption $V_n\in \mathscr{V}$ implies that
for each $x'\in \omega$ and $i=1,2$, the following matrix (denoted previously by
$\bar S^i$) is skew-symmetric:
$$\bar Q_0^{\trsp}\big[\nabla \partial_iV_n, ~ \partial\vec p_n\big] +
\big[\nabla V_n,~\vec p_n\big]^{\trsp}(\partial_iQ_0)\in so(3).$$
Equating tangential entries and observing (\ref{ml4}), yields for every $i,j,k=1,2$:
$$\langle \partial_jy_0, \partial_{ik}V_n\rangle +
\langle \partial_ky_0, \partial_{ij}V_n\rangle = -\Big(\langle \partial_jV_n, \partial_{ik}y_0\rangle + 
\langle \partial_kV_n, \partial_{ij}y_0\rangle\Big)\to 0 \quad\mbox{strongly in }
L^{2}(\omega).$$
Permuting $i,j,k$ we eventually get:
$$\langle \partial_jy_0, \partial_{ik}V_n\rangle \to 0
\quad\mbox{strongly in } L^{2}(\omega) \qquad\mbox{for all }\; i,j,k=1,2.$$
On the other hand, equating off-tangential entries, we get by (\ref{ml4}) and
(\ref{ml4.5}) that for each $i=1,2$:
$$\langle \vec b_0, \partial_{ij}V_n\rangle = -\big((\nabla
y_0)^{\trsp}\nabla \vec p_n +(\nabla V_n)^{\trsp}\nabla \vec
b_0\big)_{ij} - \langle \vec p_n, \partial_{ij}y_0\rangle \to 0 \quad\mbox{strongly in }
L^{2}(\omega).$$
Consequently, $\{Q_0^{\trsp}\partial_{ij}V_n \to 0\}_{i,j=1,2}$ in $L^{2}(\omega,
\mathbb{R}^3)$, which implies convergence (\ref{ml4.55}) as claimed. This ends the proof of (\ref{8.1}).
\end{proof}

\begin{remark}\label{nocoer}
Although the kernel of the (nonlinear) energy $\mathcal{I}_4$,
displayed in Theorem \ref{kernel_vK}, is finite dimensional, 
the full coercivity estimate of the form below is {\em false}:
\begin{equation}\label{coer4}
\begin{split}
&\min_{S\in so(3), c\in \mathbb{R}^3} \Big(\|V - (Sy_0 + c)\|_{W^{2,2}(\omega, \mathbb{R}^3)}^2
+ \|\mathbb{S} - \frac{1}{2}\big((\nabla y_0)^{\trsp}\nabla \big(S^2y_0 -
\frac{1}{12}\tilde d_0\big)\big)_{\sym}\|_{L^2(\omega,
  \mathbb{R}^{2\times 2})}^2\Big) \\ & \qquad \qquad \quad \leq
C\mathcal{I}_4(V, \mathbb{S})\qquad \mbox{for all } \; (V,\mathbb{S})\in \mathscr
V\times \mathscr S.
\end{split}
\end{equation}
For a counterexample, consider the particular case of 
classical von K\'arm\'an functional (\ref{vkflat}), specified in Remark
\ref{flat4}. Clearly, $\mathcal{I}_4(v,w)=0$ if an only if $v(x) =
\langle a,x\rangle + \alpha$ and $w(x) =\beta x^\perp -
\frac{1}{2}\langle a,x\rangle a +\gamma$, for some $a\in\mathbb{R}^2$
and $\alpha, \beta, \gamma\in \mathbb{R}$. Note that (\ref{8.1})
reflects then the Poincar\'e inequality: $\int_\omega |\nabla
v-\fint_{\omega}\nabla v|^2 \,\rmd x'\leq C\int_\omega|\nabla^2v|^2
\,\rmd x'$, whereas (\ref{coer4}) takes the form:
\begin{equation}\label{coer5}
\min_{a\in \mathbb{R}^2} \Big(\int_\omega |\nabla v - a|^2 \,\rmd x' +\int_{\omega}
|\sym\nabla w + \frac{1}{2}a\otimes a|^2 \,\rmd x'\Big) \leq C\mathcal{I}_4(v, w).
\end{equation}
Let $\omega = B_1(0)$. Given $v\in W^{2,2}(\omega)$ such that $\det\nabla^2 v=0$, let
$w$ satisfy: $\sym\nabla w = -\frac{1}{2}\nabla v\otimes \nabla v$,
which results in vanishing of the first term in
(\ref{vkflat}). Neglecting the first term in the left hand side of
(\ref{coer5}), leads in this context to the following weaker form, which we below disprove:
\begin{equation}\label{coer6}
\min_{a\in \mathbb{R}^2} \int_\omega |\nabla v\otimes\nabla v -
a\otimes a|^2 \,\rmd x' \leq C \int_\omega |\nabla^2v|^2 \,\rmd x'.
\end{equation}

Define $v_n(x) = n (x_1+x_2) +\frac{1}{2}(x_1+x_2)^2$ for all $x=(x_1, x_2)\in\omega$.
Then $\nabla v_n = (n+x_1+ x_2) (1,1)$ and $\det\nabla^2v_n =0$.
Minimization in (\ref{coer6}) becomes:
$\min_{a\in\mathbb{R}^2}\int_\omega |(n+x_1+x_2)^2 (1,1)\otimes (1,1)-
a\otimes a|^2 \,\rmd x'$ and an easy explicit calculation yields the
necessary form of the minimizer: $a=\delta (1,1)$. Thus, the same minimization can
be equivalently written and estimated in:
$$4 \cdot \min_{\delta\in\mathbb{R}}  \int_\omega \big|(n+x_1+x_2)^2 -
\delta^2\big|^2 \,\rmd x'\sim 4n^2\to \infty \quad \mbox{ as }\; n\to\infty.$$
On the other hand, $|\nabla^2v_n|^2 = 4$ at each $x'\in\omega$. Therefore, the estimate
(\ref{coer6}) cannot hold.
\end{remark}

\section{Beyond the von K\'arm\'an regime: an example}\label{example_conformal}

Given a function $\phi\in\mathcal{C}^\infty((-\frac{1}{2}, \frac{1}{2}), \R)$, consider the conformal metric: 
$$G(x', x_3) = e^{2\phi(x_3)}Id_3\qquad \mbox{ for all }\, x=(x', x_3)\in\Omega^h.$$ 
The midplate metric $\bar{\mathcal{G}}_{2\times 2} = e^{2\phi(0)}Id_2$ has a smooth
isometric immersion $ y_0=e^{\phi(0)}id_2:\omega\to\R^2$ and thus  by
Theorem \ref{Glimsup} there must be: 
$$\inf \mathcal{E}^h\leq Ch^2.$$
By a computation, we get that the only possibly non-zero
Christoffel symbols of $G$ are:
$\Gamma_{11}^3=\Gamma_{22}^3=-\phi'(x_3)$ and $
\Gamma_{13}^1=\Gamma_{23}^2=\Gamma_{33}^3= \phi'(x_3)$,
while the only possibly nonzero Riemann curvatures are:
\begin{equation}\label{curva}
R_{1212}=-\phi'(x_3)^2e^{2\phi(x_3)}, \qquad
R_{1313}=R_{2323}=-\phi''(x_3)e^{2\phi(x_3)}.
\end{equation}
Consequently, the results of this paper provide the following hierarchy of
possible energy scalings:
\begin{itemize}
\item[(a)] $\{ch^2\leq \inf \mathcal{E}^h\leq Ch^2\}_{h\to 0}$ with
$c,C>0$. This scenario is equivalent to $\phi'(0)\neq 0$. The
functionals $\frac{1}{h^2}\mathcal{E}^h$
as in Theorems \ref{compactness_thm}, \ref{Gliminf} and \ref{Glimsup}
exhibit the indicated compactness properties and $\Gamma$-converge to the
following energy $\mathcal{I}_2$ defined on the
  set of deformations: $\{y\in W^{2,2}(\omega, \R^3); ~ (\nabla  y)^{\trsp}\nabla y = Id_2\}$:
$$\mathcal{I}_2(y) = \frac{1}{24}\int_\omega \mathcal{Q}_2\big(
\Pi_{y}-\phi'(0)Id_2\big) \,\rmd x'.$$
Here $\mathcal{Q}_2(F_{2\times 2}) = \min\big\{D^2W(Id_3)(\tilde
F, \tilde F); ~ \tilde F\in\R^{3\times 3} \mbox{ with } \tilde
F_{2\times 2} = F_{2\times 2}\big\}$.
\item[(b)] $\{ch^4\leq \inf \mathcal{E}^h\leq Ch^4\}_{h\to 0}$ with
  $c, C>0$. This scenario is equivalent to $\phi'(0)= 0$ and
  $\phi''(0)\neq 0$. The unique (up to rigid motions) minimizing
  isometric immersion is then $id_2:\omega\to\mathbb{R}^2$ and the
  functionals $\frac{1}{h^4}\mathcal{E}^h$ have the compactness and
  $\Gamma$-convergence properties as in Theorems \ref{thm: lower_bound_4},
  \ref{Gliminf4} and \ref{Glimsuph4}. The following limiting functional
  $\mathcal{I}_4$ is defined on the set of displacements $\{(v,w)\in
  W^{2,2}(\omega, \R)\times W^{1,2}(\omega,\R^2)\}$ as in Remark \ref{flat4}:
\begin{equation*}
\begin{split}
\mathcal{I}_4(v,w) = & ~ \frac{1}{2}\int_\omega \mathcal{Q}_2\big(
\sym\nabla w + \frac{1}{2}\nabla v\otimes\nabla v - \frac{1}{24}\phi''(0)Id_2\big) \,\rmd x'
\\ & ~ + \frac{1}{24}\int_\omega \mathcal{Q}_2\big( \nabla^2 v\big) \,\rmd x'
+ \frac{1}{1440} \phi''(0)^2|\omega| \mathcal{Q}_2\big(Id_2\big).
\end{split}
\end{equation*}
\item[(c)] $\{\inf \mathcal{E}^h\leq Ch^6\}_{h\to 0}$ with $C>0$. This
scenario is equivalent to $\phi'(0)= 0$ and $\phi''(0)= 0$
and in fact we have the following more precise result below. 
\end{itemize} 

\begin{thm}\label{conformal_allscalings}
Let $G(x', x_3) = e^{2\phi(x_3)}Id_3$, where $\phi^{(k)}(0)=0$ for
$k=1\ldots n-1$ up to some $n>2$. Then: $\inf \mathcal{E}^h\leq
Ch^{2n}$ and:
\begin{equation}\label{6limit}
\lim_{h\to 0}\frac{1}{h^{2n}}\inf\mathcal{E}^h \geq c_n\, \phi^{(n)}(0)^2 |\omega| \mathcal{Q}_2(Id_2),
\end{equation}
where $c_n>0$. In particular, if $\phi^{(n)}(0)\neq 0$ then we have: $ch^{2n}\leq
\inf \mathcal{E}^h\leq Ch^{2n}$ with $c, C>0$.
\end{thm}
\begin{proof}
{\bf 1.} For the upper bound, we compute:
\begin{equation*}
\begin{split}
\mathcal{E}^h\big(e^{\phi(0)}id_3\big) & =\frac{1}{h}\int_{\Omega^h}W\big(e^{\phi(0)-\phi(x_3)}Id_3\big) \,\rmd x
= \frac{1}{2h}\int_{\Omega^h}\mathcal{Q}_3\big(\phi^{(n)}(0)\frac{x_3^n}{n!}Id_3\big) 
+ \mathcal{O}(h^{2n+2})\,\rmd x \\ & = h^{2n} \bigg(
\frac{\phi^{(n)}(0)^2}{(n!)^2}\frac{1}{(2n+1) 2^{2n+1}} |\omega|
\mathcal{Q}_3(Id_3) + o(1)\bigg)\leq Ch^{2n},
\end{split}
\end{equation*}
where we used the fact that $e^{\phi(0)-\phi(x_3)} = 1-\phi^{(n)}(0)\frac{x_3^n}{n!}
+\mathcal{O}(|x_3|^{n+1})$.

\smallskip

{\bf 2.} To prove the lower bound (\ref{6limit}), let $\{u^h\in
W^{1,2}(\Omega^h, \R^3)\}_{h\to 0}$ satisfy $\mathcal{E}^h(u^h)\leq Ch^{2n}$. Then: 
\begin{equation*}
\begin{split}
\mathcal{E}^h(u^h) & \geq \frac{c}{h}\int_{\Omega^h}\dist^2\big(\nabla
u^h, e^{\phi(x_3)}SO(3)\big) \,\rmd x \\ & \geq \frac{c}{h}\int_{\Omega^h}\dist^2\big(\nabla
u^h, e^{\phi(0)}SO(3)\big) \,\rmd x  -
\frac{\bar c}{h}\int_{\Omega^h}\big|\phi^{(n)}(0)\frac{x_3^n}{n!} + \mathcal{O}(h^{n+1})\big|^2 \,\rmd x,
\end{split}
\end{equation*}
which results in: $\frac{1}{h}\int_{\Omega^h} \dist^2\big(e^{-\phi(0)}\nabla
u^h, SO(3)\big) \,\rmd x\leq Ch^{2n}$.
Similarly as in Lemma \ref{approx} and Corollary \ref{cor_approx}, it
follows that there exist approximating rotation fields $\{R^h\in
W^{1,2}(\omega, SO(3))\}_{h\to 0}$ such that:
\begin{equation}\label{approx_2n}
\frac{1}{h}\int_{\Omega^h}|\nabla u^h - e^{\phi(0)}R^h|^2 \,\rmd x\leq
Ch^{2n}, \qquad \int_{\omega}|\nabla R^h|^2 \,\rmd x\leq Ch^{2n-2}.
\end{equation}

As in sections \ref{2} and \ref{4}, we define the following displacement and deformation fields: 
\begin{equation*}
\begin{split}
& y^h(x', x_3) = (\bar R^h)^{\trsp}\big(u^h(x', hx_3) -
\fint_{\Omega^h}u^h\big)\in W^{1,2}(\Omega, \R^3), ~~
\mbox{where}~~\bar R^h = \mathbb{P}_{SO(3)}\fint_{\Omega^h}e^{-\phi(0)}\nabla u^h(x) \,\rmd x, \\
& V^h(x')= \frac{1}{h^{n-1}}\int_{-1/2}^{1/2}y^h(x', x_3) -
e^{\phi(0)}\big(id_2 + hx_3 e_3\big) \,\rmd x_3\in W^{1,2}(\omega, \R^3).
\end{split}
\end{equation*} 
In view of (\ref{approx_2n}), we obtain then the following
convergences (up to a not relabelled subsequence):
\begin{equation*}
\begin{split}
& y^h\to e^{\phi(0)}id_2 \quad \mbox{ in } \, W^{1,2}(\omega, \R^3), \qquad
\frac{1}{h}\partial_3y^h\to e^{\phi(0)}e_3 \mbox{ in } \, L^{2}(\omega, \R^3),\\
& V^h \to V\in W^{2,2}(\omega, \R^3) \quad \mbox{ in } \, W^{1,2}(\omega, \R^3), \quad
\frac{1}{h}(\nabla V^h)_{2\times 2, \sym}\rightharpoonup \sym\nabla
w\quad  \mbox{ weakly in } \, L^{2}(\omega, \R^{2\times 2}).
\end{split}
\end{equation*} 
This allows to conclude the claimed lower bound:
\begin{equation*}
\begin{split}
\liminf_{h\to 0}\frac{1}{h^{2n}}\mathcal{E}^h(u^h) & \geq \frac{1}{2}
\big\|e^{-\phi(0)}\sym\nabla w - x_3 e^{-\phi(0)}\nabla^2V^3 -
\phi^{(n)}(0)\frac{x_3^n}{n!}Id_2\big\|^2_{\mathcal{Q}_2} \\ & \geq 
\frac{1}{2} \Big\| \phi^{(n)}(0)\frac{x_3^n}{n!}Id_2 -
\mathbb{P}_1\Big(\phi^{(n)}(0)\frac{x_3^n}{n!}Id_2\Big)\Big\|^2_{\mathcal{Q}_2}  
\\ & = \frac{1}{2}\frac{\phi^{(n)}(0)^2}{(n!)^2}\cdot
\int_{-1/2}^{1/2}\big(x_3^n - \mathbb{P}_1(x_3^n)\big)^2 \,\rmd x_3\cdot
|\omega|\mathcal{Q}_2(Id_2)\\ &
= c_n \cdot \phi^{(n)}(0)^2 |\omega| \mathcal{Q}_2(Id_2),
\end{split}
\end{equation*} 
as in (\ref{6limit}), with the following constant $c_n$:
$$c_n =
\frac{1}{2^{2n+1}(n!)^2}\left\{\begin{array}{ll}\frac{(n-1)^2}{(2n+1)(n+2)^2}
    & \mbox{ for $n$ odd}\\ \frac{n^2}{(2n+1)(n+1)^2} & \mbox{ for $n$ even}.
\end{array}\right. $$
Observe that $c_2=\frac{1}{1440}$, consistently with the previous
direct application of Theorem \ref{Gliminf4}.
\end{proof}




\end{document}